\documentclass[a4paper,12pt]{article}
\usepackage[colorlinks=true,allcolors=blue]{hyperref}
\usepackage{amssymb,amsmath,amsthm}
\usepackage[text={.9\paperwidth,.68\paperheight},ratio=1:1]{geometry}
\usepackage{kpfonts}
\usepackage{enumitem}

\usepackage{fancyhdr}
\fancypagestyle{plain}{%
\fancyhf{} 
\fancyfoot[C]{\sffamily\fontsize{9pt}{9pt}\selectfont\thepage} 

}
\usepackage[para,bottom]{footmisc}

\numberwithin{equation}{section}

\usepackage{authblk}
\usepackage{float,graphicx}
\usepackage[labelfont=bf,textfont={it}]{caption}
\usepackage{subcaption}
\usepackage{bbm}
\usepackage{tikz}
\usetikzlibrary{cd,arrows,decorations.markings,positioning,shapes,calc,patterns,tqft,intersections,shapes.misc}
\tikzset{>=latex}

\newcommand{\ie}{\textit{i.e. }}
\newcommand{\eg}{\textit{e.g. }}
\newcommand{\git}{/\!\!/}
\newcommand{\de}{{\text d}}

\newcommand{\iu}{{\text i}}
\newcommand{\eu}{{\text e}}
\newcommand{\ddfrac}[2]{\frac{\displaystyle{#1}}{\displaystyle{#2}}}

\DeclareMathOperator{\Ad}{Ad}
\DeclareMathOperator{\vd}{vd}

\DeclareMathOperator{\Tr}{Tr}

\DeclareMathOperator{\rk}{rk}
\DeclareMathOperator{\ch}{ch}
\DeclareMathOperator{\Hilb}{Hilb}
\DeclareMathOperator{\td}{td}

\DeclareMathOperator{\Ell}{Ell}
\DeclareMathOperator{\Ellform}{\mathcal E\ell\ell}

\DeclareMathOperator{\Hom}{Hom}
\DeclareMathOperator{\End}{End}
\DeclareMathOperator{\im}{Im}

\DeclareMathOperator{\ext}{Ext}
\DeclareMathOperator{\spec}{Spec}
\DeclareMathOperator{\Proj}{Proj}

\DeclareMathOperator{\Sch}{\mathsf Sch}
\DeclareMathOperator{\Sets}{\mathsf Sets}
\DeclareMathOperator{\quot}{Quot}
\DeclareMathOperator{\Quot}{\mathsf Quot}

\newtheorem{theorem}{Theorem}[section]
\newtheorem{proposition}[theorem]{Proposition}
\newtheorem{lemma}[theorem]{Lemma}
\newtheorem{conjecture}{Conjecture}
\newtheorem{definition}{Definition}
\newtheorem{remark}{Remark}[section]
\newtheorem{corollary}[theorem]{Corollary}

\newtheorem*{theorem*}{Theorem}
\newtheorem*{proposition*}{Proposition}
\newtheorem*{lemma*}{Lemma}
\newtheorem*{conjecture*}{Conjecture}
\newtheorem*{definition*}{Definition}
\newtheorem*{remark*}{Remark}
\newtheorem*{corollary*}{Corollary}

\title{\bf Flags of sheaves, quivers and symmetric polynomials}
\author[ ]{Giulio \textsc{Bonelli}\thanks{bonelli@sissa.it}\hspace{3em}Nadir \textsc{Fasola}\thanks{nfasola@sissa.it}\hspace{3em}Alessandro \textsc{Tanzini}\thanks{tanzini@sissa.it}}
\date{}
\begin{document}
\maketitle
\abstract{
We study the representation theory of the nested instantons quiver presented in \cite{Bonelli:2019lal}, which describes a particular class of surface defects in four-dimensional supersymmetric gauge theories. We show that the moduli space of its stable representations provides an ADHM-like construction for nested Hilbert schemes of points on $\mathbb C^2$, for rank one, and for the moduli space of flags of framed torsion-free sheaves on $\mathbb P^2$, for higher rank. We introduce a natural torus action on this moduli space and use equivariant localization to compute some of its (virtual) topological invariants, including the case of compact toric surfaces.
We conjecture that the generating function of holomorphic Euler characteristics 
for rank one is given in terms of polynomials in the equivariant weights, which, for specific numerical types, coincide with (modified) Macdonald polynomials. 
}

{
  \hypersetup{linkcolor=black}
  \tableofcontents
}

\section*{Introduction}\addcontentsline{toc}{section}{Introduction}

In \cite{Bonelli:2019lal} we introduced the moduli space of {\it nested instantons} as the moduli space of stable representations of a suitable quiver. This arises in the study of surface defects in 
supersymmetric gauge theory on $T^2\times {\mathcal C}_{g,k}$, where
$T^2$ is a real two torus and ${\mathcal C}_{g,k}$
a genus $g$ complex projective curve with $k$ marked points.
Let us briefly describe some string theory motivations before presenting the content of the paper.

\paragraph{String theory motivations:}
The D-brane set-up engineering the surface defect is described in \cite{Bonelli:2019lal}, and its analysis naturally led to a description in terms of representations of a quiver in the category of vector spaces, the D-branes being the objects and the open strings being the morphisms. Let us briefly resume the D-brane geometry and its relation with the relevant mathematical problems. One considers type IIB supersymmetric background given by 
$T^2\times T^*{\mathcal C}_{g,k}\times {\mathbb C}^2$, with 
$r$ D7-branes located at points of the fiber of the cotangent bundle
and 
$n$ D3-branes along $T^2\times {\mathcal C}_{g,k}$. The low energy effective theory of the D7-branes is {\it equivariant higher rank Donaldson-Thomas theory} \cite{Donaldson:1996kp} on the four-fold 
$T^2\times {\mathcal C}_{g,k}\times {\mathbb C}^2$, 
while the low energy effective theory of the D3-branes is {\it equivariant Vafa-Witten theory} on $T^2\times {\mathcal C}_{g,k}$, \cite{Vafa:1994tf}. In the chamber of small volume of ${\mathcal C}_{g,k}$, the effective theory 
describing the surface defect is encoded in the theory of maps from $T^2$
to the moduli space of stable representations of
the {\it comet shaped} quiver 
displayed in figure \ref{fig:quiver_d3global}. For $k=1$, this is described by the total space of a bundle ${\mathcal V}_g$
over the {\it nested instanton} moduli space, which in turn is the moduli space of stable representations of the quiver displayed in figure \ref{fig:quiver_local}. Let us remark that virtual invariants of $\mathcal V_g$ have a connection to the cohomology of character varieties of punctured Riemann surfaces, and in particular to the conjecture proposed in \cite{hausel2011} whose physical interpretation was provided in \cite{Chuang:2013wpa}. The interested reader can find the details in \cite{Bonelli:2019lal}.

\begin{figure}[htb!]
\centering
\vspace{-5mm}
\begin{tikzpicture}
\node[](F) at (2.5+10,-2){$\bullet$};
\node[](Gk0) at (2.5+10,0){$\bullet$};
\node[](Gk1) at (5+10,2){$\bullet$};
\node[](F1) at (6.5+10,1){$\bullet$};
\node[](Gk12) at (5+10,-2){$\bullet$};
\node[](F12) at (6.5+10,-3){$\bullet$};
\node[](label) at (7+10,2){$\cdots$};
\node[](label2) at (7+10,-2){$\cdots$};
\node[](dots1) at (11,0){$\cdots$};
\node[](2g) at (11,-.2){{\scriptsize $2g+2$}};
\node[](vdots1) at (5+10,0){$\vdots$};
\node[](vdots2) at (9+10,0){$\vdots$};
\node[](Gkn) at (9+10,2){$\bullet$};
\node[](Gkn2) at (9+10,-2){$\bullet$};
\node[](Fn) at (10.5+10,1){$\bullet$};
\node[](Fn2) at (10.5+10,-3){$\bullet$};
\draw[->](F.15+90) to[bend left,looseness=.6] (Gk0.165+90);
\draw[->](Gk0.195+90) to[bend left,looseness=.6] (F.345+90);
\draw[->](F1.125) to[bend right,looseness=.6] (Gk1.-20);
\draw[->](Gk1.-40) to[bend right,looseness=.6] (F1.145);
\draw[->](F12.125) to[bend right,looseness=.6] (Gk12.-20);
\draw[->](Gk12.-40) to[bend right,looseness=.6] (F12.145);
\draw[->](Gk1) to (Gk0);
\draw[->](Gk12) to (Gk0);
\draw[->](label) to (Gk1);
\draw[->](label2) to (Gk12);
\draw[->](Gkn) to (label);
\draw[->](Gkn2) to (label2);
\draw[->](Fn.125) to[bend right,looseness=.6] (Gkn.-20);
\draw[->](Gkn.-40) to[bend right,looseness=.6] (Fn.145);
\draw[->](Fn2.125) to[bend right,looseness=.6] (Gkn2.-20);
\draw[->](Gkn2.-40) to[bend right,looseness=.6] (Fn2.145);
\draw[-](Gk0) to[out=60+90,in=120+90,looseness=25] (Gk0);
\draw[-](Gk0) to[out=70+90,in=110+90,looseness=15] (Gk0);
\draw[-](Gk1) to[out=65+180,in=115+180,looseness=20] (Gk1);
\draw[-](Gk1) to[out=65,in=115,looseness=20] (Gk1);
\draw[-](Gkn) to[out=65+180,in=115+180,looseness=20] (Gkn);
\draw[-](Gkn) to[out=65,in=115,looseness=20] (Gkn);
\draw[-](Gk12) to[out=65+180,in=115+180,looseness=20] (Gk12);
\draw[-](Gk12) to[out=65,in=115,looseness=20] (Gk12);
\draw[-](Gkn2) to[out=65+180,in=115+180,looseness=20] (Gkn2);
\draw[-](Gkn2) to[out=65,in=115,looseness=20] (Gkn2);
\end{tikzpicture}\caption{The comet-shaped quiver.}\label{fig:quiver_d3global}
\end{figure}
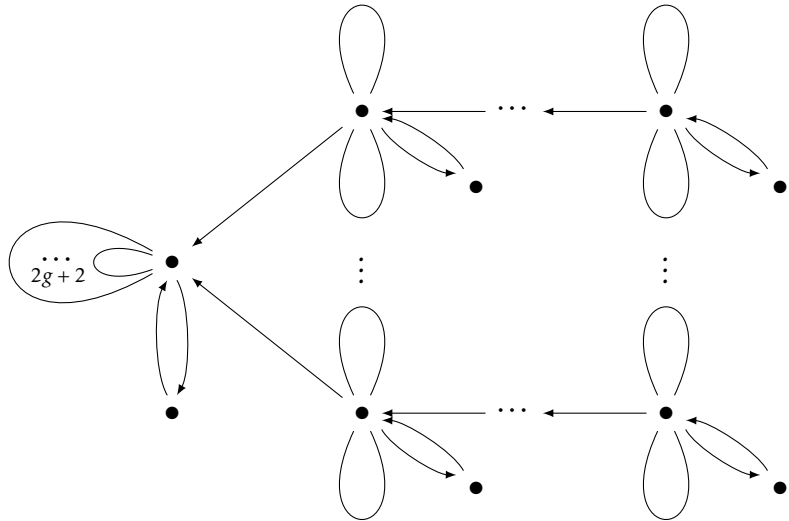

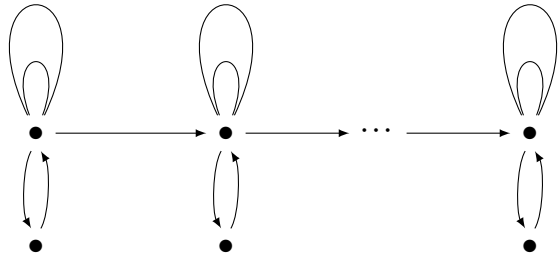
\begin{figure}[H]
\centering
\vspace{-5mm}
\begin{tikzpicture}
\node[](F0) at (2.5+10,-1.5){$\bullet$};
\node[](F1) at (5+10,-1.5){$\bullet$};
\node[](Fn) at (9+10,-1.5){$\bullet$};
\node[](Gk0) at (2.5+10,0){$\bullet$};
\node[](Gk1) at (5+10,0){$\bullet$};
\node[](label) at (7+10,0){$\cdots$};
\node[](Gkn) at (9+10,0){$\bullet$};
\draw[<-](F0.15+90) to[bend left,looseness=.6] (Gk0.165+90);
\draw[<-](Gk0.195+90) to[bend left,looseness=.6] (F0.345+90);
\draw[<-](F1.15+90) to[bend left,looseness=.6] (Gk1.165+90);
\draw[<-](Gk1.195+90) to[bend left,looseness=.6] (F1.345+90);
\draw[<-](Fn.15+90) to[bend left,looseness=.6] (Gkn.165+90);
\draw[<-](Gkn.195+90) to[bend left,looseness=.6] (Fn.345+90);
\draw[<-](Gk1) to (Gk0);
\draw[<-](label) to (Gk1);
\draw[<-](Gkn) to (label);
\draw[-](Gk0) to[out=65,in=115,looseness=25] (Gk0);
\draw[-](Gk0) to[out=70,in=110,looseness=15] (Gk0);
\draw[-](Gk1) to[out=65,in=115,looseness=25] (Gk1);
\draw[-](Gk1) to[out=70,in=110,looseness=15] (Gk1);
\draw[-](Gkn) to[out=65,in=115,looseness=25] (Gkn);
\draw[-](Gkn) to[out=70,in=110,looseness=15] (Gkn);
\end{tikzpicture}\vspace{2mm}\caption{The nested instantons quiver.}\label{fig:quiver_local}
\end{figure}
 
 \paragraph{Content of the paper:}
In this paper we concentrate on the study of representations of the {\it nested instantons} quiver with a {\it single} framing, namely we choose the dimension vector for the framing to be $\bold r=(r,0,\dots,0)$ where $r$ is the dimension of the rightmost framing node. We also study its relation to flags of framed torsion-free sheaves on ${\mathbb P}^2$ and {\it nested Hilbert schemes}, and compute some relevant virtual invariants via equivariant localisation.

We want to point out that the moduli space we are studying seems to be analogous to the Filt-scheme studied in \cite{mochizuki_filt} in the case of smooth projective curves. The importance of studying these moduli spaces on (smooth projective) surfaces lies in their application to the computation of monopole contributions to Vafa-Witten invariants defined in \cite{Tanaka:2017jom,Tanaka:2017bcw}. In fact these monopole contributions to Vafa-Witten invariants are expressed in terms of invariants of flags of sheaves, which in some cases reduce to nested Hilbert schemes, see \cite{Gottsche:2018meg,laarakker} for computations in this case. The deformation-obstruction theory and virtual cycle for the components of the monopole branch in Vafa-Witten theory giving rise to flags of higher rank sheaves were explicitly constructed in \cite{Sheshmani:2019tvp}. Nested Hilbert schemes on surfaces were interpreted in terms of degeneracy loci in \cite{Gholampour:2017dcp,Gholampour:2019wld}, where they are also shown to be equipped with a perfect obstruction theory. Similarly nested Hilbert schemes of points were also studied in \cite{2017arXiv170108899G}, and a perfect obstruction theory and virtual cycles are explicitly constructed. Their application to reduced DT and PT invariants are also discussed in \cite{2017arXiv170108899G,GSY2,diaconescu2018atiyah}.

In the following we give a summary of the result we obtained in this paper. In section \ref{sec1} we start our analysis by proving the following
\begin{theorem*}
The moduli space $\mathcal N(r,\bold n)$ of stable representation of the nested instantons quiver of numerical type $(r,\bold n)$ is a virtually smooth quasi-projective variety over $\mathbb C$ equipped with a natural action of $\mathbb T=T\times (\mathbb C^*)^r$, $T=(\mathbb C^*)^2$, and a perfect obstruction theory.
\end{theorem*}

We also prove that $\mathcal N(r,\bold n)$ embeds into a smooth projective variety $\mathcal M(r,\bold n)$, see Section 1.3.

In section \ref{sec2}, we construct the moduli space $\mathcal F(r,\boldsymbol\gamma)$ of flags of framed torsion
free sheaves on ${\mathbb P}^2$ and prove the existence of an isomorphism with 
$\mathcal N(r,\bold n)$. 
As a particular case, we have 
\begin{theorem*}
The moduli space of nested instantons $\mathcal N(1,\bold n)$ is isomorphic to the nested Hilbert scheme of points on $\mathbb C^2$, namely
\begin{equation}
\mathcal N(1,\bold n)=\mathbb X_0\git_\chi\mathcal G\simeq\Hilb^{\hat{\bold n}}(\mathbb C^2).
\end{equation}
\end{theorem*}

The moduli space of flags of sheaves is constructed by means of a functor
\[
\mathsf F_{(r,\boldsymbol\gamma)}:\Sch_{\mathbb C}^{\rm op}\to\Sets
\]
parametrizing flags of torsion-free sheaves on $\mathbb P^2$ in
\begin{proposition*}
The moduli functor $\mathsf F_{(r,\boldsymbol\gamma)}$ is representable. The (quasi-projective) variety representing $\mathsf F_{(r,\boldsymbol\gamma)}$ is the moduli space of flags of framed (coherent) torsion-free sheaves on $\mathbb P^2$, denoted by $\mathcal F(r,\boldsymbol\gamma)$.
\end{proposition*}
while its isomorphism with $\mathcal N(r,\bold n)$
is proven in 
\begin{theorem*}
The moduli space of stable representations of the nested instantons quiver is a fine moduli space isomorphic to the moduli space of flags of framed torsion-free sheaves on $\mathbb P^2$: $\mathcal F(r,\boldsymbol\gamma)\simeq\mathcal N(r,\bold n)$, as schemes, where $n_i=\gamma_i+\cdots+\gamma_N$.
\end{theorem*}
The ADHM construction of a particular class of flags of sheaves on $\mathbb P^2$ was given in \cite{negut}, where their connection to shuffle algebras on $K-$theory is also studied. Moreover the construction of the functor $\mathsf F_{(r,\boldsymbol\gamma)}$ shows that the moduli space of nested instantons is isomorphic to a relative $\Quot-$scheme. Perfect obstruction theories on $\Quot-$schemes and the description of their local model in terms of a quiver is discussed in \cite{ricolfi1,ricolfi2}.

In section \ref{sec3} we proceed to the evaluation of the relevant virtual invariants via equivariant localisation. The classification of the $T$-fixed locus of $\mathcal N(r,\bold n)$
is presented in the
\begin{proposition*}
The $T-$fixed locus of $\mathcal N(r,n_0,\dots,n_{s-1})$ can be described by $s-$tuples of nested coloured partitions $\boldsymbol\mu_1\subseteq\cdots\subseteq\boldsymbol\mu_{s-1}\subseteq\boldsymbol\mu_0$, with $|\boldsymbol\mu_0|=n_0$ and $|\boldsymbol\mu_{i>0}|=n_0-n_i$.
\end{proposition*}

In \ref{sec:euler} we compute the generating function of the virtual Euler characteristics of 
$\mathcal N(1,\bold n)$, see eq.\eqref{uno} for the explicit combinatorial formula.
We conjecture that, by summing over the nested partitions, this generating function is expressed in terms of polynomials:

\begin{conjecture*}
The generating function
$\chi^{\rm vir}(\mathcal N(1,n_0,\dots,n_N);\mathfrak{q}_1^{-1},\mathfrak q_2^{-1})=\sum_{\mu_0}P_{\mu_0}(q,t)/N_{\mu_0}(q,t)$ is such that
\[
P_{\mu_0}(q,t)=\frac{Q_{\mu_0}(q,t)}{(1-qt)^N},
\]
with $Q_{\mu_0}(q,t)\in\mathbb Z[q,t]$.
\end{conjecture*}

For specific profiles of the nesting,
these polynomials are conjectured to compute sums of $(q,t)-$Kostka polynomials: 
\begin{conjecture*}
When $|\mu_0|=|\mu_N|+1=|\mu_{N-1}|+2=\cdots=|\mu_1|+N$ we have
\[
\begin{split}
        Q_{\mu_0}(q,t)&=\left\langle h_{\mu_0}(\bold x),\widetilde{H}_{\mu_0}(\bold x;q,t)\right\rangle\\
        &=\left\langle h_{\mu_0}(\bold x),\sum_{\lambda,\nu\in\mathcal P(n_0)}\widetilde{K}_{\lambda,\mu_0}(q,t)K_{\mu_0,\nu}m_\nu(\bold x)\right\rangle\\
        &=\sum_{\substack{\lambda\in\mathcal P(n_0)\\ m_\lambda(\bold x)\neq 0}}\widetilde{K}_{\lambda,\mu_0}(q,t),
\end{split}
\]
where the Hall pairing $\langle-,-\rangle$ is such that $\langle h_\mu,m_\lambda\rangle=\delta_{\mu,\lambda}$ and $\widetilde H_{\mu}(\bold x;q,t)$, $\widetilde K_{\lambda,\mu}(q,t)$ are the modified Macdonald polynomials and the modified Kostka polynomials, respectively.
\end{conjecture*}

In \ref{sec:hirzebruch} we compute the generating function of the virtual $\chi_{-y}$-genus
of $\mathcal N(1,\bold n)$, see eq.\eqref{due}, and of 
$\mathcal N(r,\bold n)$ , see eq.\eqref{chi_y_r}.

We also show that, by specialising at $y=1$, one gets that the generating function of nested partitions of arbitrary 
length is the Macmahon function as expected, see eq. \eqref{tre}.

In \ref{sec:elliptic} we compute the generating function of the virtual elliptic genus
of $\mathcal N(1,\bold n)$, see eq.\eqref{quattro}, and of 
$\mathcal N(r,\bold n)$ , see eq.\eqref{ell_vir_r}.

Finally, in section \ref{sec4}, we extend our results to  
${\mathbb P}^2$ and ${\mathbb P}^1\times{\mathbb P}^1$ in the case of $\chi_{-y}-$genera, see formulae
\eqref{cinque} and \eqref{sei} respectively. Notice that the choice of computing $\chi_{-y}-$genera was due to the expected simple polynomial dependence in $y$. Everything which was done in this context is however completely general and holds for any complex genus.

\vspace{7mm}

{\bf Acknowledgements:}
We thank U. Bruzzo, E. Diaconescu, L. G\"ottsche, T. Hausel, M. Kool, A. Mellit, A. T. Ricolfi, F. Rodriguez-Villegas, A. Sheshmani, Y. Tanaka for useful discussion. The work of G.B. is partially supported by INFN - ST\&FI and by the PRIN project "Non-perturbative Aspects Of Gauge Theories And Strings".
The work of N.F. and A.T. is partially supported by INFN - GAST and by the PRIN project "Geometria delle variet\`a algebriche".

\section{The nested instantons quiver}\label{sec1}

\subsection{Quiver representations and stability}
In the following we will mainly be interested in studying the following quiver, which will be called the \textit{nested instantons quiver}
\begin{figure}[H]
\centering
\begin{equation}\label{quiv_tail}
\begin{tikzcd}
V_{N} \arrow[out=70,in=110,loop,swap,"\alpha_{N}"] \arrow[out=250,in=290,loop,swap,"\beta_{N}"] \arrow[r,shift left=.5ex,"\phi_{N}"] & \cdots \arrow[r,shift left=.5ex,"\phi_2"] \arrow[l,shift left=.5ex,"\gamma_{N}"] & V_1 \arrow[out=70,in=110,loop,swap,"\alpha_1"] \arrow[out=250,in=290,loop,swap,"\beta_1"] \arrow[r,shift left=.5ex,"\phi_1"] \arrow[l,shift left=.5ex,"\gamma_2"] & V_0 \arrow[l,shift left=.5ex,"\gamma_1"]  \arrow[out=70,in=110,loop,swap,"\alpha_0"] \arrow[out=250,in=290,loop,swap,"\beta_0"] \arrow[r,shift left=.5ex,"\eta"] & W \arrow[l,shift left=.5ex,"\xi"]
\end{tikzcd}
\end{equation}
\end{figure}
with relations
\begin{displaymath}
\begin{split}
&[\alpha_0,\beta_0]+\xi\eta=0,\quad [\alpha_i,\beta_i]=0,\quad \alpha_i\phi_i-\phi_i\alpha_{i+1}=0=\beta_i\phi_i-\phi_i\beta_{i+1}\\
&\gamma_i\alpha_i-\alpha_{i+1}\gamma_i=0=\gamma_i\beta_i-\beta_{i+1}\gamma_i,\quad \phi_i\gamma_i=0,\quad \eta\phi_1=0,\quad \gamma_1\xi=0
\end{split}
\end{displaymath}

Given 
\begin{displaymath}
\begin{split}
\mathbb X=&\End V_0^{\oplus 2}\oplus\Hom(V_0,W)\oplus\Hom(W,V_0)\oplus\End(V_1)^{\oplus 2}\oplus\Hom(V_1,V_0)\\
&\oplus\Hom(V_0,V_1)\oplus\cdots\oplus\End(V_{N})^{\oplus 2}\oplus\Hom(V_{N},V_{N-1})\oplus\Hom(V_{N-1},V_{N})
\end{split}
\end{displaymath}
a representation of numerical type $(r,\bold n)$ of \eqref{quiv_tail} in the category of vector spaces will be given by the datum of $X=\mathbb W\oplus h$, with $\mathbb W=(W,V_0,\dots,V_{N})$, with $\dim W=r$ and $\dim V_i=n_i$, and $$\mathbb X\ni h=(B_1^0,B_2^0,I,J,B_1^1,B_2^1,F^1,G^1,\dots),$$ satisfying
\begin{equation}
\begin{split}
&[B_1^0,B_2^0]+IJ=0,\quad [B_1^i,B_2^i]=0,\quad B_1^iF^i-F^iB_1^{i+1}=0=B_2^iF^i-F^iB_2^{i+1}\\
&G^iB_1^i-B_2^{i+1}G^i=0=G^iB_2^i-B_2^{i+1}G^i,\quad F^iG^i=0,\quad JF^1=0,\quad G^1I=0
\end{split}
\end{equation}
which we will call \textit{nested ADHM equations}. In the following we need to address the problem of King stability for representations of the nested instantons quiver.
\begin{definition}
Let $\Theta=(\boldsymbol\theta,\theta_\infty)\in\mathbb Q^{s+1}$ be such that $\Theta(X)=\bold n\cdot\boldsymbol\theta+r\theta_{\infty}=0$. We will say that a framed representation $X$ of \eqref{quiv_tail} is $\Theta-$semistable if
\begin{itemize}
\item $\forall 0\neq\tilde X\subset X$ of numerical type $(0,\tilde{\bold n})$ we have $\Theta(\tilde X)=\boldsymbol\theta\cdot\tilde{\bold n}\le 0$;
\item $\forall 0\neq\tilde X\subset X$ of numerical type $(\tilde r,\tilde{\bold n})$ we have $\Theta(\tilde X)=\boldsymbol\theta\cdot\tilde{\bold n}+\tilde r\theta_\infty\le 0$.
\end{itemize}
If strict inequalities hold $X$ is said to be $\Theta-$stable.
\end{definition}
In \cite{bruzzo2011,flach_jardim} the two node case, namely $N=1$ was considered and we can here generalize their result to the more general nested instantons quiver \eqref{quiv_tail}.
\begin{proposition}\label{theta_stability}
Let $X$ be a representation of \eqref{quiv_tail} of numerical type $(r,\bold n)\in\mathbb N^{N+2}_{>0}$, then choose $\theta_i>0$, $\forall i>0$ and $\theta_0$ s.t. $\theta_0+n_1\theta_1+\cdots n_{s-1}\theta_{s-1}<0$. The following are equivalent:
\begin{enumerate}[label=(\roman*)]
\item $X$ is $\Theta-$stable;
\item $X$ is $\Theta-$semistable;
\item $X$ satisfies the following conditions:
\begin{description}
\item[S1] $F^i\in\Hom(V_{i+1},V_i)$ is injective, $\forall i\ge 1$;
\item[S2] the ADHM datum $\mathcal A=(W,V_0,B_1^0,B_2^0,I,J)$ is stable.
\end{description}
\end{enumerate}
\end{proposition}
\begin{proof}
$\boldsymbol{(i)\Rightarrow (ii)}$ This is obvious, as a $\Theta-$stable representation is also $\Theta-$semistable.

$\boldsymbol{(ii)\Rightarrow (iii)}$ Let us first take a $\Theta-$semistable representation $X$ having at least one of the $F^i$ not injective. Without loss of generality let $F^k$ be the only one to be such a map. Then, if $v_k\in\ker F^k\Rightarrow B_2^{k+1}v_k\in\ker F^k$, due to the nested ADHM equations, and $B_2^{k+1}(\ker F^k)\subset\ker F^k$ (the same is obviously true for $B_1^{k+1}$. Now
$$\tilde X=(0,\dots,0,\ker F^k,0,\dots,F^k,B_1^{k+1}|_{\ker F^k},B_2^{k+1}|_{\ker F^k},0,\dots,0)$$
is a subrepresentation of $X$ of numerical type $(0,\dots,0,\dim\ker F^k,0,\dots,0)$. Thus 
$$\tilde{\bold n}\cdot\boldsymbol\theta+\tilde r\theta_\infty=\theta_k\dim\ker F^k>0,$$
which contradicts the hypothesis of $X$ being $\Theta-$semistable.

If instead we take $X$ to be $\Theta-$semistable and suppose {\bf S2} to be false, then $\exists 0\subset S\subset V_1$ s.t. $B_1^0(S)$, $B_2^0(S)$, $\im(I)\subseteq S$. In this case 
$$\tilde X=(W,S,V_1,\dots,B_1^0|_S,B_2^0|_S,I,J|_S,\dots)$$
is a subrepresentation of $X$ of numerical type $(r,\dim S,n_1,\dots)$ but, since $\bold n\cdot\boldsymbol\theta+r\theta_\infty=0$ having $\theta_{i>0}>0$ and $\theta_0-n_1\theta_1-\cdots<0$, we have 
$$\dim S\theta_0+n_1\theta_1+\cdots+r\theta_\infty=(\dim S-n_0)\theta_0>0,$$
which again leads to a contradiction.

$\boldsymbol{(iii)\Rightarrow(i)}$ If we take a proper subrepresentation $\tilde X$ of numerical type $(\tilde r,\tilde{\bold n})$, we just need to check the cases $\tilde r=0$ and $\tilde r=r$.
\begin{itemize}
\item If $\tilde r=r$ then $\tilde W=W$, which in turn implies that $I\neq 0$, otherwise the ADHM datum $(B_1^0,B_2^0,I,J)$ would not be stable. Since $\tilde X$ is proper the following diagram commutes
\begin{equation}
\begin{tikzcd}
W\arrow[r,"I"] & V_0\\
W\arrow[u,"\mathbbm 1_W"]\arrow[r,"\tilde I"] & \tilde V_0\arrow[hookrightarrow]{u}{i}
\end{tikzcd}\Rightarrow i\circ\tilde I=I\circ\mathbbm 1_W
\end{equation}
so that $\tilde n_0>0$, otherwise we would have $I=0$. Moreover the following diagram also commutes (and so does the analogous one for $B_2^0$)
\begin{equation}
\begin{tikzcd}
V_0\arrow[r,"B_1^0"] & V_0\\
\tilde V_0\arrow[hookrightarrow]{u}{i}\arrow[r,"\tilde B_1^0"] & \tilde V_0\arrow[hookrightarrow]{u}{i}
\end{tikzcd}\Rightarrow i\circ\tilde B_1^0=B_1^0\circ i\Rightarrow B_1^0(\tilde V_0)\subset\tilde V_0,
\end{equation}
leading to a contradiction with the stability of $(W,V_0,B_1^0,B_2^0,I,J)$. Since we are interested in proper subrepresentations of $X$, at least one $\tilde n_{i>0}$ is not zero, and at least one of these non-zero $\tilde n_k<n_k$, so that $\boldsymbol\theta\cdot\tilde{\bold n}+\theta_\infty r<0$, and $X$ is stable.
\item Let now $\tilde r=0$. Since we are interested in proper subrepresentations we must choose $\tilde n_0>0$, otherwise $\tilde V_{k>0}=0$ by virtue of the injectivity of $F_k$. In the same way as in the previous case the only option is $\tilde n_0=n_0$. Following the same steps we previously carried out $\boldsymbol\theta\cdot\tilde{\bold n}=\sum_{k>0}\theta_k(\tilde n_k-n_k)-\theta_\infty r<0$.
\end{itemize}
\end{proof}
\begin{corollary}
If $X$ is a stable representation of the nested instantons quiver, $G^k=0,\ \forall k$.
\end{corollary}
\begin{proof}
By the previous proposition, due to the injectivity of $F^k$, $F^kG^k=0\Rightarrow G^k=0$.
\end{proof}

\subsection{The nested instantons moduli space}\label{sec:moduli_space}
We want now to discuss the construction of the moduli space of stable representations of the quiver \eqref{quiv_tail}, and its connection to GIT theory and stability. First of all we define the space of the nested ADHM data to be the space $\mathbb X$ we defined previously, and an element $X\in\mathbb X$ is called an nested ADHM datum. On $\mathbb X$ we have a natural action of $\mathcal G=GL(V_0)\times\cdots\times GL(V_{N})$ defined by
\begin{equation}
\begin{split}
\Psi:(g_0,g_1,\dots,g_{N},X)\longmapsto(g_0B_1^0g_0^{-1},g_0B_2^0g_0^{-1},g_0I,Jg_0^{-1},\\
g_1B_1^1g_1^{-1},g_1B_2^1g_1^{-1},g_0F^1g_1^{-1},g_1G^1g_0^{-1},\\
\dots\\
g_{N}B_1^{N}g_{N}^{-1},g_{N}B_2^{N}g_{N}^{-1},g_{N-1}F^{N}g_{N}^{-1},g_{N}G^{N}g_{N-1}^{-1})
\end{split}
\end{equation}
This action of $\mathcal G$ on $\mathbb X$ is free on the stable points of $\mathbb X$. In fact if $\boldsymbol g\in\mathbb G$ is such that $\boldsymbol g\cdot X=X$, $\forall X\in\mathbb X$, we claim that $\boldsymbol g=(\mathbbm 1_{V_0},\dots,\mathbbm 1_{V_{N}})$. In order to see this, let $S=\ker(g_0-\mathbbm 1_{V_0})$. Since $\boldsymbol g\cdot X=X$ it follows that $g_0I=I$, which means $\im I\subset S$. Moreover $g_0B_1^0=B_1^0g_0$ and $g_0B_2^0=B_2^0g_0$, but if $v\in S\Rightarrow (g_0-\mathbbm 1_{V_0})v=0\Rightarrow g_0v=v$, thus implying that $B_1^0(S),B_2^0(S)\subset S$. The stability of $(W,V_0,B_1^0,B_2^0,I,J)$ then force $S=V_0$. Finally since, $g_0=\mathbbm 1_{V_0}$, $F^1(\mathbbm 1_{V_1}-g_1^{-1})=0\Rightarrow g_1=\mathbbm 1_{V_1}$ by the injectivity of $F^1$. By using this procedure then one can prove by iteration that $g_k=\mathbbm 1_{V_k}$, $\forall k$, thus $\boldsymbol g\cdot X=X\ \forall X\in\mathbb X\Leftrightarrow\boldsymbol g=\mathbbm 1$. This proves that the action of $\mathcal G$ is free on the stable points of $\mathbb X$, and it is easy to prove that it preserves $\mathbb X_0$, which denotes the space of nested ADHM data satisfying the relations of quiver \eqref{quiv_tail}.

Now if $\chi:\mathcal G\to\mathbb C^*$ is an algebraic character for the algebraic reductive group $\mathcal G$, we can produce the moduli space of $\chi-$semistable orbits following a construction due to \cite{king}, $\mathcal N^{ss}_\chi(r,\bold n)$, which is a quasi-projective scheme over $\mathbb C$ and is defined as
\begin{displaymath}
\mathcal N_\chi^{ss}(r,\bold n)=\mathbb X_0\git_\chi\mathcal G=\Proj\left(\bigoplus_{n\ge 0}A(\mathbb X_0(r,\bold n))^{\mathcal G,\chi^n}\right)
\end{displaymath}
with 
$$A(\mathbb X_0(r,\bold n))^{\mathcal G,\chi^n}=\{f\in A(\mathbb X_0(r,\bold n))|f(\boldsymbol h\cdot X)=\chi(\boldsymbol h)^nf(X),\forall\boldsymbol h\in\mathcal G\}.$$

The scheme $\mathcal N_\chi^{ss}(r,\bold n)$ contains an open subscheme $\mathcal N^{s}_\chi(r,\bold n)\subset\mathcal N^{ss}_\chi(r,\bold n)$ encoding $\chi-$stable orbits. It turns out that also in this framed case there is a relation between $\chi-$stability and $\Theta-$stability, as it was shown in \cite{king} in the non framed setting.
\begin{proposition}\label{prop:equivalence_stabilities}
Let $\Theta=(\theta_0,\theta_1,\dots,\theta_{N})\in\mathbb Z^{N+1}$ and define $\chi_\Theta:\mathcal G\to\mathbb C^*$ the character
\begin{equation}
\chi_\Theta(\boldsymbol h)=\det(h_0)^{-\theta_0}\cdots\det(h_{N})^{-\theta_{N}}.
\end{equation}
A representation $X$ of the nested ADHM quiver \eqref{quiv_tail} is $\chi_\Theta-$(semi)stable iff it is $\Theta-$(semi)stable.
\end{proposition}
Since the proof for proposition \ref{prop:equivalence_stabilities} deeply relies on some known results about equivalent characterizations of $\chi-$stability, we will first recall them. In full generality, let $V$ be a vector space over $\mathbb C$ equipped with the action of a connected subgroup $G$ of $U(V)$, whose complexification is denoted by $G^{\mathbb C}$. Then if $\chi:G\to U(1)$ is a character of $G$, we can extend it to form its complexification $\chi:G^{\mathbb C}\to\mathbb C^*$. We then form the trivial line bundle $V\times\mathbb C$, which carries an action of $G^{\mathbb C}$ via $\chi$:
\begin{displaymath}
g\cdot (x,z)=(g\cdot x,\chi(g)^{-1}z),\qquad g\in G,(x,z)\in V\times\mathbb C.
\end{displaymath}
\begin{definition}
An element $x\in V$ is
\begin{enumerate}
\item $\chi-$semistable if there exists a polynomial $f\in A(V)^{G^{\mathbb C},\chi^n}$, with $n\ge 1$ such that $f(x)\neq 0$;
\item $\chi-$stable if it satisfies the previous condition and if
\begin{enumerate}
\item $\dim (G^{\mathbb C}\cdot x)=\dim(G^{\mathbb C}/\Delta)$, where $\Delta\subseteq G^{\mathbb C}$ is the subgroup of $G^{\mathbb C}$ acting trivially on $V$;
\item the action of $G^{\mathbb C}$ on $\{x\in V:f(x)\neq 0\}$ is closed.
\end{enumerate}
\end{enumerate}
\end{definition}
Given the previous definition, the next lemma due to King \cite{king} gives an alternative characterization of $\chi-$(semi)stable points under the $G^{\mathbb C}-$action.
\begin{lemma}[Lemma 2.2 and Proposition 2.5 in \cite{king}]
Given the character $\chi:G^{\mathbb C}\to\mathbb C^*$ for the action of $G^{\mathbb C}$ on the vector space $V$, and the lift of this action to the trivial line bundle $V\times\mathbb C$, a point $x\in V$ is
\begin{enumerate}
\item $\chi-$semistable iff $\overline{G^{\mathbb C}\cdot (x,z)}\cap (V\times\{0\})=\emptyset$, for any $z\neq 0$;
\item $\chi-$stable iff $G^{\mathbb C}\cdot(x,z)$ is closed and the stabilizer of $(x,z)$ contains $\Delta$ with finite index.
\end{enumerate}
Equivalently, a point $x\in V$ is
\begin{enumerate}
\item $\chi-$semistable iff $\chi(\Delta)=\{1\}$ and $\chi(\lambda)\ge 0$ for any $1-$parameter subgroup $\lambda(t)\subseteq G^{\mathbb C}$ for which $\lim_{t\to 0}\lambda(t)\cdot x$ exists;
\item $\chi-$stable iff the only $\lambda(t)$ such that $\lim_{t\to 0}\lambda(t)\cdot x$ exists and $\chi(\lambda)=0$ are in $\Delta$.
\end{enumerate}
\end{lemma}
With these notations, if $V^{\rm{ss}}(\chi)$ denotes the set of $\chi-$semistable points of $V$, $V\git_\chi G^{\mathbb C}$ can be identified with $V^{\rm{ss}}(\chi)/\sim$, where $x\sim y$ in $V^{\rm{ss}}(\chi)$ iff $\overline{G^{\mathbb C}\cdot x}\cap\overline{G^{\mathbb C}\cdot y}\neq\emptyset$ in $V^{\rm{ss}}(\chi)$.

\begin{proof}[Proof of proposition \ref{prop:equivalence_stabilities}]
Take a $\theta-$semistable representation $X\in\mathbb X$ and assume it doesn't satisfy $\chi_\theta-$semistability. Then there exists a $1-$parameter subgroup $\lambda(t)$ of $\mathcal G$ such that $\lim_{t\to 0}\lambda(t)\cdot X$ exists and $\chi(\lambda)<0$. However each such $1-$parameter subgroup $\lambda$ determines a filtration $\cdots\supseteq X_n\supseteq X_{n+1}\supseteq\cdots$ of subrepresentations of $X$, \cite{king}, and
\begin{equation}
\chi_\theta(\lambda)=-\sum_{n\in\mathbb Z}\theta(X_n)\ge 0,
\end{equation}
thus proving one side of the proposition, as the part concerning stability is obvious from the fact that trivial subrepresentations of $X$ correspond to subgroups in $\Delta$.

Conversely, if $X$ is a $\chi_{\theta}-$semistable representation, we want to show that it is also a $\theta-$semistable one. We only need to consider two cases, corresponding to subrepresentations $\tilde X$ of $X$ with $\tilde r=r$ or $\tilde r=0$. Each vector space in $X$, say $V_i$ will have then a direct sum decomposition $V_i=\widetilde V_i\oplus\widehat V_i$. We will then take a $1-$parameter subgroup $\lambda(t)$ such that
\begin{equation}
\lambda_i(t)=
\begin{bmatrix}
t\mathbbm 1_{\widetilde V_i} & 0\\
0 & \mathbbm 1_{\widehat V_i}
\end{bmatrix}.
\end{equation}
Then one can easily compute
\begin{equation}
\begin{split}
\chi_{\theta}(\lambda(t))\cdot z&=\left[\det(\lambda_0(t))^{-\theta_0}\cdots\det(\lambda_N(t))^{-\theta_N}\right]^{-1}\cdot z\\
&=t^{\tilde{\bold n}\cdot\boldsymbol\theta}z
\end{split}
\end{equation}
It is then a matter of a simple computation to verify that, if $X$ wasn't $\theta-$semistable, then one would have had $\lim_{t\to 0}\lambda(t)\cdot X\in\mathbb X\times\{0\}$, thus contradicting the $\chi_\theta-$semistability. A completely analogous computation can be carried over when $\tilde r=r$, taking
\begin{equation}
\lambda_0(t)=
\begin{bmatrix}
\mathbbm 1_{\widetilde V_0} & 0\\
0 & t^{-1}\mathbbm 1_{\widehat V_0}
\end{bmatrix},\qquad
\lambda_i(t)=
\begin{bmatrix}
\mathbbm 1_{\widetilde V_i} & 0\\
0 & t^{-1}\mathbbm 1_{\widehat V_i}
\end{bmatrix},\ i>0,
\end{equation}
and since $(\tilde{\bold n}-\bold n)\cdot\boldsymbol\theta>0$ if $X$ is supposed not to be $\theta-$semistable, this would still lead to a contradiction.

Finally, if $X$ was to be $\chi_\theta-$stable but not $\theta-$stable, the $1-$parameter subgroups previously described would have stabilized the pair $(X,z)$, $z\neq 0$, in the two different cases $\tilde r=0$ and $\tilde r=r$ respectively, thus again giving rise to a contradiction.

\end{proof}
\begin{corollary}\label{corollary_stability}
Given a representation of the nested instantons quiver \eqref{quiv_tail} of numerical type $(r,\bold n)$, there exists a chamber in $\mathbb Q^{N+1}\ni(\boldsymbol\theta,\theta_\infty)=\Theta$ in which $\theta_{i>0}>0$ and $\theta_0+n_1\theta_1+\cdots+n_{s-1}\theta_{s-1}<0$ such that the following are equivalent:
\begin{enumerate}
\item X is $\Theta-$semistable;
\item X is $\Theta-$stable;
\item X is $\chi_\Theta-$semistable;
\item X is $\chi_\Theta-$stable;
\item X satisfies {\bf S1} and {\bf S2} in proposition \ref{theta_stability}.
\end{enumerate}
\end{corollary}
Because of the previous corollary, in the stability chamber defined by proposition \ref{theta_stability} all notions of stability are actually the same, so that a representation satisfying anyone of the conditions in corollary \ref{corollary_stability} will be called stable, and the corresponding $\mathcal N^{ss}_{\chi_\Theta}(r,\bold n)=\mathcal N(r,\bold n)\simeq\mathcal N_{r,[r^1],n,\mu}$ (with the notations of \cite{Bonelli:2019lal}) will be addressed to as the moduli space of stable representations of \eqref{quiv_tail} or, equivalently, as the moduli space of nested instantons. Altogether, the previous considerations prove the following theorem.\footnote{We thank Valeriano Lanza for pointing out to us a correction to the original proof for the two-nodes quiver found in \cite{flach_jardim}.}
\begin{theorem}
The moduli space $\mathcal N(r,\bold n)$ of stable representation of the nested instantons quiver of numerical type $(r,\bold n)$ is a virtually smooth quasi-projective variety equipped with a natural action of $\mathbb T=T\times (\mathbb C^*)^r$, $T=(\mathbb C^*)^2$, and a perfect obstruction theory. The moduli space $\mathcal N(r,\bold n)$ can thus be identified in a suitable stability chamber with the moduli space of nested instantons $\mathcal N_{r,[r^1],n,\mu}$.
\end{theorem}

\begin{proof}
The first part of the proof has already been proved. Consider then the following complex

\begin{equation}\label{tang_obs_N}
\mathcal C(X):
\begin{tikzcd}
\mathcal C^0(X)\arrow[r,"d_0"] & \mathcal C^1(X)\arrow[r,"d_1"] & \mathcal C^2(X)\arrow[r,"d_2"] & \mathcal C^3(X)
\end{tikzcd}
\end{equation}
with
\begin{equation}
\left\{
\begin{aligned}
& \mathcal C^0(X)=\bigoplus_{i=0}^N\End(V_i),\\
& \mathcal C^1(X)=\End(V_0)^{\oplus 2}\oplus\Hom(W,V_0)\oplus\Hom(V_0,W)\oplus\left[\bigoplus_{i=1}^N\left(\End(V_i)^{\oplus 2}\oplus\Hom(V_i,V_{i-1}\right)\right],\\
& \mathcal C^2(X)=\End(V_0)\oplus\Hom(V_1,W)\oplus\left[\bigoplus_{i=1}^N\left(\Hom(V_i,V_{i-1})^{\oplus 2}\oplus\End(V_i)\right)\right],\\
& \mathcal C^3(X)=\bigoplus_{i=1}^N\Hom(V_i,V_{i-1}),
\end{aligned}
\right.
\end{equation}
while the morphisms $d_i$ are defined as:
\begin{equation}
\begin{aligned}
&d_0(\boldsymbol h) = \begin{pmatrix}
[h_0,B_1^0]\\
[h_0,B_2^0]\\
h_0I\\
-Jh_0\\
[h_1,B_1^1]\\
[h_1,B_2^1]\\
h_0F^1-F^1h^1\\
\vdots
\end{pmatrix},\qquad 
d_1\begin{pmatrix}
b_1^0\\
b_2^0\\
i\\
j\\
b_1^1\\
b_2^1\\
f^1\\
\vdots
\end{pmatrix}=\begin{pmatrix}
[b_1^0,B_2^0]+[B_1^0,b_2^0]+iJ+Ij\\
jF^1+Jf^1\\
B_1^0f^1+b_1^0F^1-F^1b_1^1-f^1B_1^1\\
B_2^0f^1+b_2^0F^1-F^1b_2^1-f^1B_2^1\\
\vdots\\
[b_1^1,B_2^1]+[B_1^1,b_2^1]\\
\vdots\\
[b_1^N,B_2^N]+[B_1^N,b_2^N]
\end{pmatrix},\\
&d_2\begin{pmatrix}
c_1\\
c_2\\
c_3\\
\vdots\\
c_{3N+3}
\end{pmatrix}=\begin{pmatrix}
c_1F^1+B_2^0c_3-c_3B_2^1+c_4B_1^1-B_1^0c_4-Ic_2-F^1c_{2N+3}\\
\vdots\\
c_{2N+2+i}F^{i}+B_2^ic_{2+2i}-c_{2+2i}B_2^1+c_{3+2i}B_1^1-B_1^0c_{3+2i}-F^{i}c_{2N+3+i}\\
\vdots
\end{pmatrix}.
\end{aligned}
\end{equation}

Notice that the maps $d_0$ and $d_1$ are the linearisation of the action of $\mathcal G$ on $\mathbb X$ and of the nested instantons quiver relations (neglecting $G^i$, since $G^i=0,\forall i$), respectively. The morphism $d_2$ is instead signalling the fact that the quiver relations are not all independent.

Our claim is then that $H^0(\mathcal C(X))=H^3(\mathcal C(X))=0$, and that $\mathcal C(X)$ is an explicit representation of the perfect obstruction theory complex, so $H^1(\mathcal C(X))$ will be identified with the Zariski tangent to $\mathcal N(r,\bold n)$, while $H^2(\mathcal C(X))$ will encode the obstructions to its smoothness. In fact elements of $H^1(\mathcal C(X))$ parametrize infinitesimal displacements at given points, up to the $\mathcal G-$action, so $H^1(\mathcal C(X))$ provides a local model for the Zariski tangent space to $\mathcal N(r,\bold n)$. In the same way $H^2(\mathcal C(X))$ is interpreted to be the local model for the obstructions as its elements encode the linear dependence of the nested ADHM equations. Actually one might explicitly determine the truncated cotangent complex $\mathbb L_{\mathcal N(r,\bold n)}^{\bullet\ge 1}$ by a standard computation in deformation theory along the lines of \cite{Diaconescu:2008ct} and compute extension and obstruction classes in terms of the cohomology of the complex $\mathcal C(X)$.

In order to see that indeed the $0-$th and $3-$rd cohomology of $\mathcal C(X)$ does indeed vanish, we construct three other complexes $\mathcal C(\mathcal A)$, $\mathcal C(\mathcal B)$ and $\mathcal C(\mathcal A,\mathcal B)$:
\begin{equation}
\mathcal C(\mathcal A):
\begin{tikzcd}[row sep=-1mm]
 & \End(V_0)^{\oplus 2} & \\
 & \oplus & \\
\End(V_0)\arrow[r,"d_0"] & \Hom(W,V_0)\arrow[r,"d_1"] & \End(V_0)\\
 & \oplus & \\
 & \Hom(V_0,W)
\end{tikzcd}
\end{equation}
with
\begin{displaymath}
d_0(h_0)=\begin{pmatrix}
[h_0,B_1^0]\\
[h_0,B_2^0]\\
h_0I\\
-Jh_0
\end{pmatrix},\qquad
d_1\begin{pmatrix}
b_1^0\\
b_2^0\\
i\\
j
\end{pmatrix}=[b_1^0,B_2^0]+[B_1^0,b_2^0]+Ij+iJ;
\end{displaymath}

\begin{equation}
\mathcal C(\mathcal B):
\begin{tikzcd}[row sep=-1mm]
\bigoplus_{i=1}^N\End(V_i)\arrow[r,"d_0"] & \bigoplus_{i=1}^N\End(V_i)^{\oplus 2}\arrow[r,"d_1"] & \bigoplus_{i=1}^N\End(V_i)
\end{tikzcd}
\end{equation}
with
\begin{displaymath}
d_0\begin{pmatrix}
h_1\\
\vdots\\
h_N
\end{pmatrix}=\begin{pmatrix}
[h_1,B_1^1]\\
[h_1,B_2^1]\\
\vdots\\
[h_N,B_1^N]\\
[h_N,B_2^N]
\end{pmatrix},\qquad
d_1\begin{pmatrix}
b_1^1\\
b_2^1\\
\vdots\\
b_1^N\\
b_2^N
\end{pmatrix}=\begin{pmatrix}
[b_1^1,B_2^1]+[B_1^1,b_2^1]\\
\vdots\\
[b_1^N,B_2^N]+[B_1^N,b_2^N]
\end{pmatrix};
\end{displaymath}

\begin{equation}
\mathcal C(\mathcal A,\mathcal B):
\begin{tikzcd}[row sep=-1mm]
 & \bigoplus_{i=1}^N\Hom(V_i,V_{i-1})^{\oplus 2} & \\
\bigoplus_{i=1}^N\Hom(V_i,V_{i-1})\arrow[r,"d_0"] & \qquad\qquad\oplus\qquad\qquad\arrow[r,"d_1"] & \bigoplus_{i=1}^N\Hom(V_i,V_{i-1})\\
 & \Hom(V_1,W)
\end{tikzcd}
\end{equation}
with
\begin{displaymath}
d_0\begin{pmatrix}
f^1\\
\vdots\\
f^N
\end{pmatrix}=\begin{pmatrix}
-B_1^0f^1+f^1B_1^1\\
-B_2^0f^1+f^1B_2^1\\
\vdots
-B_1^{N-1}f^N+f^NB_1^N\\
-B_2^{N-1}f^N+f^NB_2^N\\
-Jf^1
\end{pmatrix},\qquad
d_1\begin{pmatrix}
c_3\\
\vdots\\
c_{2N+2}\\
c_2
\end{pmatrix}=\begin{pmatrix}
-B_2^0c_3+c_3B_2^1-c_4B_1^1+B_1^0c_4+Ic_2\\
\vdots\\
B_2^{N-1}c_{2N+2}-c_{2N+2}B_2^1+c_{2N+3}B_1^1-B_1^0c_{2N+3}
\end{pmatrix}.
\end{displaymath}

Then one can prove that there exists a distinguished triangle
\begin{equation}\label{dist_triangle}
\begin{tikzcd}
\mathcal C(X)\arrow[r] & \mathcal C(\mathcal A)\oplus\mathcal C(\mathcal B)\arrow[r,"\rho"] & \mathcal C(\mathcal A, \mathcal B)
\end{tikzcd},
\end{equation}
coming from the fact that $\mathcal C(X)[1]$ is a cone for $\rho=(\rho_0,\rho_1,\rho_2)$, where
\begin{subequations}
\begin{align}
\rho_0\begin{pmatrix}
h_0\\
\vdots\\
h_N
\end{pmatrix}&=\begin{pmatrix}
-h_0F^1+F^1h_1\\
\vdots\\
-h_{N-1}F^N+F^Nh_N
\end{pmatrix}\\
\rho_1\begin{pmatrix}
b_1^0\\
b_2^0\\
i\\
j\\
b_1^1\\
b_2^1\\
\vdots
\end{pmatrix}&=\begin{pmatrix}
-b_1^0F^1+F^1b_1^1\\
-b_2^0F^1+F^1b_2^1\\
\vdots\\
-b_1^{N-1}F^N+F^Nb_1^N\\
-b_2^{N-1}F^N+F^Nb_2^N\\
-jF^1
\end{pmatrix}\\
\rho_2\begin{pmatrix}
c_1\\
c_{2N+3}\\
\vdots\\
c_{3N+3}
\end{pmatrix}&=\begin{pmatrix}
-c_1F^1+F^1c_{2N+3}\\
\vdots\\
-c_{3N+2}F^N+F^Nc_{3N+3}
\end{pmatrix}
\end{align}
\end{subequations}
By the triangle \eqref{dist_triangle} one gets the long sequence of cohomologies:
\begin{equation}
\begin{split}
0\longrightarrow H^0(\mathcal C(X))\longrightarrow H^0(\mathcal C(\mathcal A)\oplus\mathcal C(\mathcal B))\xrightarrow{H^0(\rho)}H^0(\mathcal C(\mathcal A,\mathcal B))\longrightarrow H^1(\mathcal C(X))\longrightarrow\\
\longrightarrow H^1(\mathcal C(\mathcal A)\oplus\mathcal C(\mathcal B))\xrightarrow{H^1(\rho)} H^1(\mathcal C(\mathcal A,\mathcal B))\longrightarrow H^2(\mathcal C(X))\longrightarrow H^2(\mathcal C(\mathcal A)\oplus\mathcal C(\mathcal B))\longrightarrow\\
\longrightarrow H^2(\mathcal C(\mathcal A,\mathcal B))\longrightarrow H^3(\mathcal C(X))\longrightarrow 0,
\end{split}
\end{equation}
and, since $\mathcal A$ is a stable representation of the standard ADHM quiver, $H^0(\mathcal C(\mathcal A))=H^2(\mathcal C(\mathcal A))=0$. Then $H^0(\mathcal C(X))=0$ by the injectivity of $H^0(\rho):H^0(\mathcal C(\mathcal B))\to H^0(\mathcal C(\mathcal A,\mathcal B))$. In fact we have
\begin{equation}
H^0(\rho)\begin{pmatrix}
h_1\\
\vdots\\
h_N
\end{pmatrix}=0\Rightarrow\begin{pmatrix}
F^1h_1\\
\vdots\\
F^Nh_N
\end{pmatrix}=0\Rightarrow\begin{pmatrix}
h_1\\
\vdots\\
h_N
\end{pmatrix}=0,
\end{equation}
since $F^i$ is injective. Moreover the stability of $X$ implies that $d_1:\mathcal C(\mathcal A,\mathcal B)^1\to\mathcal C(\mathcal A,\mathcal B)^2$ is surjective: this in turn means that $H^2(\mathcal C(\mathcal A,\mathcal B))=0$, which implies $H^3(\mathcal C(X))=0$. In fact let's take $d_1^\vee$:
\begin{equation}
d_1^\vee(\boldsymbol\phi)=d_1^\vee\begin{pmatrix}
\phi_1\\
\vdots\\
\phi_N
\end{pmatrix}=\begin{pmatrix}
B_2^1\phi_1-\phi_1B_2^0\\
-B_1^1\phi_1+\phi_1B_1^1\\
\vdots\\
B_2^N\phi_N-\phi_NB_2^{N-1}\\
-B_1^N\phi_N+\phi_NB_1^{N-1}\\
\phi_1I
\end{pmatrix},
\end{equation}
and if $\boldsymbol\phi\in\ker(d_1^\vee)$ then $\ker(\phi_1)$ would be a $(B_1^0,B_2^0)-$invariant subset of $V_0$ containing $\im(I)$ which contradicts the stability of $X$, by which we conclude that $\ker(\phi_1)=V_0$. Similar statements hold also for each other component of $\boldsymbol\phi$, which we then conclude to be $\boldsymbol\phi=0$.

The only thing left to prove is the virtual smoothness, namely that the moduli space $\mathcal N(r,\bold n)$ of stable representations of the nested instantons quiver is embedded in a smooth variety which is obtained as an hyperk\"ahler quotient. We will leave this for section \ref{sec:hyperkahler}.
\end{proof}

For future reference we want now to exhibit some morphisms between different nested instantons moduli spaces and between them and usual moduli spaces of instantons, which are moduli spaces of framed torsion-free sheaves on $\mathbb P^2$. We obviously have iterative forgetting projections $\eta_i:\mathcal N(r,n_0,\dots,n_i)\to\mathcal N(r,n_0,\dots,n_{i-1})$. Moreover we also have other morphisms to underlying Hilbert schemes of points on $\mathbb C^2$, which are summarized by the commutative diagram in figure \ref{comm_diag_tail}.
\begin{figure}[htb!]
\begin{center}
\includegraphics[width=\textwidth]{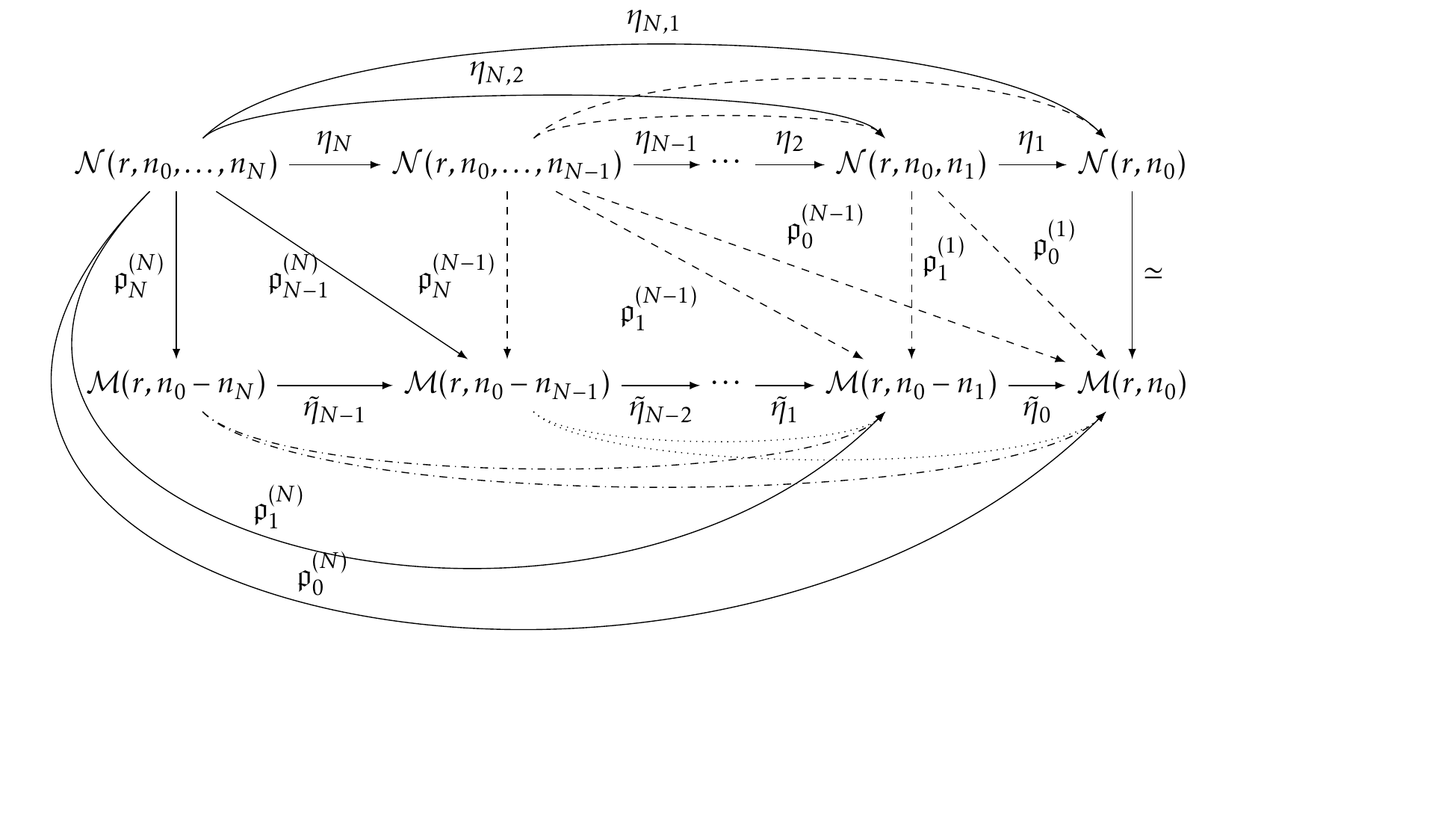}
\caption{Morphisms between moduli spaces of flags of sheaves and moduli spaces of sheaves.}\label{comm_diag_tail}
\end{center}
\end{figure}
In order to see that these maps do indeed exist, take a stable representation $[X]$ of the nested instantons quiver. The fact that $[X]$ is stable implies that the morphisms $F^i$ are injective, so that we can construct the stable ADHM datum $(W,\tilde V_i,\tilde B_1^i,\tilde B_2^i,\tilde I^i,\tilde J^i)$ as follows. Let $\tilde V_i$ be $V_0/\im(F^1\cdots F^i)$ and choose a basis of $V_i$ in such a way that
\begin{displaymath}
F^1\cdots F^i=\begin{pmatrix}
\mathbbm 1_{V_i}\\ 0
\end{pmatrix},\qquad F^1\circ F^2\circ\cdots\circ F^i:V_i\to V_0,
\end{displaymath} 
whence $V_0=V_i\oplus\tilde V_i$. Then define the projections $\pi_i=V_0\to V_i$ and $\tilde\pi_i:V_0\to\tilde V_i$ as $\pi_i(v,\tilde v)=v$ and $\tilde\pi_i(v,\tilde v)=\tilde v$, with $v\in V_i$, $\tilde v\in\tilde V_i$. We can then show how $\tilde V_i$ inherits an ADHM structure by its embedding in $V_0$. Indeed if we define $\tilde B_1^i=B_1^0|_{\tilde V_i}$, $\tilde B_2^i=B_2^0|_{\tilde V_i}$, $\tilde I^i=\tilde\pi_i\circ I$ and $\tilde J^i=J|_{\tilde V_i}$ the datum $(W,\tilde V_i,\tilde B_1^i,\tilde B_2^i,\tilde I^i,\tilde J^i)$ satisfies the ADHM equation \eqref{adhm_proj}.
\begin{equation}\label{adhm_proj}
[\tilde B_1^i,\tilde B_2^i]+\tilde I^i\tilde J^i=[B_1^0|_{\tilde V_i},B_2^0|_{\tilde V_i}]+\tilde\pi_i\circ I\circ J|_{\tilde V_i}=\bigg(\left[B_1^0,B_2^0\right]+IJ\bigg)\bigg|_{\tilde V_i}=0.
\end{equation}
This new ADHM datum is moreover stable, as if it would exist $0\subset\tilde S_i\subset\tilde V_i$ such that $\tilde B_{1,2}^i(\tilde S_i),\tilde I^i(W)\subset\tilde{S_i}$ it would imply that also the ADHM datum $(W,V_0,B_1^0,B_2^0,I,J)$ wouldn't be stable. In fact in that case we could take $0\subset V_i\oplus\tilde S_i\subset V_0$ and it would be such that $B_1^0(V_i\oplus\tilde S_i),B_2^0(V_i\oplus\tilde S_i),I(W)\subset V_i\oplus\tilde S_i$. In fact if we take any $(v,\tilde s)\in V_i\oplus\tilde S_i$ it happens that $B_1^0(v,s)=(B_1^0|_{V_i}(v),B_1^0|_{\tilde V_i}\tilde s)=(B_1^0|_{V_i},\tilde B_1^i(\tilde s))\in V_i\oplus\tilde S_i$, $B_2^0(v,s)=(B_2^0|_{V_i}(v),B_2^0|_{\tilde V_i}\tilde s)=(B_2^0|_{V_i},\tilde B_2^i(\tilde s))\in V_i\oplus\tilde S_i$ and $I(W)=I(W)\cap V_i\oplus I(W)\cap\tilde V_i=(\pi_i\circ I)(W)\oplus(\tilde\pi_i\circ I)(W)\subset V_i\oplus\tilde S_i$. Thus we constructed a map $\mathfrak p^{(N)}_i:\mathcal N(r,n_0,\dots,n_{N})\to\mathcal M(r,n_0-n_i)$.

\subsection{Virtual smoothness}\label{sec:hyperkahler}
In this section we exhibit an embedding of the moduli space of nested instantons into a smooth projective variety, which is moreover hyperk\"ahler. In the following vector space
\begin{equation}\label{direct_sum_dec_M}
\begin{split}
\mathbb X=\End(V_0)^{\oplus 2}\oplus\Hom(W,V_0)\oplus\Hom(V_0,W)\bigoplus_{k=1}^{N}\left[\End(V_k)^{\oplus 2}\oplus\Hom(V_{k-1},V_k)\right.\\
\left.\oplus\Hom(V_k,V_{k-1})\right]
\end{split}
\end{equation}
we will introduce a family of relations:
\begin{align}\label{relsM1}
& [B_1^0,B_2^0]+IJ+F^1G^1=0,\\ \label{relsM2}
& [B_1^i,B_2^i]-G^iF^i+F^{i+1}G^{i+1}=0,\quad i=1,\dots,N.
\end{align}
Then an element $(B_1^0,B_2^0,I,J,\{B_1^i,B_2^i,F^i,G^i\})X\in\mathbb X$ is called stable if it satisfies conditions  $\bold{S1}$ and $\bold{S2}$ in proposition \ref{theta_stability}. With these conventions we will define $\mathbb M(r,\bold n)$ to be the space of stable elements of $\mathbb X$ satisfying the relations \eqref{relsM1}-\eqref{relsM2}:
\begin{equation}
\mathbb M(r,\bold n)=\{X\in\mathbb X:X\ \rm{is\ stable\ and\ satisfies\ }\eqref{relsM1},\eqref{relsM2}\}.
\end{equation}
Exactly in the same way as we did before we can easily see that there is a natural action of $\mathcal G=GL(V_0)\times\cdots\times GL(V_N)$ which is free on $\mathbb M(r,\bold n)$ and preserves the equations \eqref{relsM1}-\eqref{relsM2}: the same is then true for the natural $\mathcal U-$action on $\mathbb M(r,\bold n)$, with $\mathcal U=U(V_0)\times\cdots\times U(V_N)$. Thus a moduli space $\mathcal M(r,\bold n)$ of stable $\mathcal U-$orbits in $\mathbb M(r,\bold n)$ can be defined by means of GIT theory, as it was the case for $\mathcal N(r,\bold n)$ in the previous sections. It is moreover obvious that $\mathcal N(r,\bold n)\hookrightarrow\mathcal M(r,\bold n)$, as any stable point of $\mathbb X$ satisfying the nested ADHM equations automatically satisfies \eqref{relsM1} and \eqref{relsM2}.

Next let us point out that on each $T\Hom(V_i,V_k)$ we can introduce an hermitean metric by defining
\begin{equation}
\langle X,Y\rangle=\frac{1}{2}\Tr\left(X\cdot Y^\dagger+X^\dagger\cdot Y\right),\qquad\forall X,Y\in\Hom(V_i,V_k),
\end{equation}
which in turn can be linearly extended to a hermitean metric $\langle-,-\rangle:T\mathbb M(r,\bold n)\times T\mathbb M(r,\bold n)\to\mathbb C$. Finally we can introduce some complex structures on $T\mathbb M(r,\bold n)$: given $X\in T\mathbb M(r,\bold n)$ these are defined as the following $I,J,K\in\End(T\mathbb M(r,\bold n))$
\begin{align}
& I(X)=\sqrt{-1}X,\\
& J(X)=(-b_2^{0\dagger},b_1^{0\dagger},-j^\dagger,i^\dagger,\{-b_2^{i\dagger},b_1^{i\dagger},-g^{i\dagger},f^{i\dagger}\}),\\
& K(X)=I\circ J(X),
\end{align}
with $X=(b_1^0,b_2^0,i,j,\{b_1^i,b_2^i,f^i,g^i\})$. These three complex structures make the datum of 
$$(\mathbb M(r,\bold n),\langle-,-\rangle,I,J,K)$$
a hyperk\"ahler manifold, as one can readily verify. It is a standard fact that once we fix a particular complex structure, say $I$, and its respective K\"ahler form, $\omega_I$, the linear combination $\omega_{\mathbb C}=\omega_J+\sqrt{-1}\omega_K$ is a holomorphic symplectic form for $\mathbb M(r,\bold n)$.  The thing we finally want to prove is that the hyperk\"ahler structure on $\mathbb M(r,\bold n)$ induce a hyperk\"ahler structure on the GIT quotient $\mathcal M(r,\bold n)$, which will be moreover proven to be smooth. This is made possible by the fact that the natural $\mathcal U-$action on $\mathbb M(r,\bold n)$ preserves the hermitean metric and the complex structures we introduced. Then, letting $\mathfrak u$ be the Lie algebra of the group $\mathcal U$, we need to construct a moment map
\begin{displaymath}
\mu:\mathbb M(r,\bold n)\to \mathfrak u^*\otimes\mathbb R^3,
\end{displaymath}
satisfying
\begin{enumerate}
\item $\mathcal G-$equivariance: $\mu(g\cdot X)=\Ad_{g^-1}^*\mu(X)$;
\item $\langle\de\mu_i(X),\xi\rangle=\omega_i(\xi^*,X)$, for any $X\in T\mathbb M(r,\bold n)$ and $\xi\in\mathfrak u$ generating the vector field $\xi^*\in T\mathbb M(r,\bold n)$.
\end{enumerate}
If then $\zeta\in\mathfrak u^*\otimes\mathbb R^3$ is such that $\Ad_g^*(\zeta_i)=\zeta_i$ for any $g\in\mathcal U$, $\mu^{-1}(\zeta)$ is $\mathcal U-$invariant and it makes sense to consider the quotient space $\mu^{-1}(\zeta)/\mathcal U$. It is known, \cite{Hitchin1987}, that if $\mathcal U$ acts freely on $\mu^{-1}(\zeta)/\mathcal U$, the latter is a smooth hyperk\"ahler manifold, with complex structures and metric induced by those of $\mathbb M(r,\bold n)$.

Our task of finding a moment map $\mu:\mathbb M(r,\bold n)\to\mathfrak u^*\otimes\mathbb R^3$ then translates into the following. Define $(\mu_1^0,\dots,\mu_1^N)=\mu_1:\mathbb M(r,\bold n)\to\mathfrak u$
\begin{equation}\label{mu1}
\left\{
\begin{aligned}
&\mu_1^0(X)=\frac{\sqrt{-1}}{2}\left([B_1^0,B_1^{0\dagger}]+[B_2^0,B_2^{0\dagger}]+II^\dagger-J^\dagger J+F^1F^{1\dagger}-G^{1\dagger}G^1\right)\\
&\mu_1^1(X)=\frac{\sqrt{-1}}{2}\left([B_1^1,B_1^{1\dagger}]+[B_2^1,B_2^{1\dagger}]-F^{1\dagger}F^1+G^1G^{1\dagger}+F^2F^{2\dagger}-G^{2\dagger}G^2\right)\\
&\qquad\qquad\vdots\\
&\mu_1^N(X)=\frac{\sqrt{-1}}{2}\left([B_1^N,B_1^{N\dagger}]+[B_2^N,B_2^{N\dagger}]-F^{N\dagger}F^N+G^NG^{N\dagger}\right),
\end{aligned}
\right.
\end{equation}
with $X=(B_1^0,B_2^0,I,J,\{B_1^i,B_2^i,F^i,G^i\})\in\mathbb M(r,\bold n)$. In addition to $\mu_1$ we also define a map $\mu_{\mathbb C}:\mathbb M(r,\bold n)\to\mathfrak{g}$, with $\mathfrak g=\mathfrak{gl}(V_0)\times\cdots\times\mathfrak{gl}(V_N)$:
\begin{equation}\label{muC}
\left\{
\begin{aligned}
&\mu_{\mathbb C}^0(X)=[B_1^0,B_2^0]+IJ+F^1G^1\\
&\mu_{\mathbb C}^1(X)=[B_1^1,B_2^1]-G^1F^1+F^2G^2\\
&\qquad\qquad\vdots\\
&\mu_{\mathbb C}^N(X)=[B_1^N,B_2^N]-G^NF^N,
\end{aligned}
\right.
\end{equation}
by means of which we define $\mu_{2,3}:\mathbb M(r,\bold n)\to\mathfrak u$ as $\mu_{\mathbb C}(X)=(\mu_2+\sqrt{-1}\mu_3)(X)$. Notice that in absence of $B_i^j$ and $I,J$ the complex moment map we defined would reduce to the Crawley-Boevey moment map in \cite{crawley-boevey_2001}. We then claim that $\mu=(\mu_1,\mu_2,\mu_3)$ is a moment map for the $\mathcal U-$action on $\mathbb M(r,\bold n)$. If this is true and $\chi$ is the algebraic character we introduced in section \ref{sec:moduli_space}, the space
\begin{equation}
\widetilde{\mathcal M}(r,\bold n)=\frac{\mu_1^{-1}(\sqrt{-1}\de\chi)\cap\mu_{\mathbb C}^{-1}(0)\cap\mathbb M(r,\bold n)}{\mathcal U}=\frac{\mu^{-1}(\sqrt{-1}\de\chi,0,0)\cap\mathbb M(r,\bold n)}{\mathcal U}
\end{equation}
is a smooth hyperk\"ahler manifold which, by an analogue of Kempf-Ness theorem is also isomorphic to $\mathcal M(r,\bold n)$. In fact it is known, due to a result of \cite{king,nakajima_lectures} and the characterization of $\chi-$(semi)stable points we gave in the previous sections, that there exists a bijection between $\mu_1^{-1}(\sqrt{-1}\de\chi)$ and the set of $\chi-$(semi)stable points in $\mu_{\mathbb C}^{-1}(0)$. Then, in order to prove that $\mu$ is actually a moment map, we will first compute the vector field $\xi^*$ generated by a generic $\xi\in\mathfrak u$. Let then $X=(b_1^0,b_2^0,i,j,\{b_1^i,b_2^i,f^i,g^i\})$ be a vector in $T\mathbb M(r,\bold n)$ and $\Psi_X:\mathcal U\to\mathbb M(r,\bold n)$ the action of $\mathcal U$ onto $X\in\mathbb M(r,\bold n)$: the fundamental vector field generated by $\xi\in\mathfrak u$ is
\begin{equation}
\xi^*|_X=\de\Psi_X(\mathbbm 1_{\mathcal U})(\xi)=\frac{\de}{\de t}\left.\left(\Psi_X\circ\gamma\right)\right|_{t=0},
\end{equation}
where $\gamma$ is a smooth curve $\gamma:(-\epsilon,\epsilon)\to\mathcal U$ such that $\gamma(0)=\mathbbm 1_{\mathcal U}$ and $\dot\gamma(0)=\xi$. Thus we can compute
\begin{equation}
\begin{split}
\xi^*|_X=\left([\xi_0,b_1^0],[\xi_0,b_2^0],\xi_0i,-j\xi_0,[\xi_1,b_1^1],[\xi_1,b_2^1],\xi_0f^1-f^1\xi_1,\xi_1g^1-g^1\xi_0,\dots\right.\\
\left.\dots,[\xi_N,b_1^N],[\xi_N,b_2^N],\xi_{N-1}f^N-f^N\xi_N,\xi_Ng^N-g^N\xi_{N-1}\right).
\end{split}
\end{equation}
Then if $\pi_i:\mathbb M(r,\bold n)\to\mathbb M(r,\bold n)$ denotes the projection on the $i-$th component of the direct sum decomposition induced by \eqref{direct_sum_dec_M} so that $i$ runs over the index set $\mathcal I$, by inspection one can see that $\omega_1$ is exact, and in particular $\omega_1=\de\lambda_1$, with
\begin{equation}
\lambda_1=\frac{\sqrt{-1}}{2}\Tr\left(\sum_{i\in\mathcal I}\pi_i\wedge\pi_{i*}^\dagger\right).
\end{equation}
This implies that
\begin{equation}
\langle\mu_1(x),\xi\rangle=\imath_{\xi^*}\lambda_1,
\end{equation}
and it is easy to verify that $\mu_1:\mathbb M(r,\bold n)\to\mathfrak u^*$ thus defined indeed matches with the definition \eqref{mu1}. Similarly one can realize that
\begin{align}
&\lambda_2=\Re\left[\Tr\left(\sum_{i\in2\mathbb Z\cap\mathcal I}\pi_i\wedge\pi_{1+1*}\right)\right],\\
&\lambda_3=-\sqrt{-1}\Im\left[\Tr\left(\sum_{i\in2\mathbb Z\cap\mathcal I}\pi_i\wedge\pi_{1+1*}\right)\right].
\end{align}
and the moment map components satisfying $\langle\mu_i(x),\xi\rangle=\imath_{\xi^*}\lambda_i$ agree with the combination $\mu_2+\sqrt{-1}\mu_3=\mu_{\mathbb C}$ we gave previously in equation \eqref{muC}.

\section{Flags of framed torsion-free sheaves on $\mathbb P^2$}\label{sec2}
We give in this paragraph the construction of the moduli space of flags of framed torsion-free sheaves of rank $r$ on the complex projective plane. We also show that there exists a natural isomorphism between the moduli space of flags of framed torsion-free sheaves on $\mathbb P^2$ and the stable representations of the nested instantons quiver. In the rank $r=1$ case our definition reduces to the nested Hilbert scheme of points on $\mathbb C^2$, as it is to be expected. By this reason we first want to carry out the analysis of the simpler $r=1$ case, which also has the advantage of providing us with a new characterization of punctual nested Hilbert schemes on $\mathbb C^2$, analogous to that of \cite{bulois_evain}.

\subsection{$\Hilb^{\hat{\bold n}}(\mathbb C^2)$ and $\mathcal N(1,\bold n)$}\label{sec:nhs_iso}
Before delving into the analysis of the relation between nested instantons moduli spaces and flags of framed torsion-free sheaves on $\mathbb P^2$, we want to show a special simpler case. In particular we will prove the existence of an isomorphism between the nested Hilbert scheme of points in $\mathbb C^2$ and the nested instantons moduli space $\mathcal N(1,n_0,\dots,n_N)$. This effectively gives us the ADHM construction of a general nested punctual nested Hilbert scheme on $\mathbb C^2$, which will also be the local model for more general nested Hilbert schemes of points on, say, toric surfaces $S$. In order to see this, we first recall the definition of a nested Hilbert scheme of points.
\begin{definition}
Let $S$ be a complex (projective) surface and $n_1\ge n_2\ge\cdots\ge n_k$ a sequence of integers. The nested Hilbert scheme of points on $S$ is defined as
\begin{equation}
\Hilb^{(n_1,\dots,n_k)}(S)=S^{[n_1,\dots,n_k]}=\left\{I_1\subseteq I_2\subseteq\cdots\subseteq I_k\subseteq\mathcal O_S:{\rm length}\left(\mathcal O_S/I_i\right)=n_i\right\}.
\end{equation}
Alternatively, if $X$ is a quasi-projective scheme over the complex numbers, we can equivalently define the nested Hilbert scheme $X^{[n_1,\dots,n_k]}=\Hilb^{(n_1,\dots,n_k)}(X)$ as
\begin{equation}
\Hilb^{(n_1,\dots,n_k)}(X)=\left\{(Z_1,\dots,Z_k):Z_i\in\Hilb^{n_i}(X),\ Z_i\ {\rm is\ a\ subscheme\ of }\ Z_j\ {\rm if }\ i<j\right\},
\end{equation}
with $Z_i$ being a zero-dimensional scheme, for every $i=1,\dots,k$.
\end{definition}

Before actually exhibiting the isomorphism we are interested in, we want to prove an auxiliary result, which gives an alternative definition for the nested Hilbert schemes over the affine plane, analogously to the case of Hilbert schemes studied in \cite{nakajima_lectures}.

\begin{proposition}\label{prop:adhm_1}
Let $\mathbbm k$ be an algebraically closed field, and $\bold n$ a sequence of integers $n_0\ge n_1\ge\cdots\ge n_k$. Define $\hat{\bold n}$ to be the sequence of integers $\hat{n}_0=n_0\ge\hat{n}_1=n_0-n_k\ge\cdots\ge\hat{n}_k=n_0-n_1$, then there exists an isomorphism
\begin{equation}
\Hilb^{\hat{\bold n}}(\mathbb A^2)\simeq\left\{(b_1^0,b_2^0,i,b_1^1,b_2^1,f_1,\dots,b_1^k,b_2^k,f_k)\left|
\begin{aligned}
{\rm (i)}&\ [b_1^i,b_2^i]=0\\
{\rm (ii)}&\ b_{1,2}^{i-1}f_i-f_ib_{1,2}^{i}=0\\
{\rm (iii)}&\ \nexists S\subset\mathbbm k^{n_0}:b^0_{1,2}(S)\subset S\ {\rm and}\\
&\qquad\qquad\quad \im(i)\subset S\\
{\rm (iv)}&\ f_i:\mathbbm k^{n_i}\to\mathbbm k^{n_{i-1}}\ {\rm is\ injective}
\end{aligned}
\right.\right\}\bigg/\mathcal G_{\bold n},
\end{equation}
where $\mathcal G_{\bold n}=GL_{n_0}(\mathbbm k)\times\cdots\times GL_{n_k}(\mathbbm k)$, $b_{1,2}^i\in\End(\mathbbm k^{n_i})$, $i\in\Hom(\mathbbm k,\mathbbm k^{n_0})$ and $f_i\in\Hom(\mathbbm k^{n_i},\mathbbm k^{n_{i-1}})$. The action of $\mathcal G_{\bold n}$ is given by
\begin{displaymath}
\bold g\cdot(b_1^0,b_2^0,i,\dots,b_1^k,b_2^k,f_k)=(g_0b_1^0g_0^{-1},g_0b_2^0g_0^{-1},g_0i,\dots,g_kb_1^kg_k^{-1},g_kb_2^kg_k^{-1},g_{k-1}f_kg_k^{-1}).
\end{displaymath}
\end{proposition}
\begin{proof}
Suppose we have a sequence of ideals $I_0\subseteq I_1\subseteq\cdots I_k\in\Hilb^{\hat{\bold n}}(\mathbb A^2)$. Let's first define $V_0=\mathbbm k[z_1,z_2]/I_0$, $b_{1,2}^0\in\End(V_0)$ to be the multiplication by $z_{1,2}\mod I_0$, and $i\in\Hom(\mathbbm k,V_0)$ by $i(1)=1\mod I_0$. Then obviously $[b_1^0,b_2^0]=0$ and condition (iii) holds since $1$ multiplied by products of $z_1$ and $z_2$ spans the whole $\mathbbm k[z_1,z_2]$. Then define $\tilde V_i=\mathbbm k[z_1,z_2]/I_i$ and, since $I_0\subseteq I_i$ for any $i>0$, complete $\tilde V_i$ to $V_0$ as $V_0=\tilde V_i\oplus V_i$, so that $V_i\simeq\mathbbm k^{n_i}$. The restrictions of $b_{1,2}^0$ to $V_i$ then yield homomorphisms $b_{1,2}^i\in\End(V_i)$ naturally satisfying $[b_1^i,b_2^i]=0$, while the inclusion of the ideals $I_0\subseteq I_1\subseteq\cdots\subseteq I_k$ implies the existence of an embedding $f_i:V_i\hookrightarrow V_{i-1}$ such that condition (ii) holds by construction.

Conversely, let $(b_1^0,b_2^0,i,\dots,b_1^k,b_2^k,f_k)$ be given as in the proposition. In the first place one can define a map $\phi_0:\mathbbm k[z_1,z_2]\to\mathbbm k^{n_0}$ to be $\phi_0(f)=f(b_1^0,b_2^0)i(1)$. This map is surjective, so that $I_0=\ker\phi$ is an ideal for $\mathbbm k[z_1,z_2]$ of length $n_0$. Then, since $f_i\in\Hom(\mathbbm k^{n_i},\mathbbm k^{n_{i-1}})$ is injective we can embed $\mathbbm k^{n_i}$ into $\mathbbm k^{n_0}$ though $F_i=f_1\circ\cdots\circ f_{i-1}\circ f_i$ in such a way that $b_{1,2}^i=b_{1,2}^{0}|_{\mathbbm k^{n_{i}}\hookrightarrow\mathbbm k^{n_0}}$, which is a simple consequence of condition (ii). Then we have the direct sum decomposition $\mathbbm k^{n_0}=\mathbbm k^{n_0-n_i}\oplus\mathbbm k^{n_i}$,the restrictions $\tilde b_{1,2}^i=b_{1,2}^0|_{\mathbbm k^{n_0-n_i}}$ and the projection $\tilde\imath_i=\pi_i\circ i$, with $\pi_i=\mathbbm k^{n_0}\to\mathbbm k^{n_0-n_i}$, satisfying $[\tilde b_1^i,\tilde b_2^i]=0$ and a stability condition analogous to (iii). Thus we define $\phi_i:\mathbbm k[z_1,z_2]\to\mathbbm k^{n_0-n_i}$ by $\phi_i(f)=f(\tilde b_1^i,\tilde b_2^i)\tilde\imath(1)$. This map is surjective, just like $\phi_0$, so that $I_j=\ker(\phi_j)$ is an ideal for $\mathbbm k(z_1,z_2)$ of length $n_0-n_i$. Finally, due to the successive embeddings $\mathbbm k^{n_k}\hookrightarrow\mathbbm k^{n_{k-1}}\hookrightarrow\cdots\hookrightarrow\mathbbm k^{n_0}$ we have the inclusion of the ideals $I_j\subset I_{j-1}$.
\end{proof}

One can readily notice that the description given by the previous proposition of the nested Hilbert scheme of points doesn't really coincide with the quiver we were studying throughout this section. However we can very easily overcome this problem by using the fact that if $(b_1^0,b_2^0,i,j)$ is a stable ADHM datum with $r=1$, then $j=0$, \cite{nakajima_lectures}. This proves the following proposition.
\begin{proposition}\label{prop:adhm_2}
With the same notations of proposition \ref{prop:adhm_1}, we have that
\begin{displaymath}
\Hilb^{\hat{\bold n}}(\mathbb A^2)\simeq\left\{(b_1^0,b_2^0,i,b_1^1,b_2^1,f_1,\dots,b_1^k,b_2^k,f_k)\left|
\begin{aligned}
{\rm (a)}&\ [b_1^0,b_2^0]+ij=0\\
{\rm (a')}&\ [b_1^i,b_2^i]=0,\ i>0\\
{\rm (b)}&\ b_{1,2}^{i-1}f_i-f_ib_{1,2}^{i}=0\\
{\rm (c)}&\ jf_1=0\\
{\rm (d)}&\ \nexists S\subset\mathbbm k^{n_0}:b^0_{1,2}(S)\subset S\ {\rm and}\\
&\qquad\qquad\quad \im(i)\subset S\\
{\rm (e)}&\ f_i:\mathbbm k^{n_i}\to\mathbbm k^{n_{i-1}}\ {\rm is\ injective}
\end{aligned}
\right.\right\}\bigg/\mathcal G_{\bold n}.
\end{displaymath}
\end{proposition}

All the previous observations, together with corollary \ref{corollary_stability}, immediately prove the following theorem.

\begin{theorem}\label{thm:iso_nhs}
The moduli space of nested instantons $\mathcal N_{r,\lambda,n,\mu}$ is isomorphic as a scheme to the nested Hilbert scheme of points on $\mathbb C^2$ when $r=1$ and $\lambda=[1^1]$.
\begin{equation}
\mathcal N(1,\bold n)=\mathbb X_0\git_\chi\mathcal G\simeq\Hilb^{\hat{\bold n}}(\mathbb C^2).
\end{equation}
\end{theorem}

\subsection{$\mathcal F(r,\boldsymbol\gamma)$ and $\mathcal N(r,\bold n)$}
A more general result relates the moduli space of flags of framed torsion-free sheaves on $\mathbb P^2$ to the moduli space of nested instantons. In the case of the two-step quiver this result was proved in \cite{flach_jardim}, here we give a generalization of their theorem in the case of the moduli space $\mathcal N_{r,[r^1],n,\mu}$ represented by a quiver with an arbitrary number of nodes.

\begin{definition}
Let $\ell_\infty\subset\mathbb P^2$ be a line and $F$ a coherent sheaf on $\mathbb P^2$. A framing $\phi$ for $F$ is then a choice of an isomorphism $\phi:F|_{\ell_\infty}\xrightarrow{\simeq}\mathcal O_{\ell_\infty}^{\oplus r}$, with $r=\rk F$. An $(N+2)-$tuple $(E_0,E_1,\dots,E_N,\phi)$ is a framed flag of sheaves on $\mathbb P^2$ if $E_0$ is a torsion-free (coherent) sheaf on $\mathbb P^2$ framed at $\ell_\infty$ by $\phi$, and $E_{j>0}$ form a flag of subsheaves $E_N\subseteq\cdots\subseteq E_0$ of $E_0$ s.t. the quotients $E_i/E_j$, $i<j$, are supported away from $\ell_\infty$.
\end{definition}

By the framing condition we get that $c_1(E_0)=0$, while the quotient condition on the subsheaves of $E_0$ naturally implies that the quotients $E_j/E_N$ are $0-$dimensional subsheaves and $c_1(E_{j>0})=0$. Then a framed flag of sheaves on $\mathbb P^2$ is characterized by the set of integers $(r,\boldsymbol\gamma)$, where $r=\rk E_0=\cdots=\rk E_N$, $c_2(E_0)=\gamma_0$, $h^0(E_0/E_j)=\gamma_1+\cdots+\gamma_j$ so that $c_2(E_{j>0})=\gamma_0+\cdots+\gamma_j$.

We now define the moduli functor
\begin{equation}
\mathsf F_{(r,\boldsymbol\gamma)}:\Sch_{\mathbb C}^{\rm op}\to\Sets,
\end{equation}
by assigning to a $\mathbb C-$scheme $S$ the set
$$\mathsf F_{(r,\boldsymbol\gamma)}(S)=\{{\rm isomorphism\ classes\ of\ }(2N+2)-{\rm tuples}\  (F_S,\varphi_S,Q^1_S,g^1_S,\dots,Q^N_S,g^N_S\}$$
with
\begin{itemize}
\item $F_S$ a coherent sheaf over $\mathbb P^2\times S$ flat over $S$ and such that $F_S|_{\mathbb P^2\times\{s\}}$ is a torsion-free sheaf for any closed point $s\in S$, $\rk F_S=r$, $c_1(F_S)=0$ and $c_2(F_S)=\gamma_0$;
\item $\varphi_S:F_S|_{\ell_\infty\times S}\to\mathcal O^{\oplus r}_{\ell_\infty\times S}$ is an isomorphism of $\mathcal O_{\ell_\infty\times S}-$modules;
\item $Q^i_S$ is a coherent sheaf on $\mathbb P^2\times S$, flat over $S$ and supported away from $\ell_\infty\times S$, such that $h^0(Q_S^i|_{\mathbb P^2\times\{s\}})=\gamma_1+\cdots+\gamma_i$, for any closed point $s\in S$;
\item $g^i_S:F_S\to Q^i_S$ is a surjective morphism of $\mathcal O_{\mathbb P^2\times S}-$modules.
\end{itemize}

Two tuples $(F_S,\varphi_S,Q_S^1,g_S^1,\dots,Q_S^N,g_S^N)$ and $(F_S',\varphi_S',Q_S^{1\prime},g_S^{1\prime},\dots,Q_S^{N\prime},g_S^{N\prime})$ are said to be isomorphic if there exist isomorphisms of $\mathcal O_{\mathbb P^2\times S}-$modules $\Theta_S:F_S\to F_S'$ and $\Gamma_S^i:Q_S^i\to Q_S^{i\prime}$ such that the following diagrams commute
\begin{equation}
\begin{tikzcd}
F_S|_{\ell_\infty\times S}\arrow{d}[left]{\Theta_S|_{\ell_\infty\times S}}\arrow[r,"\varphi_S"] & \mathcal O^{\oplus r}_{\ell_\infty\times S}\\
F_S'|_{\ell_\infty\times S}\arrow{ur}[right]{\varphi_S'}
\end{tikzcd}\qquad\qquad
\begin{tikzcd}
F_S\arrow[d,"\Theta_S"]\arrow[r,"g_S^i"] & Q_S^i\arrow[d,"\Gamma_S^i"]\\
F_S'\arrow[r,"g_S^{i\prime}"] & Q_S^{i\prime}
\end{tikzcd}
\end{equation}

If this functor is representable, the variety representing it will called the moduli space of flags of framed torsion-free sheaves on $\mathbb P^2$.

What we want to show next is that the moduli space of flags of torsion free sheaves on $\mathbb P^2$ is a fine moduli space, and that it is indeed isomorphic (as a scheme) to the moduli space of nested instantons we defined previously. First of all we will focus our attention on proving the following statement.
\begin{proposition}
The moduli functor $\mathsf F_{(r,\boldsymbol\gamma)}$ is represented by a (quasi-projective) variety $\mathcal F(r,\boldsymbol\gamma)$ isomorphic to a relative quot-scheme.
\end{proposition}

\begin{proof}
We base our proof on the concept of $\Quot$ functor, so let us recall its constuction and basic properties. First of all let us take the universal framed sheaf $(U^{(0)},\varphi_0)$ on $\mathbb P^2\times\mathcal M(r,\gamma_0)$, with $\varphi_0:U^{(0)}_{\ell_\infty\times\mathcal M(r,\gamma_0)}\xrightarrow{\simeq}\mathcal O^{\oplus r}_{\ell_\infty\times\mathcal M(r,\gamma_0)}$ an isomorphism of $\mathcal O_{\ell_\infty\times\mathcal M(r,\gamma_0)}-$modules. We then define
\begin{equation}
\Quot_{(U^{(0)},\gamma_1)}:\Sch^{\rm op}_{\mathcal M(r,\gamma_0)}\to\Sets
\end{equation}
by
\begin{equation}
\Quot_{(U^{(0)},\gamma_1)}(S)=\{{\rm isomorphism\ classes\ of\ }(Q_S,q_S)\}
\end{equation}
where
\begin{itemize}
\item $Q_S$ is a torsion-free sheaf on $\mathbb P^2\times S$, flat over $S$, supported away from $\ell_\infty\times S$ and such that $h^0(Q_S|_{\ell_\infty\times\{s\}})=\gamma_1$, for any $s\in S$ closed;
\item $q_S:U^{(0)}\to Q_S$, defined by $q_S:=(\mathbbm 1_{\mathbb P^2}\times\pi)^*$, where $\pi:S\to\mathcal M(r,\gamma_0)$, is a surjective morphism of $\mathcal O_{\mathbb P^2\times S}-$modules.
\end{itemize}
By Grothendieck theory this is a representable functor and it was proved in \cite{flach_jardim} to be isomorphic to the moduli functor of flags of couples of framed torsion-free sheaves on $\mathbb P^2$. In fact there exist a natural forgetting map $\mathsf F_{(r,\gamma_0,\gamma_1)}\to\Quot_{(U^{(0)},\gamma_1)}$ which act as $(F_S,\varphi_S,Q_S^1,g_S^1)\mapsto(Q_S^1,g_S^1)$. This map also has an inverse given by setting $F_S=\ker(g^1|_S)$, which has a framing $\varphi_S$ at $\ell_\infty\times S$ induced by the framing $\varphi_0$ of $U^{(0)}$ at $\ell_\infty\times\mathcal M(r,\gamma_0)$. The variety representing $\mathsf F_{(r,\gamma_0,\gamma_1)}$ is then the quot scheme $\quot^{\gamma_1}(U^{(0)})$ relative to $\mathcal M(r,\gamma_0)$. We can then construct a universal framed sheaf $(U^{(1)},\varphi_1)$ on $\mathbb P^2\times\mathcal F(r,\gamma_1,\gamma_2)$ with $\varphi_0:U^{(1)}_{\ell_\infty\times\mathcal F(r,\gamma_1,\gamma_2)}\xrightarrow{\simeq}\mathcal O^{\oplus r}_{\ell_\infty\times\mathcal F(r,\gamma_0,\gamma_1)}$ an isomorphism of $\mathcal O_{\ell_\infty\times\mathcal F(r,\gamma_0,\gamma_1)}-$modules. One can then use the quot functor
\begin{equation}
\Quot_{(U^{(1)},\gamma_2)}:\Sch^{\rm op}_{\mathcal F(r,\gamma_1,\gamma_2)}\to\Sets,
\end{equation}
in order to show that $\mathsf F_{(r,\gamma_0,\gamma_1,\gamma_2)}$ is isomorphic to $\Quot_{(U^{(1)},\gamma_2)}$, exactly in the same way as before and since the latter is representable so is the former. By iterating this procedure we can finally show that our moduli functor $\mathsf F_{(r,\boldsymbol\gamma)}$ is indeed representable, being isomorphic to a quot functor $\Quot_{(U^{(N-1)},\gamma_N)}$. The moduli space of flags of framed torsion-free sheaves on $\mathbb P^2$, $\mathcal F(r,\boldsymbol\gamma)$, is then a fine moduli space isomorphic to the relative quot-scheme $\quot^{\gamma_N}(U^{(N-1)})$.
\end{proof}

\begin{remark}
The previous description of the moduli space of framed flags of sheaves on $\mathbb P^2$ suggests we could also take a slightly different perspective on $\mathcal F(r,\boldsymbol\gamma)$, namely as the moduli of the sequence of quotients
\[
Z_N\hookrightarrow\cdots\hookrightarrow Z_1\hookrightarrow F\twoheadrightarrow Q_1\twoheadrightarrow\cdots\twoheadrightarrow Q_N,
\]
where $F$ is a vector bundle. In this sense $\mathcal F(r,\boldsymbol\gamma)$ seems to be analogous to the Filt-scheme studied by Mochizuki in \cite{mochizuki_filt} in the case of curves.
\end{remark}

Now that we proved that the definition of moduli space of framed flags of sheaves on $\mathbb P^2$ is indeed a good one we are ready to tackle the problem of showing that there exists an isomorphism between this moduli space and the space of stable representation of the nested instantons quiver we studied in the previous sections. First of all let us point out that our definition of flags of framed torsion-free sheaves reduce in the rank $1$ case to the nested Hilbert scheme of points on $\mathbb C^2$, and the isomorphism we are interested in was showed to exist in theorem \ref{thm:iso_nhs} of section \ref{sec:nhs_iso}. This is in fact compatible with the statement of the following theorem \ref{thm:iso_flags}.

\begin{theorem}\label{thm:iso_flags}
The moduli space of stable representations of the nested ADHM quiver is a fine moduli space isomorphic to the moduli space of flags of framed torsion-free sheaves on $\mathbb P^2$: $\mathcal F(r,\boldsymbol\gamma)\simeq\mathcal N(r,\bold n)$, as schemes, where $n_i=\gamma_i+\cdots+\gamma_N$.
\end{theorem}

\begin{proof}
We first want to show how, starting from an element of $\mathcal N(r,n_0,\dots,n_N)$ one can construct a flag of framed torsion-free sheaves on $\mathbb P^2$. As we showed previously, to each $(V_i,B_1^i,B_2^i,F^i)$ in the datum of $X\in\mathcal N(r,n_0,\dots,n_N)$ we can associate a stable ADHM datum $(W,\widetilde V_i,\widetilde B_1^i,\widetilde B_2^i,\tilde I^i,\tilde J^i)$, fitting in the diagram \eqref{diag_sheaves}
\begin{equation}\label{diag_sheaves}
\begin{tikzcd}[row sep=small, column sep=small]
V_1 \arrow[rr,"F^1"]\arrow[dr,shift left=.5ex,dashed] & & V_0 \arrow[rr]\arrow[dr,shift left=.5ex,dashed] & & \widetilde V_1\arrow[dr,shift left=.5ex,dashed] & \\
 & \{0\}\arrow[ul,shift left=.5ex,dashed]\arrow[rr,dashed]\arrow[dd,dashed] & & W\arrow[dd,dashed]\arrow[ul,shift left=.5ex,dashed]\arrow[rr,dashed] & & W\arrow[dd,dashed]\arrow[ul,shift left=.5ex,dashed] \\
V_2 \arrow[rr,crossing over]\arrow[dr,shift left=.5ex,dashed]\arrow[uu,"F^2"] & & V_0 \arrow[rr,crossing over]\arrow[from=uu,crossing over]\arrow[dr,shift left=.5ex,dashed] & & \widetilde V_2 \arrow[from=uu,crossing over]\arrow[dr,shift left=.5ex,dashed] & \\
 & \{0\}\arrow[ul,shift left=.5ex,dashed]\arrow[rr,dashed]\arrow[dd,dashed] & & W\arrow[dd,dashed]\arrow[ul,shift left=.5ex,dashed]\arrow[rr,dashed] & & W\arrow[dd,dashed]\arrow[ul,shift left=.5ex,dashed] \\
\vdots \arrow[rr,crossing over]\arrow[dr,shift left=.5ex,dashed,gray]\arrow[uu,"F^3"] & & \vdots \arrow[rr,crossing over]\arrow[from=uu,crossing over]\arrow[dr,shift left=.5ex,dashed,gray] & & \vdots \arrow[from=uu,crossing over]\arrow[dr,shift left=.5ex,dashed,gray] & \\
 & \textcolor{gray}{\{0\}}\arrow[ul,shift left=.5ex,dashed,gray]\arrow[rr,dashed,gray]\arrow[dd,dashed] & & \textcolor{gray}{W}\arrow[dd,dashed]\arrow[ul,shift left=.5ex,dashed,gray]\arrow[rr,dashed,gray] & & \textcolor{gray}{W}\arrow[dd,dashed]\arrow[ul,shift left=.5ex,dashed,gray] \\
V_N \arrow[rr,crossing over]\arrow[dr,shift left=.5ex,dashed]\arrow[uu,"F^N"] & & V_0 \arrow[rr,crossing over]\arrow[from=uu,crossing over]\arrow[dr,shift left=.5ex,dashed] & & \widetilde V_N \arrow[from=uu,crossing over]\arrow[dr,shift left=.5ex,dashed] & \\
 & \{0\}\arrow[ul,shift left=.5ex,dashed]\arrow[rr,dashed] & & W\arrow[ul,shift left=.5ex,dashed]\arrow[rr,dashed] & & W\arrow[ul,shift left=.5ex,dashed] \\
\end{tikzcd}
\end{equation}
where we suppressed all of the endomorphisms $B_{1,2}^i$, $\widetilde B_{1,2}^j$. We will then call $\bold Z_i$, $\bold S$ and $\bold Q_i$ the representations of the ADHM data $(\{0\},V_i,B_1^i,B_2^i)$, $(W,V_0,B_1^0,B_2^0,I,J)$ and $(W,\widetilde V_i,\widetilde B_1^i,\widetilde B_2^i,\tilde I^i,\tilde J^i)$, respectively. The the diagram \eqref{diag_sheaves} can be restated in the following form:
\begin{equation}\label{seq_reps}
\begin{tikzcd}
0\arrow[r]\arrow[d] & \bold Z_1\arrow[r] & \bold S\arrow[d]\arrow[r] & \bold Q_1\arrow[d]\arrow[r] & 0\arrow[d]\\
0\arrow[r]\arrow[d] & \bold Z_2\arrow[r]\arrow[u] & \bold S\arrow[d]\arrow[r] & \bold Q_2\arrow[d]\arrow[r] & 0\arrow[d]\\
\vdots\arrow[r]\arrow[d] & \vdots\arrow[r]\arrow[u] & \vdots\arrow[d]\arrow[r] & \vdots\arrow[d]\arrow[r] & \vdots\arrow[d]\\
0\arrow[r] & \bold Z_N\arrow[r]\arrow[u] & \bold S\arrow[r] & \bold Q_N\arrow[r] & 0\\
\end{tikzcd}
\end{equation}
Moreover, if $E^\bullet_{\bold Z_i}$, $E^\bullet_{\bold S}$ and $E^\bullet_{\bold Q_i}$ denotes the ADHM complex corresponding to $\bold Z_i$, $\bold S$ and $\bold Q_i$ the diagram \eqref{seq_reps} induces the following
\begin{equation}\label{seq_compl}
\begin{tikzcd}
0\arrow[r]\arrow[d] & E^\bullet_{\bold Z_1}\arrow[r] & E^\bullet_{\bold S}\arrow[d]\arrow[r] & E^\bullet_{\bold Q_1}\arrow[d]\arrow[r] & 0\arrow[d]\\
0\arrow[r]\arrow[d] & E^\bullet_{\bold Z_2}\arrow[r]\arrow[u] & E^\bullet_{\bold S}\arrow[d]\arrow[r] & E^\bullet_{\bold Q_2}\arrow[d]\arrow[r] & 0\arrow[d]\\
\vdots\arrow[r]\arrow[d] & \vdots\arrow[r]\arrow[u] & \vdots\arrow[d]\arrow[r] & \vdots\arrow[d]\arrow[r] & \vdots\arrow[d]\\
0\arrow[r] & E^\bullet_{\bold Z_N}\arrow[r]\arrow[u] & E^\bullet_{\bold S}\arrow[r] & E^\bullet_{\bold Q_N}\arrow[r] & 0\\
\end{tikzcd}
\end{equation}
Then, since $\bold S$ and $\bold Q_i$ are stable one has that $H^p(E^\bullet_{\bold S})=H^p(E^\bullet_{\bold Q_i})=0$, for $p=-1,1$, so that for each line in \eqref{seq_compl} the long exact sequence for the cohomology associated to it reduces to:
\begin{equation}
0\longrightarrow H^0(E^\bullet_{\bold S})\longrightarrow H^0(E^\bullet_{\bold Q_i})\longrightarrow H^1(E^\bullet_{\bold Z_i})\longrightarrow 0,
\end{equation}
and by the ADHM construction $(H^0(E^\bullet_{\bold Q_i}),\varphi)$ is a rank $r$ framed torsion-free sheaf on $\mathbb P^2$, with framing $\varphi:H^0(E^\bullet_{\bold Q_i})|_{\ell_\infty}\xrightarrow{\simeq}W\otimes\mathcal O_{\ell_\infty}$. Moreover $H^0(E^\bullet_{\bold S})$ is a subsheaf of $H^0(E^\bullet_{\bold Q_i})$, and $H^1(E^\bullet_{\bold Z_i})$ is a quotient sheaf
$$
H^1(E^\bullet_{\bold Z_i})\simeq H^0(E^\bullet_{\bold Q_i})/H^0(E^\bullet_{\bold S}),
$$
which is $0-$dimensional and supported away from $\ell_\infty\subset\mathbb P^2$. Finally one can immediately see from \eqref{seq_compl} that $H^0(E^\bullet_{\bold Q_i})$ is a subsheaf of $H^0(E^\bullet_{\bold Q_{i+1}})$. One can moreover check that the numerical invariants classifying flags of sheaves do agree with the statement of the theorem.

Conversely let $(E_0,\dots,E_N,\varphi)$ be a flag of framed torsion-free sheaves on $\mathbb P^2$ such that $\rk E_j=r$, $c_2(E_0)=\gamma_0$, $h^0(E_0/E_{j>0})=\gamma_1+\cdots+\gamma_j$. By definition each $(E_j,\varphi)$ defines a stable ADHM datum $\bold Q_j=(\widetilde W_j,\widetilde V_j,\widetilde B_1^j,\widetilde B_2^j,\tilde I^j,\tilde J^j)$ (with the convention of calling $\bold S=\bold Q_N$), since it can be identified with a framed torsion-free sheaf on $\mathbb P^2$, with $\rk E_j=r$, $c_2(E_j)=\gamma_0+\cdots+\gamma_j$. Moreover we have the inclusion $E_0\hookrightarrow E_j$, which induces an epimorphism $\Psi_j:\bold S\to\bold Q_j$. In fact, we can construct vector spaces $V_0$, $\widetilde V_j$, $W$ and $\widetilde W_j$ as in \cite{nakajima_lectures}, so that
\begin{equation}
V_0\simeq H^0(E_N(-1)),\quad \widetilde V_j\simeq H^0(E_j(-1)),\quad W\simeq H^0\left(E_N|_{\ell_\infty}\right),\quad \widetilde W_j\simeq H^0\left(E_j|_{\ell_\infty}\right),
\end{equation}
and by the fact that the quotient sheaf $E_j/E_N$ is $0-$dimensional and supported away from $\ell_\infty$ we can construct an isomorphism
\begin{equation}
\Psi_{j,2}:H^0\left(E_N|_{\ell_\infty}\right)\xrightarrow{\simeq}H^0\left(E_j|_{\ell_\infty}\right).
\end{equation}
Finally we have the exact sequence
\begin{displaymath}
0\longrightarrow E_N\longrightarrow E_j\longrightarrow E_j/E_N\longrightarrow 0,
\end{displaymath}
which induces the following exact sequence of cohomology, thanks to the fact that $H^0(E_j(-1))=0$, being that $E_j$ is a framed torsion-free $\mu-$semistable sheaf with $c_1(E_j)=0$ (due to the standard ADHM construction), and $H^1(E_j/E_N(-1))=0$, since the quotient sheaf $E_j/E_N$ is $0-$dimensional,
\begin{equation}
0\longrightarrow H^0(E_j/E_N(-1))\longrightarrow H^1(E_N(-1))\xrightarrow{\Psi_{j,1}} H^1(E_j(-1))\longrightarrow 0.
\end{equation}
The morphism $\Psi_j=(\Psi_{j,1},\Psi_{j,2})$ is then an epimorphism, since both $\Psi_{j,1}$ and $\Psi_{j,2}$ are surjective. Taking into account the flag structure of the datum $(E_0,\dots,E_N,\varphi)$, the sequences
\begin{equation}
\begin{tikzcd}
0\arrow[r]\arrow[d] & \ker\Psi_{N-1}\arrow[r] & \bold S\arrow[d]\arrow[r] & \bold Q_1\arrow[d]\arrow[r] & 0\arrow[d]\\
0\arrow[r]\arrow[d] & \ker\Psi_{N-2}\arrow[r]\arrow[u] & \bold S\arrow[d]\arrow[r] & \bold Q_2\arrow[d]\arrow[r] & 0\arrow[d]\\
\vdots\arrow[r]\arrow[d] & \vdots\arrow[r]\arrow[u] & \vdots\arrow[d]\arrow[r] & \vdots\arrow[d]\arrow[r] & \vdots\arrow[d]\\
0\arrow[r] & \ker\Psi_{0}\arrow[r]\arrow[u] & \bold S\arrow[r] & \bold Q_N\arrow[r] & 0\\
\end{tikzcd}
\end{equation}
give us $(N+1)$ stable ADHM data fitting in the following diagram.
\begin{equation}
\begin{tikzcd}[row sep=small, column sep=small]
V_1 \arrow[rr]\arrow[dr,shift left=.5ex,dashed] & & V_0 \arrow[rr,red,"\Psi_{N-1,1}" pos=.2]\arrow[dr,shift left=.5ex,dashed] & & \widetilde V_1\arrow[dr,shift left=.5ex,dashed] & \\
 & \{0\}\arrow[ul,shift left=.5ex,dashed]\arrow[rr,dashed]\arrow[dd,dashed] & & W\arrow[dd,dashed]\arrow[ul,shift left=.5ex,dashed]\arrow[rr,dashed,blue,"\Psi_{N-1,2}" pos=.2] & & W\arrow[dd,dashed]\arrow[ul,shift left=.5ex,dashed] \\
V_2 \arrow[rr,crossing over]\arrow[dr,shift left=.5ex,dashed]\arrow[uu] & & V_0 \arrow[rr,red,crossing over,"\Psi_{N-2,1}" pos=.2]\arrow[from=uu,crossing over]\arrow[dr,shift left=.5ex,dashed] & & \widetilde V_2 \arrow[from=uu,crossing over]\arrow[dr,shift left=.5ex,dashed] & \\
 & \{0\}\arrow[ul,shift left=.5ex,dashed]\arrow[rr,dashed]\arrow[dd,dashed] & & W\arrow[dd,dashed]\arrow[ul,shift left=.5ex,dashed]\arrow[rr,dashed,blue,"\Psi_{N-2,2}" pos=.2] & & W\arrow[dd,dashed]\arrow[ul,shift left=.5ex,dashed] \\
\vdots \arrow[rr,crossing over]\arrow[dr,shift left=.5ex,dashed,gray]\arrow[uu] & & \vdots \arrow[rr,crossing over]\arrow[from=uu,crossing over]\arrow[dr,shift left=.5ex,dashed,gray] & & \vdots \arrow[from=uu,crossing over]\arrow[dr,shift left=.5ex,dashed,gray] & \\
 & \textcolor{gray}{\{0\}}\arrow[ul,shift left=.5ex,dashed,gray]\arrow[rr,dashed,gray]\arrow[dd,dashed] & & \textcolor{gray}{W}\arrow[dd,dashed]\arrow[ul,shift left=.5ex,dashed,gray]\arrow[rr,dashed,gray] & & \textcolor{gray}{W}\arrow[dd,dashed]\arrow[ul,shift left=.5ex,dashed,gray] \\
V_N \arrow[rr,crossing over]\arrow[dr,shift left=.5ex,dashed]\arrow[uu] & & V_0 \arrow[rr,crossing over,red,"\Psi_{0,1}" pos=.2]\arrow[from=uu,crossing over]\arrow[dr,shift left=.5ex,dashed] & & \widetilde V_N \arrow[from=uu,crossing over]\arrow[dr,shift left=.5ex,dashed] & \\
 & \{0\}\arrow[ul,shift left=.5ex,dashed]\arrow[rr,dashed] & & W\arrow[ul,shift left=.5ex,dashed]\arrow[rr,dashed,blue,"\Psi_{0,2}" pos=.2] & & W\arrow[ul,shift left=.5ex,dashed] \\
\end{tikzcd}
\end{equation}
\end{proof}

\section{Virtual invariants}\label{sec3}
In this section we study fixed points under the action of a torus on the moduli space of framed stable representations of fixed numerical type of the nested instantons quiver. By doing this we are then able to apply virtual equivariant localization and compute certain relevant virtual topological invariants. On the physics side this is equivalent to the computation of certain partition functions of some suitable quiver GLSM theory by means of the SUSY localization technique.

\subsection{Equivariant torus action and localization}\label{sec:fixedpts}
We begin the analysis of the fixed locus under a certain toric action on the moduli space of nested instantons with a brief recall of the results obtained in \cite{flach_jardim} and show how they enable us to fully characterize the $T-$fixed locus of the two-step nested instantons quiver. The main result we want to recall is the following:
\begin{theorem}[von Flach-Jardim, \cite{flach_jardim}]
The moduli space $\mathcal N(r,n_0,n_1)\simeq\mathcal F(r,n_0-n_1,n_1)$ of stable representations of the nested ADHM quiver is a quasi-projective variety equipped with a perfect obstruction theory. Its $T-$equivariant deformation complex is the following
\begin{equation}
\begin{tikzcd}[row sep=-1mm]
& & Q\otimes\End(V_0)\\
& &\oplus & \Lambda^2Q\otimes\End(V_0)\\
& &\Hom(W,V_0) & \oplus\\
&\End(V_0) &\oplus & Q\otimes\Hom(V_1,V_0)\\
& \quad\oplus\quad \arrow[r,"d_0"] &\Lambda^2Q\otimes\Hom(V_0,W) \arrow[r,"d_1"] & \quad\qquad\oplus\qquad\quad \arrow[r,"d_2"] & \Lambda^2Q\otimes\Hom(V_1,V_0)\\
&\End(V_1) &\oplus & \Lambda^2Q\otimes\Hom(V_1,W)\\
& & Q\otimes\End(V_1) & \oplus\\
& &\oplus & \Lambda^2Q\otimes\End(V_1)\\
& &\Hom(V_1,V_0)
\end{tikzcd}
\end{equation}
with
\begin{displaymath}
\left\{
\begin{aligned}
&d_0(h_0,h_1)=\left([h_0,B_1^0],[h_0,B_2^0],h_0I,-Jh_0,[h_1,B_1^1],[h_1,B_2^1],h_0F-Fh_1\right) \\
&d_1(b_1^0,b_2^0,i,j,b_1^1,b_2^1,f)=\left([b_1^0,B_2^1]+[B_1^0,b_2^0]+iJ+Ij,B_1^0f+b_1^0F-Fb_1^1-fB_1^1,\right. \\
&\qquad\qquad\qquad\qquad\qquad\quad \left.B_2^0f+b_2^0F-Fb_2^1-fB_2^1,jF+Jf,[b_1^1,B_2^1]+[B_1^1,b_2^1]\right)\\
&d_2(c_1,c_2,c_3,c_4,c_5)=c_1F+B_2^0c_2-c_2B_2^1+c_3B_1^0-B_1^1c_3-Ic_4-Fc_5
\end{aligned}\right.
\end{displaymath}
Thus the infinitesimal deformation space and the obstruction space at any $X$ will be isomorphic to $H^1[\mathcal C(X)]$ and $H^2[\mathcal C(X)]$, respectively. $\mathcal N(r,n_1,n_2)$ is smooth iff $n_1=1$ (\cite{cheah}).
\end{theorem}

Moreover, it turns out, \cite{flach_jardim}, that there exists a surjective morphism $\mathfrak q:(B_1^0,B_2^0,I,J,B_1^1,B_2^1,F)\mapsto (B_1',B_2',I',J')$ mapping the nested ADHM data of type $(r,n_0,n_1)$ to the ADHM data of numerical type $(r,n_0-n_1)$. Thus we have two different maps sending the moduli space of stable representations of the nested ADHM quiver to the moduli space of stable representations of ADHM data. The situation is depicted by the following commutative diagram
\begin{figure}[H]
\centering
\begin{tikzcd}
\mathcal N(r,n_0,n_1)\arrow[r,"\eta"]\arrow[dr,"\mathfrak q"] & \mathcal M(r,n_0)\\
 & \mathcal M(r,n_0-n_1)\arrow[u,"\tilde f"]
\end{tikzcd}
\end{figure}
by means of which one can characterize $T-$fixed points of $\mathcal N(r,n_0,n_1)$ by means of fixed points of $\mathcal M(r,n_0)$ and $\mathcal M(r,n_0-n_1)$. In particular we can first take the decomposition $V_0=V\oplus V_1$, then decompose the vector spaces $V_0$, $V$ with respect to the action of $\bold T$: if $\lambda_0:T\to U(V_0)$ and $\lambda:T\to U(V)$ are morphisms for the toric action on $V_0$, $V$, we have
\begin{equation}
\left\{
\begin{aligned}
& V & =\quad & \bigoplus_{k,l} V(k,l) &=\quad & \bigoplus_{k,l}\left\{v\in V|\lambda(t)v=t_1^kt_2^lv\right\}\\
& V_0 & =\quad & \bigoplus_{k,l} V_0(k,l) &=\quad & \bigoplus_{k,l}\left\{v_0\in V_0|\lambda_0(t)v_0=t_1^kt_2^lv_0\right\}
\end{aligned}
\right.
\end{equation}
Thus if $X=(W,V,B_1',B_2',I',J')$, $X_0=(W,V_0,B_1^0,B_2^1,I,J)$ are fixed points for this torus action, the very well known results about the classification of fixed points for ADHM data leads us to the following commutative diagram.
\begin{equation}\label{comm_diag_1}
\begin{tikzcd}
& V(k-1,l) \arrow[from = dl] \arrow[rr,"B_2'"] \arrow[from=dd] & & V(k-1,l-1) \arrow[from = dl] \arrow[from=dd,"B_1'"]\\
V_0(k-1,l) \arrow[rr, crossing over] \arrow[from=dd,"B_1^0"] & & V_0(k-1,l-1) \\
& V(k,l) \arrow[from = dl] \arrow[rr] &  & V(k,l-1) \arrow[from = dl,"\tilde f"] \\
V_0(k,l) \arrow[rr,"B_2^0"] & & V_0(k,l-1) \arrow[uu,crossing over]\\
\end{tikzcd}
\end{equation}
\begin{proposition}
Let $X\in\mathbb X_0$ be a fixed point of the toric action. The following statements hold:
\begin{enumerate}
\item If $k>0$ or $l>0$, then $V_{0}(k,l)=0$, $V(k,l)=0$;
\item $\dim V_{0}(k,l)\le 1$, $\forall k,l$ and $\dim V(k,l)\le 1$, $\forall k,l$;
\item If $k,l\le 0$, then $\dim V_{0}(k,l)\ge\dim V_{0}(k-1,l)$, $\dim V_{0}(k,l)\ge\dim V_{0}(k,l-1)$, $\dim V(k,l)\ge\dim V(k-1,l$, $\dim V(k,l)\ge\dim V(k,l-1)$ and $\dim V_0(k,l)\ge\dim V(k,l)$.
\end{enumerate}
\end{proposition}

The previous propositions give us an easy way of visualizing fixed points of the $T-$action on the nested ADHM data. If we suitably normalize each non-zero map to $1$ by the action of $\prod_{k,l}GL(V_0(k,l))\times\prod_{k',l'}GL(V(k',l'))$ each critical point point can be put into one-to-one correspondence with nested Young diagrams $Y_\mu\subseteq Y_\nu$. Thus the fixed points of the original nested ADHM data are classified by couples $(\nu,\nu\setminus\mu)$, where $\mu\subset\nu$ and $\nu\setminus\mu$ is the skew Young diagram constructed by taking the complement of $\mu$ in $\nu$.

If we now take a fixed point $Z=(\nu,\mu)$ and define $\nu_i=\sum_k\dim V_0(k,1-i)$, $\nu'_j=\sum_l\dim V_0(1-j,l)$ and similarly $\mu_i=\sum_k\dim V(k,1-i)$, $\mu'_j=\sum_l\dim V(1-j,l)$, we can regard $V_0$ and $V$ as $T-$modules and write them as
\begin{equation}
\left\{
\begin{aligned}
&V_0=\bigoplus_{k,l}V_0(k,l)=\sum_{i=1}^{M_1}\sum_{j=1}^{\nu'_i}T_1^{-i+1}T_2^{-j+1}=\sum_{j=1}^{N_1}\sum_{j=1}^{\nu_j}T_1^{-i+1}T_2^{-j+1}\\
&V=\bigoplus_{k,l}V(k,l)=\sum_{i=1}^{M_2}\sum_{j=1}^{\mu'_i}T_1^{-i+1}T_2^{-j+1}=\sum_{j=1}^{N_2}\sum_{j=1}^{\mu_j}T_1^{-i+1}T_2^{-j+1}\\
\end{aligned}
\right.
\end{equation}
with $M_1=\nu_1$, $M_2=\mu_1$, $N_1=\nu'_1$, $N_2=\mu'_1$. If we now take $V_0=V\oplus V_1$, then
\begin{equation}
V_1=\sum_{(i,j)\in\nu\setminus\mu}T_1^{-i+1}T_2^{-j+1}=\sum_{i=1}^{M_1}\sum_{j=1}^{\nu'_i-\mu'_i}T_1^{-i+1}T_2^{-\mu_i'-j+1}
\end{equation}

The virtual tangent space $T_Z^{\textrm{vir}}\mathcal N(1,n_0,n_1)$ to $\mathcal N(1,n_0,n_1)$ at $Z$ can be regarded as a $\bold T^2-$module, so that
\begin{equation}
\begin{split}
T_Z^{\textrm{vir}}\mathcal N(1,n_0,n_1)&=\End(V_0)\otimes(Q-1-\Lambda^2Q)+\End(V_1)\otimes(Q-1-\Lambda^2Q)+\Hom(W,V_0)+\\
&\quad+\Hom(V_0,W)\otimes\Lambda^2Q-\Hom(V_1,W)\otimes\Lambda^2Q+\Hom(V_1,V_0)(1+\Lambda^2Q-Q)\\
&=(V_1\otimes V_0^*+V_1\otimes V_1^*-V_1^*\otimes V_0)\otimes(Q-1-\Lambda^2Q)+V_0+V_0^*\otimes\Lambda^2Q+\\
&\quad-V_2^*\otimes\Lambda^2Q.
\end{split}
\end{equation}
In the first place we might recognize the term $V_0^*\otimes V_0\otimes(Q-\Lambda^2Q-1)+V_0+V_0^*\otimes\Lambda^2Q$ in the sum as being the tangent space at the moduli space of stable representation of the ADHM quiver $T_{\tilde Z}\mathcal M(1,n_0)$, with $\tilde Z=(\nu)$. Thus we have
\begin{equation}\label{tan_rep_char}
\begin{split}
T_Z^{\textrm{vir}}\mathcal N(1,n_0,n_1)&=T_{\tilde Z}\mathcal M(1,n_0)+(V_1\otimes V_1^*-V_1^*\otimes V_0)\otimes(Q-1-\Lambda^2Q)-V_1^*\otimes\Lambda^2Q.
\end{split}
\end{equation}
We have
\begin{equation}
\begin{split}
V_1^*\otimes(Q-1-\Lambda^2Q)&=(T_1-1)(1-T_2)\sum_{i=1}^{M_1}\sum_{j=1}^{\nu_i'-\mu_i'}T_1^{i-1}T_2^{\mu_i'+j-1}\\
&=(T_1-1)\sum_{i=1}^{M_1}T_1^{i-1}T_2^{\mu_i'-1}(1-T_2)\sum_{j=1}^{\nu'_i-\mu'_i}T_2^{j}\\
&=(T_1-1)\sum_{i=1}^{M_1}T_1^{i-1}T_2^{\mu_i'-1}(1-T_2)\left(\frac{1-T_2^{\nu_i'-\mu'_i+1}}{1-T_2}-1\right)\\
&=(T_1-1)\sum_{i=1}^{M_1}T_1^{i-1}T_2^{\mu_i'-1}(1-T_2)\left(\frac{T_2-T_2^{\nu_i'-\mu'_i+1}}{1-T_2}\right)\\
&=(T_1-1)\sum_{i=1}^{M_1}T_1^{i-1}(T_2^{\mu_i'}-T_2^{\nu_i'}),
\end{split}
\end{equation}
so that
\begin{equation}
\begin{split}
V_1^*\otimes V_1\otimes(Q-1-\Lambda^2Q)&=(T_1-1)\sum_{j=1}^{N_1}\sum_{j'=1}^{\nu_{j}-\mu_{j}}T_1^{-\mu_j-j'+1}T_2^{-j+1}\sum_{i=1}^{M_1}T_1^{i-1}(T_2^{\mu_i'}-T_2^{\nu_i'})\\
&=\sum_{i=1}^{M_1}\sum_{j=1}^{N_1}T_1^{i-\mu_j}(T_2^{-j+\mu'_i+1}-T_2^{-j+\nu'_i+1})(T_1-1)\sum_{j'=1}^{\nu_j-\mu_j}T_1^{-j'}\\
&=\sum_{i=1}^{M_1}\sum_{j=1}^{N_1}T_1^{i-\mu_j}(T_2^{-j+\mu'_i+1}-T_2^{-j+\nu'_i+1})(T_1-1)\left(\frac{1-T_1^{-\nu_j+\mu_j-1}}{1-T_1^{-1}}-1\right)\\
&=\sum_{i=1}^{M_1}\sum_{j=1}^{N_1}(T_1^{i-\mu_j}-T_1^{i-\nu_j})(T_2^{-j+\mu'_i+1}-T_2^{-j+\nu'_i+1}),
\end{split}
\end{equation}
while we have
\begin{equation}
\begin{split}
V_1^*\otimes V_0\otimes(Q-1-\Lambda^2Q)&=(T_1-1)\sum_{j=1}^{N_1}\sum_{j'=1}^{\nu_{j}}T_1^{-j'+1}T_2^{-j+1}\sum_{i=1}^{M_1}T_1^{i-1}(T_2^{\mu_i'}-T_2^{\nu_i'})\\
&=\sum_{i=1}^{M_1}\sum_{j=1}^{N_1}T_1^i(T_2^{-j+\mu'_i+1}-T_2^{-j+\nu'_i+1})(T_1-1)\sum_{j'=1}^{\nu_j}T_1^{-j'}\\
&=\sum_{i=1}^{M_1}\sum_{j=1}^{N_1}T_1^i(T_2^{-j+\mu'_i+1}-T_2^{-j+\nu'_i+1})(T_1-1)\left(\frac{1-T_1^{-\nu_j-1}}{1-T_1^{-1}}-1\right)\\
&=\sum_{i=1}^{M_1}\sum_{j=1}^{N_1}(T_1^i-T_1^{i-\nu_j})(T_2^{-j+\mu'_i+1}-T_2^{-j+\nu'_i+1}),
\end{split}
\end{equation}
and
\begin{equation}
\begin{split}
V_1^*\otimes\Lambda^2Q&=T_1T_2\sum_{i=1}^{M_1}\sum_{j=1}^{\nu_i'-\mu_i'}T_1^{i-1}T_2^{\mu_i'+j-1}\\
&=\sum_{i=1}^{M_1}\sum_{j=1}^{\nu_i'-\mu_i'}T_1^iT_2^{\mu_i'+j}.
\end{split}
\end{equation}
Putting everything together we finally get that
\begin{equation}\label{tan_rep}
\begin{split}
T_Z^{\textrm{vir}}\mathcal N(1,n_0,n_1)=&T_{\tilde Z}\mathcal M(1,n_0)+\sum_{i=1}^{M_1}\sum_{j=1}^{N_1}(T_1^{i-\mu_j}-T_1^{i})(T_2^{-j+\mu_i'+1}-T_2^{-j+\nu_i'+1})\\
&-\sum_{i=1}^{M_1}\sum_{j=1}^{\nu_i'-\mu_i'}T_1^iT_2^{j+\mu_i'}.
\end{split}
\end{equation}
As an immediate generalization of \eqref{tan_rep} we can easily see that
\begin{equation}
\begin{split}
T_Z^{\textrm{vir}}\mathcal N(r,n_0,n_1)=&T_{\tilde Z}\mathcal M(r,n_0)+\sum_{a,b=1}^r\sum_{i=1}^{M_1^{(a)}}\sum_{j=1}^{N_1^{(b)}}R_bR_a^{-1}\left(T_1^{i-\mu_j^{(b)}}-T_1^{i}\right)\left(T_2^{-j+\mu_i^{(a)'}+1}+\right.\\ & \left.-T_2^{-j+\nu_i^{(a)'}+1}\right)-\sum_{i=1}^{M_1^{(a)}}\sum_{j=1}^{\nu_i^{(a)'}-\mu_i^{(a)'}}T_1^iT_2^{j+\mu_i^{(a)'}},
\end{split}
\end{equation}
where $(T_1,T_2,R_a),a=1,\dots,r$ are the canonical generators of the representation ring of $T\curvearrowright \mathcal N(r,n_0,n_1)$.


\begin{remark}
It turns out that the character representation for the virtual tangent $T_Z^{\emph{vir}}\mathcal N$ can be computed by exploiting deformation theory techniques. These techniques may also be employed to compute the virtual fundamental class and ($\bold T-$character of) the virtual tangent bundle at fixed points of nested Hilbert schemes on surfaces, as it's done in \cite{2017arXiv170108899G}.


If in particular one takes $\left(\mathbb C^2\right)^{[N_0\ge N_1]}$ to be the nested Hilbert scheme of points on $\mathbb C^2=\spec(R)$, with $\mathbb C[x_0,x_1]$, by lifting the natural torus action on $\mathbb C^2$ to $\left(\mathbb C^2\right)^{[N_0\ge N_1]}$, it is proved in \cite{2017arXiv170108899G} that the $T-$fixed locus is isolated and given by the inclusion of monomial ideals $I_0\subseteq I_1$, which is equivalent to the assignment of couples of nested partitions $\mu\subseteq\nu$. Then the virtual tangent space at a fixed point is given by
\begin{displaymath}
T_{I_0\subseteq I_1}^{\emph{vir}}=-\chi(I_0,I_0)-\chi(I_1,I_1)+\chi(I_0,I_1)+\chi(R,R),
\end{displaymath}
with $\chi(-,-)=\sum_{i=0}^2(-1)^i\ext_R^i(-,-)$. Then the $T-$representation of $T^{\emph{vir}}_{I_0\subseteq I_1}$ can be explicitly written in terms of Laurent polynomials in the torus characters $t_1,t_2$ of $T$. Then in terms of the characters $\mathsf Z_0,\mathsf Z_1$ of the $T-$fixed $0-$dimensional subschemes $Z_1\subseteq Z_0\subset\mathbb C^2$ corresponding to $I_0\subseteq I_1$ one has (see eq. (29) in \cite{2017arXiv170108899G})
\begin{displaymath}
\Tr T^{\emph{vir}}_{I_0\subseteq I_1}=\mathsf Z_0+\frac{\overline{\mathsf Z}_1}{t_1t_2}+\left(\overline{\mathsf Z}_0\mathsf Z_1-\overline{\mathsf Z}_0\mathsf Z_0-\overline{\mathsf Z}_1\mathsf Z_1\right)\frac{(1-t_1)(1-t_2)}{t_1t_2}.
\end{displaymath}
If we now make the necessary identifications $t_i=T_i^{-1}$, $\mathsf Z_0=V_0$ and $\mathsf Z_1=V$ we can see that equation (29) of \cite{2017arXiv170108899G} exactly agrees with our prescription for the character representation \eqref{tan_rep_char} of the virtual tangent space $T^{\emph{vir}}_Z\mathcal N(1,n_0,n_1)$, with $n_0=N_0$ and $n_1=N_0-N_1$.
\end{remark}

We now move on studying the fixed locus of the more general nested instantons moduli space $\mathcal N(r,n_0,\dots,n_N)$. However, similarly to the previous case we first want to show that the moduli space of stable representations of the nested ADHM quiver is equivalently described by the datum of $(N+1)$ moduli spaces of framed torsion-free sheaves on $\mathbb P^2$, namely $\mathcal M(r,n_0),\mathcal M(r,n_0-n_1),\dots,\mathcal M(r,n_0-n_{s-1})$. In order to do this we want to know if it is possible to recover the structure of the nested ADHM quiver given a set of stable ADHM data. First of all we can notice that, as $F^i$ is injective $\forall i$, we have the sum decomposition $V_0=V_i\oplus\tilde V_i$, but also $V_i=V_{i+1}\oplus\hat V_{i+1}$, with $\hat{V}_{i+1}=V_i/\im F_i$, so that $V_0=V_i\oplus\hat V_i\oplus\tilde V_{i-1}$, thus $\tilde V_i=\hat V_i\oplus\tilde V_{i-1}$.

Let us first focus on the vector spaces $V_0$ and $V_1$. It can be shown as in \cite{bruzzo2011,flach_jardim} that once we fix a stable ADHM datum $(W,\tilde V_1,\tilde B_1^1,\tilde B_2^1,\tilde I^1,\tilde J^1)$ and the endomorphisms $B_1^1,B_2^1\in\End V_1$ it is always possible to reconstruct the stable ADHM datum $(W,V_0,B_1^0,B_2^0,I,J)$ as
\begin{equation}
B_1^0=\begin{pmatrix}
B_1^1 & B_1^{'1}\\
0 & \tilde B_1^1
\end{pmatrix},\qquad
B_2^0=\begin{pmatrix}
B_2^1 & B_2^{'1}\\
0 & \tilde B_2^1
\end{pmatrix},\qquad
I=\begin{pmatrix}
I^{'1}\\ \tilde I^1
\end{pmatrix},\qquad
J=\begin{pmatrix}
0 & \tilde J^1
\end{pmatrix}
\end{equation}
together with the morphism $F^1=\mathbbm 1_{V_1}$ such that $[B_1^1,B_2^1]=0$, $B_1^0F^1-F^1B_1^1=B_2^0F^1-F^1B_2^1=0$ and $JF^1=0$. The same can obviously be done for any of the stable ADHM data $(W,\tilde V_i,\tilde B_1^i,\tilde B_2^i,\tilde I^i,\tilde J^i)$ we constructed previously, and we would have
\begin{equation}
B_1^0=\begin{pmatrix}
B_1^i & B_1^{'i}\\
0 & \tilde B_1^i
\end{pmatrix},\qquad
B_2^0=\begin{pmatrix}
B_2^i & B_2^{'i}\\
0 & \tilde B_2^i
\end{pmatrix},\qquad
I=\begin{pmatrix}
I^{'i}\\ \tilde I^i
\end{pmatrix},\qquad
J=\begin{pmatrix}
0 & \tilde J^i
\end{pmatrix}
\end{equation}
together with the morphism $f^i=\mathbbm 1_{V_i}$ such that $[B_1^i,B_2^i]=0$, $B_1^0f^i-f^iA_i=B_2^0f^i-f^iB_2^i=0$ and $Jf^i=0$. If we now fix
\begin{equation}
F^i=\begin{pmatrix}
\mathbbm 1_{V_i}\\ 0
\end{pmatrix},\qquad F^i:V_{i}\to V_{i-1},
\end{equation}
which is clearly injective, then obviously $f^i=F^1F^2\cdots F^i$, where $F^j$ now stands for the linear extension to $V_0$, and $B_1^0f^i-f^iB_1^i=0$ (resp. $B_2^0f^i-f^iB_2^i=0$) is equivalent to $B_1^0F^1F^2\cdots F^{i-1}F^i-F^1F^2\cdots F^iB_1^i=B_1^{i-1}F^i-F^iB_1^i=0$ (resp. $B_2^{i-1}F^i-F^iB_2^i=0$), and $Jf^i=JF^1F^2\cdots F^i=0$. This construction makes it possible to us to classify the $T-$fixed locus of $\mathcal N(r,n_0,\dots,n_{s-1})$ in terms of the $T-$fixed loci of $\mathcal M(r,n_0)$ and $\{\mathcal M(r,n_0-n_i)\}_{i>0}$. In particular the $T-$fixed locus of $\mathcal M(r,k)$ is into $1-1$ correspondence with coloured partitions $\boldsymbol\mu=(\mu^1,\dots,\mu^r)\in\mathcal P^r$ such that $|\boldsymbol\mu|=|\mu^1|+\cdots+|\mu^r|=k$. This fact and the inclusion relations between the vector spaces $V_i$ prove the following
\begin{proposition}\label{fixed_pts_longtail}
The $T-$fixed locus of $\mathcal N(r,n_0,\dots,n_{s-1})$ can be described by $s-$tuples of nested coloured partitions $\boldsymbol\mu_1\subseteq\cdots\subseteq\boldsymbol\mu_{s-1}\subseteq\boldsymbol\mu_0$, with $|\boldsymbol\mu_0|=n_0$ and $|\boldsymbol\mu_{i>0}|=n_0-n_i$.
\end{proposition}

In the same way as we did in a previous section, we can read the virtual tangent space to $\mathcal N(r,n_0,\dots,n_{s-1})$ off the following equivariant lift of the complex \eqref{tang_obs_N}
\begin{equation}\label{tang_obs_equiv_N}
\begin{tikzcd}
\bigoplus_{i=0}^N\End(V_i)\arrow[d,"d_0"]\\
Q\otimes\End(V_0)\oplus\Hom(W,V_0)\oplus\Lambda^2Q\otimes\Hom(V_0,W)\oplus\left[\bigoplus_{i=1}^N\left(Q\otimes\End(V_i)\oplus\Hom(V_i,V_{i-1}\right)\right]\arrow[d,"d_1"]\\
\Lambda^2Q\otimes\left(\End(V_0)\oplus\Hom(V_1,W)\right)\oplus\left[\bigoplus_{i=1}^N\left(Q\otimes\Hom(V_i,V_{i-1})\oplus\Lambda^2Q\otimes\End(V_i)\right)\right]\arrow[d,"d_2"]\\
\bigoplus_{i=1}^N\Lambda^2Q\otimes\Hom(V_i,V_{i-1})
\end{tikzcd},
\end{equation}
which gives us \eqref{tan_rep_long}.
\begin{equation}\label{tan_rep_long}
\begin{split}
T_Z^{\textrm{vir}}\mathcal N(1,n_1,n_2)&=\End(V_0)\otimes(Q-1-\Lambda^2Q)+\Hom(W,V_0)+\Hom(V_0,W)\otimes\Lambda^2Q\\
&\quad+\End(V_1)\otimes(Q-1-\Lambda^2Q)-\Hom(V_1,W)\otimes\Lambda^2Q+\\
&\quad+\Hom(V_1,V_0)\otimes(1+\Lambda^2Q-Q)+\\
&\quad+\End(V_2)\otimes(Q-1-\Lambda^2Q)+\Hom(V_2,V_1)\otimes(1+\Lambda^2Q-1)+\\
&\quad\cdots\\
&\quad+\End(V_{s-1})\otimes(Q-1-\Lambda^2Q)+\Hom(V_{s-1},V_{s-2})\otimes(1+\Lambda^2Q-Q)
\end{split}
\end{equation}
By decomposing the vector spaces $V_i$ in terms of characters of the torus $T$ we can also rewrite the representation of \eqref{tan_rep_long} in $R(T)$ as \eqref{tan_rep_long_char}
\begin{equation}\label{tan_rep_long_char}
\begin{split}
T_Z^{\textrm{vir}}\mathcal N(r,\bold n)=&T_{\tilde Z}\mathcal M(r,n_0)+\sum_{a,b=1}^r\sum_{i=1}^{M_0^{(a)}}\sum_{j=1}^{N_0^{(b)}}R_bR_a^{-1}\left(T_1^{i-\mu_{1,j}^{(b)}}-T_1^{i}\right)\left(T_2^{-j+\mu_{1,i}^{(a)'}+1}+\right.\\ & \left.-T_2^{-j+\mu_{0,i}^{(a)'}+1}\right)-\sum_{i=1}^{M_0^{(a)}}\sum_{j=1}^{\mu_{0,i}^{(a)'}-\mu_{1,i}^{(a)'}}T_1^iT_2^{j+\mu_{1,i}^{(a)'}}+\\
&+\sum_{k=2}^{s-1}\left[\sum_{a,b=1}^r\sum_{i=1}^{M_0^{(a)}}\sum_{j=1}^{N_0^{(b)}}R_bR_a^{-1}\left(T_1^{i-\mu_{k,j}^{(b)}}-T_1^{i-\mu_{k-1,j}^{(b)}}\right)\left(T_2^{-j+\mu_{k,i}^{(a)'}+1}-T_2^{-j+\mu_{0,i}^{(a)'}+1}\right)\right]+\\
&+(s-1)(T_1T_2),
\end{split}
\end{equation}
where the fixed point $Z$ is to be identified with a choice of a sequence of coloured nested partitions $\boldsymbol\mu_1\subseteq\boldsymbol\mu_{N-1}\subseteq\cdots\subseteq\boldsymbol\mu_{s-1}\subseteq\boldsymbol\mu_0$ as in proposition \ref{fixed_pts_longtail}, $\tilde Z\leftrightarrow\boldsymbol\mu_0$ and the last term, namely $(s-1)(T_1T_2)$, has been added in order to take into account the over-counting in the relations $[B_1^{i},B_2^{i}]=0$ due to the commutator being automatically traceless.

\subsection{Virtual equivariant holomorphic Euler characteristic}\label{sec:euler}
The first virtual invariant we are going to study is the holomorphic virtual equivariant Euler characteristic of the moduli space of nested instantons. The fact that we can decompose the virtual tangent bundle as a direct sum of equivariant line bundles under the torus action we previously described greatly simplifies the computations.

In particular, given a scheme $X$ with a $1-$perfect obstruction theory $E^\bullet$, one can define a virtual structure sheaf $\mathcal O_X^{\rm vir}$. Moreover one can choose an explicit resolution of $E^\bullet$ as $[E^{-1}\to E^0]$ a complex of vector bundles. If $[E_0\to E_1]$ denotes the dual complex, then one can also define the virtual tangent bundle $T_X^{\rm vir}\in K^0(X)$ as the class $T_X^{\rm vir}=[E_0]-[E_1]$. With these definitions, the virtual Todd genus of $X$ is defined as $\td^{\rm vir}(X)=\td(T_X^{\rm vir})$, and if $X$ is proper, given any $V\in K^0(X)$ one defines the virtual holomorphic Euler characteristic as
\begin{equation}
\chi^{\rm vir}(X,V)=\chi(X,V\otimes O^{\rm vir}_X),
\end{equation}
and as a consequence of the virtual Riemann-Roch theorem \cite{fantechi2010} if $X$ is proper and $V\in K^0(X)$ the virtual holomorphic Euler characteristic admits an equivalent definition as
\begin{equation}\label{vir_hol_euler_V}
\chi^{\rm vir}(X,V)=\int_{[X]^{\rm vir}}\ch(V)\cdot\td(T_X^{\rm vir}),
\end{equation}
where $[X]^{\rm vir}$ is the virtual fundamental class of $X$, $[X]^{\rm vir}\in A_{\vd}(X)$ and $\vd$ denotes the virtual dimension of $X$, $\vd=\rk E_0-\rk E_1$. Clearly, if we are interested in $\chi^{\rm vir}(X)$ then the previous formula reduces to
\begin{equation}\label{vir_hol_euler}
\chi^{\rm vir}(X)=\int_{[X]^{\rm vir}}\td(T_X^{\rm vir}),
\end{equation}
whenever $X$ is proper.

Equations \eqref{vir_hol_euler_V} and \eqref{vir_hol_euler} can be made even more explicit. In fact if we take $n=\rk E_0$, $m=\rk E_1$, so that $\vd=n-m$, and define $x_1,\dots,x_n$ and $u_1,\dots,u_m$ to be respectively the Chern roots of $E_0$ and $E_1$, then \eqref{vir_hol_euler} becomes
\begin{equation}
\chi^{\rm vir}(X)=\int_{[X]^{\rm vir}}\prod_{i=1}^n\frac{x_i}{1-\eu^{-x_1}}\prod_{j=1}^m\frac{1-\eu^{-u_j}}{u_j},
\end{equation}
while for \eqref{vir_hol_euler_V} we have
\begin{equation}
\chi^{\rm vir}(X,V)=\int_{[X]^{\rm vir}}\left(\sum_{k=1}^r\eu^{v_k}\right)\prod_{i=1}^n\frac{x_i}{1-\eu^{-x_1}}\prod_{j=1}^m\frac{1-\eu^{-u_j}}{u_j},
\end{equation}
since we can consider $V\in K^0(X)$ to be a vector bundle on $X$ with Chern roots $v_1,\dots,v_r$.

Now, if we have a proper scheme $X$ equipped with an action of a torus $(\mathbb C^*)^N$ and an equivariant $1-$perfect obstruction theory we can apply virtual equivariant localization in order to compute virtual invariants of $X$. We will now briefly recall how virtual localization works. First of all, for any equivariant vector bundle $B$ over a proper scheme $Z$ with a $1-$perfect obstruction theory, which is moreover equipped with a trivial action of $(\mathbb C^*)^N$, we have the decomposition
\begin{equation}
B=\bigoplus_{\bold k\in\mathbb Z^N}B^\bold k,
\end{equation}
where $B^\bold k$ denotes the $(\mathbb C^*)^N-$eigenbundles on which the torus acts by $t_1^{k_1}\cdots t_N^{k_N}$. If we now give a set of variables $\varepsilon_1,\dots,\varepsilon_N$, we identify $B$ with $B=\sum_{\bold k}B^\bold k\eu^{k_1\varepsilon_1}\cdots\eu^{k_N\varepsilon_N}\in K^0(Z)[\![\varepsilon_1,\dots,\varepsilon_N]\!]$. One then defines $B^{\rm fix}=B^{\bold 0}$ and $B^{\rm mov}=\oplus_{\bold k\neq\bold 0}B^\bold k$. Then the Chern character $\ch:K^0(Z)\to A^*(Z)$ can be extended by $\mathbb Q(\!(\varepsilon_1,\dots,\varepsilon_N)\!)-$linearity to
$$\ch:K^0(Z)(\!(\varepsilon_1,\dots,\varepsilon_N)\!)\to A^*(Z)(\!(\varepsilon_1,\dots,\varepsilon_N)\!).$$
Since the Grothendieck group of equivariant vector bundles $K^0_{(\mathbb C^*)^N}(Z)$ is a subring of $K^0(Z)[\![\varepsilon_1,\dots,\varepsilon_N]\!]$, the restriction of the extension of $\ch$ to $K^0_{(\mathbb C^*)^N}(Z)$ is naturally identified with the equivariant Chern character. Finally if one denotes by $p_*^{\rm vir}$ the $\mathbb Q(\!(\varepsilon_1,\dots,\varepsilon_N)\!)-$linear extension of $\chi^{\rm vir}(Z,-):K^0(Z)\to\mathbb Z$, and $p_*$ is the equivariant pushforward to a point, one can prove as in \cite{fantechi2010} that
\begin{equation}
p_*^{\rm vir}(V)=p_*\left(\ch(V)\td(T_Z^{\rm vir})\cap[Z]^{\rm vir}\right),\qquad V\in K^0(Z)(\!(\varepsilon_1,\dots,\varepsilon_N)\!).
\end{equation}
Then, following \cite{1997graber}, if we have a global equivariant embedding of a scheme $X$ into a nonsingular scheme $Y$ with $(\mathbb C^*)^N$ action, we can identify the maximal $(\mathbb C^*)^N-$fixed closed subscheme $X^f$ of $X$ with the scheme-theoretic intersection $X^f=X\cap Y^f$, where $Y^f$ is the nonsingular set-theoretic fixed point locus. By decomposing $Y^f$ into irreducible components $Y^f=\bigcup_iY_i$ one can also define $X_i=X\cap Y_i$, which carry a perfect obstruction theory with virtual fundamental class $[X_i]^{\rm vir}$. In this way, if $\tilde V\in K^0_{(\mathbb C^*)^N}(X)$ is an equivariant lift of the vector bundle $V$, $\tilde V_i$ is its restriction to $X_i$ and $p_i:X_i\to{\rm pt}$ is the projection, one has that
\begin{equation}\label{eul_loc}
\chi^{\rm vir}(X,\tilde V;\varepsilon_1,\dots,\varepsilon_N)=\sum_i p_{i*}^{\rm vir}\left(\tilde V_i/\Lambda_{-1}(N_i^{\rm vir})^\vee\right)=\sum_i p_{i*}^{\rm vir}\left(\tilde V_i/\Lambda_{-1}(T_X^{\rm vir}|_{X_i}^{\rm mov})^\vee\right)
\end{equation}
belongs to $\mathbb Q[\![\varepsilon_1,\dots,\varepsilon_N]\!]$ and the virtual holomorphic Euler characteristic is $\chi^{\rm vir}(X,V)=\chi^{\rm vir}(X,\tilde V;\boldsymbol 0)$.

Computations are now made very easy by the fact that we represented the virtual tangent space to the $T=(\mathbb C^*)^2-$fixed points to the moduli space of nested instantons in the representation ring $R(T)$ of the torus $(\mathbb C^*)^2$. In this way $T_{X_i}^{\rm vir}$ is decomposed as a direct sum of line bundles which are moreover eigenbundles of the torus action. Then we can use the following properties
\begin{equation}
\ch(E\oplus F)=\ch E+\ch F,\quad \Lambda_{t}(E\oplus F)=\Lambda_{t}(E)\cdot\Lambda_{t}(F),\quad S_t(E\oplus F)=S_t(E)\cdot S_t(F)
\end{equation}
and equation \eqref{eul_loc} in order to compute the equivariant holomorphic Euler characteristic of the moduli space of nested instantons in terms of the fundamental characters $\mathfrak q_{1,2}$ of the torus $T$. These will be related to the equivariant parameters by $\mathfrak q_i=\eu^{\beta\varepsilon_i}$, with $\beta$ being a parameter having a very clear meaning in the physical framework modelling the moduli space of nested instantons as a low energy effective theory. In this framework is very easy to explicitly compute the virtual equivariant holomorphic Euler characteristic of the moduli space of nested instantons as we already described the $T-$fixed locus of $\mathcal N(r,n_0,\dots,n_N)$ as being $0-$dimensional and non-degenerate. As we saw in section \ref{sec:fixedpts} the fixed points of $\mathcal N(r,n_0,\dots,n_N)$ are completely described by $r-$tuples of nested coloured partitions $\boldsymbol\mu_1\subseteq\cdots\subseteq\boldsymbol\mu_N\subseteq\boldsymbol\mu_0$, with $\boldsymbol\mu_j\in\mathcal P^r$, in such a way that $|\boldsymbol\mu_0|=\sum_j|\boldsymbol\mu_0^j|=n_0$ and $|\boldsymbol\mu_0\setminus\boldsymbol\mu_{i>0}|=n_{i>0}$. In the simplest case of $r=1$ we get
\begin{equation}
\begin{split}
\chi^{\rm vir}(\mathcal N(1,n_0,\dots,n_N),\tilde V;\mathfrak q_1,\mathfrak q_2)=&\sum_{\substack{\mu_1\subseteq\cdots\subseteq\mu_0\\ |\mu_0\setminus\mu_j|=n_j}}\ddfrac{T_{\mu_0,\mu_1}(\mathfrak q_1,\mathfrak q_2)W_{\mu_0,\dots,\mu_N}(\mathfrak q_1,\mathfrak q_2)}{N_{\mu_0}(\mathfrak q_1,\mathfrak q_2)})\left.\left[\tilde V\right]\right|_{\mu_0,\dots,\mu_N},
\end{split}\label{uno}
\end{equation}
where $a(s)$ and $l(s)$ denote the arm length and the leg length of the box $s$ in the Young diagram $Y_{\mu}$ associated to $\mu$, respectively. We moreover defined
\begin{align}
N_{\mu_0}(\mathfrak q_1,\mathfrak q_2)&=\prod_{s\in Y_{\mu_0}}\left(1-\mathfrak q_1^{-l(s)-1}\mathfrak q_2^{a(s)}\right)\left(1-\mathfrak q_1^{l(s)}\mathfrak q_2^{-a(s)-1}\right),\\
T_{\mu_0,\mu_1}(\mathfrak q_1,\mathfrak q_2)&=\prod_{i=1}^{M_0}\prod_{j=1}^{\mu_{0,i}'-\mu_{1,i}'}\left(1-\mathfrak q_1^{-i}\mathfrak q_2^{-j-\mu_{1,i}'}\right),\\
W_{\mu_0,\dots,\mu_N}(\mathfrak q_1,\mathfrak q_2)&=\prod_{k=1}^{N}\prod_{i=1}^{M_0}\prod_{j=1}^{N_0}\ddfrac{\left(1-\mathfrak q_1^{\mu_{k,j}-i}\mathfrak q_2^{j-\mu_{0,i}'-1}\right)\left(1-\mathfrak q_1^{\mu_{k-1,j}-i}\mathfrak q_2^{j-\mu_{k,i}'-1}\right)}{\left(1-\mathfrak q_1^{\mu_{k,j}-i}\mathfrak q_2^{j-\mu_{k,i}'-1}\right)\left(1-\mathfrak q_1^{\mu_{k-1,j}-i}\mathfrak q_2^{j-\mu_{0,i}'-1}\right)}
\end{align}
A very interesting and surprising fact can be observed if we rearrange the expression the holomorphic virtual Euler characteristic of $\mathcal N(1,n_0,\dots,n_N)$. In fact if we perform the summation over the smaller partitions $\mu_1\subseteq\cdots\subseteq\mu_N$ and redefine $q=\mathfrak q_1^{-1}$, $t=\mathfrak q_2^{-1}$, we get
\begin{equation}
\chi^{\rm vir}(\mathcal N(1,n_0,\dots,n_N);\mathfrak q_1,\mathfrak q_2)=\sum_{\mu_0}\frac{P_{\mu_0}(q,t)}{N_{\mu_0}(q,t)}
\end{equation}
and the unexpected fact is that we think $P_{\mu_0}(q,t)$ to be a polynomial in $q,t$ except for a factor $(1-qt)^{-1}$.
\begin{conjecture}\label{conj:poly}
$P_{\mu_0}(q,t)$ is a a function of the form:
\begin{equation}\label{polynomials}
P_{\mu_0}(q,t)=\frac{Q_{\mu_0}(q,t)}{(1-qt)^N},
\end{equation}
with $Q_{\mu_0}(q,t)$ a polynomial in the $(q,t)-$variables.
\end{conjecture}
Sometimes the polynomials in \eqref{polynomials} can be given an interpretations in terms of some known symmetric polynomials. In fact, let us define the following generating function
\begin{equation}
    Z_{MD}(q,t;x_0,\dots,x_N)=\sum_{n_0\ge\dots\ge n_N}\chi^{\rm vir}(\mathcal N(1,\tilde n_0,\dots,\tilde n_N);q,t)\prod_{i=0}^N x_i^{m_i},
\end{equation}
where $m_i=n_i-n_{i+1}$ and the integers $\tilde n_i$ form a sequence obtained from $n_i$ by asking the integers $\tilde m_i=\tilde n_i-\tilde n_{i+1}$ to be ordered. By construction $Z_{MD}(q,t;x_0,\dots,x_N)\in\mathbb Q[q,t]\otimes_{\mathbb Z}\Lambda(\bold x)$, \ie it is a symmetric function in $\{x_i\}_{i=0}^N$ with coefficients in $\mathbb Q[q,t]$. By conjecture \ref{conj:poly} we have
\begin{equation}
\begin{split}
    Z_{MD}(q,t;x_0,\dots,x_N)&=\sum_{n_0\ge\cdots\ge n_N}\sum_{\mu_0\in\mathcal P(n_0)}\frac{Q_{\mu_0}(q,t)}{(1-qt)^NN_{\mu_0}(q,t)}\prod_{i=0}^Nx_i^{m_i}\\
    &=\sum_{\mu\in\mathcal P}\frac{Q_{\mu}(q,t)}{(1-qt)^NN_{\mu}(q,t)}m_\mu(\bold x).
\end{split}
\end{equation}

\begin{conjecture}
When $|\mu_0|=|\mu_N|+1=|\mu_{N-1}|+2=\cdots=|\mu_1|+N$ we have
\begin{equation}
\begin{split}
        Q_{\mu_0}(q,t)&=\left\langle h_{\mu_0}(\bold x),\widetilde{H}_{\mu_0}(\bold x;q,t)\right\rangle\\
        &=\left\langle h_{\mu_0}(\bold x),\sum_{\lambda,\nu\in\mathcal P(n_0)}\widetilde{K}_{\lambda,\mu_0}(q,t)K_{\mu_0,\nu}m_\nu(\bold x)\right\rangle\\
        &=\sum_{\substack{\lambda\in\mathcal P(n_0)\\ m_\lambda(\bold x)\neq 0}}\widetilde{K}_{\lambda,\mu_0}(q,t),
\end{split}
\end{equation}
where the Hall pairing $\langle-,-\rangle$ is such that $\langle h_\mu,m_\lambda\rangle=\delta_{\mu,\lambda}$ and $\widetilde H_{\mu}(\bold x;q,t)$, $\widetilde K_{\lambda,\mu}(q,t)$ are the modified Macdonald polynomials and the modified Kostka polynomials, respectively.
\end{conjecture}
We checked the previous conjectures up to $n_0=10$.

If instead $r>1$ we get a more complicated result, even though its structure is the same as we had previously
\begin{equation}\label{eul_char_r}
\begin{split}
\chi^{\rm vir}(\mathcal N(r,n_0,\dots,n_N),\tilde V;\mathfrak q_1,\mathfrak q_2,\{\mathfrak t_i\})=&\sum_{\substack{\boldsymbol\mu_1\subseteq\cdots\subseteq\boldsymbol\mu_0\\ |\boldsymbol\mu_0\setminus\boldsymbol\mu_j|=n_j}}\ddfrac{T_{\boldsymbol\mu_0,\boldsymbol\mu_1}^{(r)}(\mathfrak q_1,\mathfrak q_2)W_{\boldsymbol\mu_0,\dots,\boldsymbol\mu_N}^{(r)}(\mathfrak q_1,\mathfrak q_2}{N_{\boldsymbol\mu_0}^{(r)}(\mathfrak q_1,\mathfrak q_2)})\left.\left[\tilde V\right]\right|_{\boldsymbol\mu_0,\dots,\boldsymbol\mu_N},
\end{split}
\end{equation}
with
\begin{align}
N_{\boldsymbol\mu_0}^{(r)}(\mathfrak q_1,\mathfrak q_2)&=\prod_{a,b=1}^r\prod_{s\in Y_{\mu_0^{(a)}}}\left(1-\mathfrak t_{ab}\mathfrak q_1^{-l_a(s)-1}\mathfrak q_2^{a_b(s)}\right)\left(1-\mathfrak q_1^{l_a(s)}\mathfrak q_2^{-a_b(s)-1}\right),\\
T_{\boldsymbol\mu_0,\boldsymbol\mu_1}^{(r)}(\mathfrak q_1,\mathfrak q_2)&=\prod_{a,b}^r\prod_{i=1}^{M_0^{(a)}}\prod_{j=1}^{\mu_{0,i}^{(a)\prime}-\mu_{1,i}^{(a)\prime}}\left(1-\mathfrak t_{ab}\mathfrak q_1^{-i}\mathfrak q_2^{-j-\mu_{1,i}^{(a)\prime}}\right),\\
W_{\boldsymbol\mu_0,\dots,\boldsymbol\mu_N}^{(r)}(\mathfrak q_1,\mathfrak q_2)&=\prod_{k=1}^{N}\prod_{a,b}^r\prod_{i=1}^{M_0^{(a)}}\prod_{j=1}^{N_0^{(b)}}\ddfrac{\left(1-\mathfrak t_{ab}\mathfrak q_1^{\mu_{k,j}^{(b)}-i}\mathfrak q_2^{j-\mu_{0,i}^{(a)\prime}-1}\right)\left(1-\mathfrak t_{ab}\mathfrak q_1^{\mu_{k-1,j}^{(b)}-i}\mathfrak q_2^{j-\mu_{k,i}^{(a)\prime}-1}\right)}{\left(1-\mathfrak t_{ab}\mathfrak q_1^{\mu_{k,j}^{(b)}-i}\mathfrak q_2^{j-\mu_{k,i}^{(a)\prime}-1}\right)\left(1-\mathfrak t_{ab}\mathfrak q_1^{\mu_{k-1,j}^{(b)}-i}\mathfrak q_2^{j-\mu_{0,i}^{(a)\prime}-1}\right)}
\end{align}
where now $\mathfrak t_{ab}=\mathfrak t_a\mathfrak t_b^{-1}$ and $\{\mathfrak t_i\}$ are the fundamental characters of $(\mathbb C^*)^r$ in $\mathcal G=(\mathbb C^*)^r\times T$, and $a_b(s)$ denotes the arm length of the box $s$ with respect to the Young diagram $Y_{\mu^{(b)}}$ associated to the partition $\mu^{(b)}$ of $\boldsymbol\mu$ (with an analogous definition for the leg length).

\subsection{Virtual equivariant $\chi_{-y}-$genus}\label{sec:hirzebruch}
The first refinement of the equivariant holomorphic Euler characteristic we are going to study is the virtual equivariant $\chi_{-y}-$genus, as defined in \cite{fantechi2010}. In order to exhibit the definition of virtual $\chi_{-y}-$genus let us first recall that if $E$ is a rank $r$ vector bundle on $r$ one can define the antisymmetric product $\Lambda_tE$ and the symmetric one $S_tE$ as
$$
\Lambda_tE=\sum_{i=0}^r[\Lambda^iE]t^i\in K^0(X)[t],\qquad S_tE=\sum_{i\ge 0}[S^iE]t^i\in K^0(X)[\![t]\!],
$$
so that $1/\Lambda_tE=S_{-t}E$ in $K^0(X)[\![t]\!]$. We can then define the virtual cotangent bundle $\Omega_X^{\rm vir}=(T_X^{\rm vir})^\vee$ and the bundle of virtual $n-$forms $\Omega_X^{n,{\rm vir}}=\Lambda^n\Omega_X^{\rm vir}$. If then $X$ is a proper scheme equipped with a perfect obstruction theory of virtual dimension $d$, the virtual $\chi_{-y}-$genus of $X$ is defined by
\begin{equation}
\chi_{-y}^{\rm vir}(X)=\chi^{\rm vir}(X,\Lambda_{-y}\Omega_X^{\rm vir})=\sum_{i\ge 0}(-y)^i\chi^{\rm vir}(X,\Omega^{i,{\rm vir}}_X),
\end{equation}
while, if $V\in K^0(X)$, the virtual $\chi_{-y}-$genus of $X$ with values in $V$ is
\begin{equation}
\chi_{-y}^{\rm vir}(X,V)=\chi^{\rm vir}(X,V\otimes\Lambda_{-y}\Omega_X^{\rm vir})=\sum_{i\ge 0}(-y)^i\chi^{\rm vir}(X,V\otimes\Omega^{i,{\rm vir}}_X).
\end{equation}
Though in principle one would expect $\chi_{-y}^{\rm vir}(X,V)$ to be an element of $\mathbb Z[\![t]\!]$, it is in fact true that $\chi_{-y}^{\rm vir}(X,V)\in\mathbb Z[t]$, \cite{fantechi2010}.

By the form \eqref{vir_hol_euler} and \eqref{vir_hol_euler_V} of the holomorphic Euler characteristic it is easy to see that
\begin{align}
& \chi_{-y}^{\rm vir}(X)=\int_{[X]^{\rm vir}}\ch(\Lambda_{-y}T_X^{\rm vir})\cdot\td(T_X^{\rm vir})=\int_{[X]^{\rm vir}}\mathcal X_{-y}(X),\\
& \chi_{-y}^{\rm vir}(X,V)=\int_{[X]^{\rm vir}}\ch(\Lambda_{-y}T_X^{\rm vir})\cdot\ch(V)\cdot\td(T_X^{\rm vir})=\int_{[X]^{\rm vir}}\mathcal X_{-y}(X)\cdot\ch(V),
\end{align}
which, in terms of the Chern roots of $E_0$, $E_1$ and $V$ become
\begin{align}
& \chi_{-y}^{\rm vir}(X)=\int_{[X]^{\rm vir}}\prod_{i=1}^nx_i\frac{1-y\eu^{-x_i}}{1-\eu^{-x_i}}\prod_{j=1}^m\frac{1}{u_j}\frac{1-\eu^{-u_j}}{1-y\eu^{-u_j}},\\
& \chi_{-y}^{\rm vir}(X,V)=\int_{[X]^{\rm vir}}\left(\sum_{k=1}^r\eu^{v_k}\right)\prod_{i=1}^nx_i\frac{1-y\eu^{-x_i}}{1-\eu^{-x_i}}\prod_{j=1}^m\frac{1}{u_j}\frac{1-\eu^{-u_j}}{1-y\eu^{-u_j}}.
\end{align}
Finally one can define the virtual Euler number $e^{\rm vir}(X)$ and the virtual signature $\sigma^{\rm vir}(X)$ of $X$ as $e^{\rm vir}(X)=\chi_{-1}^{\rm vir}(X)$ and $\sigma^{\rm vir}(X)=\chi_1^{\rm vir}(X)$. Whenever $y=0$ one recovers the holomorphic virtual Euler characteristic instead.

By extending the definition of $\chi_{-y}-$genus to the equivariant case in the obvious way and by making use of the equivariant virtual localization technique, one gets
\begin{equation}
\chi_{-y}^{\rm vir}(X,\tilde V;\varepsilon_1,\dots,\varepsilon_N)=\sum_ip_{i*}^{\rm vir}\left(\tilde V_i\otimes\Lambda_{-y}(\Omega_X^{\rm vir}|_{X_i})/\Lambda_{-1}(N_i^{\rm vir})^\vee\right),
\end{equation}
whence $\chi_{-y}^{\rm vir}(X,V)=\chi_{-y}^{\rm vir}(X,\tilde V;0,\dots,0)$.

A simple computation in equivariant localization gives us the following result:
\begin{equation}
\begin{split}
\chi^{\rm vir}_{-y}(\mathcal N(1,n_0,\dots,n_N),\tilde V;\mathfrak q_1,\mathfrak q_2)=&\sum_{\substack{\mu_1\subseteq\cdots\subseteq\mu_0\\ |\mu_0\setminus\mu_j|=n_j}}\ddfrac{T_{\mu_0,\mu_1}^{-y}(\mathfrak q_1,\mathfrak q_2)W_{\mu_0,\dots,\mu_N}^{-y}(\mathfrak q_1,\mathfrak q_2)}{N_{\mu_0}^{-y}(\mathfrak q_1,\mathfrak q_2)}\left.\left[\tilde V\right]\right|_{\mu_0,\dots,\mu_N},
\end{split} \label{due}
\end{equation}
with
\begin{align}
N_{\mu_0}^{-y}(\mathfrak q_1,\mathfrak q_2)&=\prod_{s\in Y_{\mu_0}}\ddfrac{\left(1-\mathfrak q_1^{-l(s)-1}\mathfrak q_2^{a(s)}\right)\left(1-\mathfrak q_1^{l(s)}\mathfrak q_2^{-a(s)-1}\right)}{\left(1-y\mathfrak q_1^{-l(s)-1}\mathfrak q_2^{a(s)}\right)\left(1-y\mathfrak q_1^{l(s)}\mathfrak q_2^{-a(s)-1}\right)},\\
T_{\mu_0,\mu_1}^{-y}(\mathfrak q_1,\mathfrak q_2)&=\prod_{i=1}^{M_0}\prod_{j=1}^{\mu_{0,i}'-\mu_{1,i}'}\ddfrac{\left(1-\mathfrak q_1^{-i}\mathfrak q_2^{-j-\mu_{1,i}'}\right)}{\left(1-y\mathfrak q_1^{-i}\mathfrak q_2^{-j-\mu_{1,i}'}\right)},\\ \nonumber
W_{\mu_0,\dots,\mu_N}^{-y}(\mathfrak q_1,\mathfrak q_2)&=\prod_{k=1}^{N}\prod_{i=1}^{M_0}\prod_{j=1}^{N_0}\ddfrac{\left(1-\mathfrak q_1^{\mu_{k,j}-i}\mathfrak q_2^{j-\mu_{0,i}'-1}\right)\left(1-\mathfrak q_1^{\mu_{k-1,j}-i}\mathfrak q_2^{j-\mu_{k,i}'-1}\right)}{\left(1-y\mathfrak q_1^{\mu_{k,j}-i}\mathfrak q_2^{j-\mu_{0,i}'-1}\right)\left(1-y\mathfrak q_1^{\mu_{k-1,j}-i}\mathfrak q_2^{j-\mu_{k,i}'-1}\right)}\cdot\\
&\cdot\ddfrac{\left(1-y\mathfrak q_1^{\mu_{k,j}-i}\mathfrak q_2^{k-\mu_{k,i}'-1}\right)\left(1-y\mathfrak q_1^{\mu_{k-1,j}-i}\mathfrak q_2^{j-\mu_{0,i}'-1}\right)}{\left(1-\mathfrak q_1^{\mu_{k,j}-i}\mathfrak q_2^{k-\mu_{k,i}'-1}\right)\left(1-\mathfrak q_1^{\mu_{k-1,j}-i}\mathfrak q_2^{j-\mu_{0,i}'-1}\right)}.
\end{align}
The limit $y\to 0$ manifestly reverts to the case of the equivariant holomorphic Euler characteristic of the moduli space of nested instantons.

A similar result holds also for the general case $r>1$:
\begin{equation}\label{chi_y_r}
\begin{split}
\chi^{\rm vir}_{-y}(\mathcal N(r,n_0,\dots,n_N),\tilde V;\mathfrak q_1,\mathfrak q_2,\{\mathfrak t_i\})=&\sum_{\substack{\boldsymbol\mu_1\subseteq\cdots\subseteq\boldsymbol\mu_0\\ |\boldsymbol\mu_0\setminus\boldsymbol\mu_j|=n_j}}\ddfrac{T_{\boldsymbol\mu_0,\boldsymbol\mu_1}^{(r),y}(\mathfrak q_1,\mathfrak q_2)W_{\boldsymbol\mu_0,\dots,\boldsymbol\mu_N}^{(r),y}(\mathfrak q_1,\mathfrak q_2)}{N_{\boldsymbol\mu_0}^{(r),y}(\mathfrak q_1,\mathfrak q_2)}\left.\left[\tilde V\right]\right|_{\boldsymbol\mu_0,\dots,\boldsymbol\mu_N},
\end{split}
\end{equation}
with
\begin{align}
N_{\boldsymbol\mu_0}^{(r),y}(\mathfrak q_1,\mathfrak q_2)&=\prod_{a,b=1}^r\prod_{s\in Y_{\mu_0^{(a)}}}\ddfrac{\left(1-\mathfrak t_{ab}\mathfrak q_1^{-l_a(s)-1}\mathfrak q_2^{a_b(s)}\right)\left(1-\mathfrak t_{ab}\mathfrak q_1^{l_a(s)}\mathfrak q_2^{-a_b(s)-1}\right)}{\left(1-y\mathfrak t_{ab}\mathfrak q_1^{-l_a(s)-1}\mathfrak q_2^{a_b(s)}\right)\left(1-y\mathfrak t_{ab}\mathfrak q_1^{l_a(s)}\mathfrak q_2^{-a_b(s)-1}\right)},\\
T_{\boldsymbol\mu_0,\boldsymbol\mu_1}^{(r),y}(\mathfrak q_1,\mathfrak q_2)&=\prod_{a,b}^r\prod_{i=1}^{M_0^{(a)}}\prod_{j=1}^{\mu_{0,i}^{(a)\prime}-\mu_{1,i}^{(a)\prime}}\ddfrac{\left(1-\mathfrak t_{ab}\mathfrak q_1^{-i}\mathfrak q_2^{-j-\mu_{1,i}^{(a)\prime}}\right)}{\left(1-y\mathfrak t_{ab}\mathfrak q_1^{-i}\mathfrak q_2^{-j-\mu_{1,i}^{(a)\prime}}\right)},\\ \nonumber
W_{\boldsymbol\mu_0,\dots,\boldsymbol\mu_N}^{(r),y}(\mathfrak q_1,\mathfrak q_2)&=\prod_{k=1}^{N}\prod_{a,b}^r\prod_{i=1}^{M_0^{(a)}}\prod_{j=1}^{N_0^{(b)}}\ddfrac{\left(1-\mathfrak t_{ab}\mathfrak q_1^{\mu_{k,j}^{(b)}-i}\mathfrak q_2^{j-\mu_{0,i}^{(a)\prime}-1}\right)\left(1-\mathfrak t_{ab}\mathfrak q_1^{\mu_{k-1,j}^{(b)}-i}\mathfrak q_2^{j-\mu_{k,i}^{(a)\prime}-1}\right)}{\left(1-y\mathfrak t_{ab}\mathfrak q_1^{\mu_{k,j}^{(b)}-i}\mathfrak q_2^{j-\mu_{0,i}^{(a)\prime}-1}\right)\left(1-y\mathfrak t_{ab}\mathfrak q_1^{\mu_{k-1,j}^{(b)}-i}\mathfrak q_2^{j-\mu_{k,i}^{(a)\prime}-1}\right)}\cdot\\
&\cdot\ddfrac{\left(1-y\mathfrak t_{ab}\mathfrak q_1^{\mu_{k,j}^{(b)}-i}\mathfrak q_2^{k-\mu_{k,i}^{(a)\prime}-1}\right)\left(1-y\mathfrak t_{ab}\mathfrak q_1^{\mu_{k-1,j}^{(b)}-i}\mathfrak q_2^{j-\mu_{0,i}^{(a)\prime}-1}\right)}{\left(1-\mathfrak t_{ab}\mathfrak q_1^{\mu_{k,j}^{(b)}-i}\mathfrak q_2^{k-\mu_{k,i}^{(a)\prime}-1}\right)\left(1-\mathfrak t_{ab}\mathfrak q_1^{\mu_{k-1,j}^{(b)}-i}\mathfrak q_2^{j-\mu_{0,i}^{(a)\prime}-1}\right)},
\end{align}
with the same notations of the previous section.

\subsection*{Virtual Euler number and signature}\addcontentsline{toc}{subsubsection}{Euler number and signature}
As we already pointed out previously, two specifications of the value of $y$ in the $\chi_{-y}-$genus, namely $y=\pm 1$, give back two interesting topological invariants of a given nested instantons moduli space. Let us consider first the simpler case of rank $1$. We can easily see, in the case of the virtual Euler number, that taking the specialization $y=+1$ amounts to counting the number of nested partitions of a given size. Then, if we assemble everything in a single generating function we have
\begin{equation}
\begin{split}
M(q_1,\dots,q_N)&=\sum_{j=0}^\infty\sum_{n_0,\dots,n_j} e^{\rm vir}\left(\mathcal N(1,n_0,\dots,n_j)\right)q_0^{n_0}\cdots q_j^{n_j}=\sum_{j=0}^\infty\sum_{n_0,\dots,n_j}\#\left\{\mu_1\subseteq\cdots\subseteq\mu_j\subseteq\mu_0\right\}q_0^{n_0}\cdots q_j^{n_j},
\end{split}
\end{equation}
for which an explicit expression is not available to our knowledge. Of course if we focus our attention on smooth nested Hilbert schemes only (\ie $n_{i>0}=0$ or $n_1=1$, $n_{i>1=0}$), the generating function of the virtual Euler number of smooth nested Hilbert schemes is easily expressed in terms of standard generating functions of partitions:
\begin{equation}
\sum_{n\ge 0} e^{\rm vir}(\mathcal N(1,n,1))q^n=\prod_{k=0}^\infty\left(\frac{1}{1-q^k}\right)=(\phi(q))^{-1}=\sum_{n\ge 0}\chi(\mathcal M(1,n))q^n,
\end{equation}
which, in the case of higher rank becomes
\begin{equation}
\sum_{n\ge 0} e^{\rm vir}(\mathcal N(r,n,1))q^n=\prod_{k=0}^\infty\left(\frac{1}{1-q^k}\right)^r=(\phi(1))^{-r}=\sum_{n\ge 0}\chi(\mathcal M(r,n))q^n.
\end{equation}
We also notice that whenever $q_0=q_1=\cdots=q_N$, generating function of Euler numbers is actually accounting for the enumeration of plane partition, whose generating function is known to be the Macmahon function $\Phi(-q)$:
\begin{equation}\label{tre}
M(q,\dots,q)=\sum_{j=0}^\infty\sum_{n_0,\dots,n_j}e^{\rm vir}(\mathcal N(1,n_0,\dots,n_j))q^{n_0+\cdots+n_j}=\Phi(-q).
\end{equation}

The case of the virtual signature is much easier instead. By taking $y=-1$ we immediately see that $\sigma^{\rm vir}(\mathcal N(1,\bold n))=0$.

\subsection{Virtual equivariant elliptic genus}\label{sec:elliptic}
A further refinement of the virtual $\chi_{-y}-$genus is finally given by the virtual elliptic genus. In this case, if $F$ is any vector bundle over $X$, we define
\begin{equation}
\mathcal E(F)=\bigotimes_{n\ge 1}\left(\Lambda_{-yq^n}F^\vee\otimes\Lambda_{-y^{-1}q^n}F\otimes S_{q^n}(F\oplus F^\vee)\right)\in 1+q\cdot K^0(X)[y,y^{-1}][\![q]\!],
\end{equation}
so that the virtual elliptic genus $\Ell^{\rm vir}(X;y,q)$ of $X$ is defined by
\begin{equation}
\Ell^{\rm vir}(X;y,q)=y^{-d/2}\chi_{-y}^{\rm vir}(X,\mathcal E(T_X^{\rm vir}))\in\mathbb Q(\!(y^{1/2})\!)[\![q]\!],
\end{equation}
and also
\begin{equation}
\Ell^{\rm vir}(X,V;y,q)=y^{-d/2}\chi_{-y}^{\rm vir}(X,\mathcal E(T_X^{\rm vir})\otimes V),
\end{equation}
By using virtual Riemann-Roch again one can see that $\Ell^{\rm vir}(X;y,q)$ admits an integral form
\begin{align}
& \Ell^{\rm vir}(X;y,q)=\int_{[X]^{\rm vir}}\Ellform(T_X^{\rm vir};y,q),\\
& \Ell^{\rm vir}(X,V;y,q)=\int_{[X]^{\rm vir}}\Ellform(T_X^{\rm vir};y,q)\cdot\ch(V),
\end{align}
with
\begin{equation}
\Ellform(F;y,q)=y^{-\rk F/2}\ch(\Lambda_{-y}F^\vee)\cdot\ch(\mathcal E(F))\cdot\td(F)\in A^*(X)[y^{-1/2},y^{1/2}][\![q]\!].
\end{equation}
It is also interesting to study how the virtual elliptic genus is described in terms of the usual Chern roots $x_i$, $u_j$, $v_k$, as its formula involves the Jacobi theta function $\theta(z,\tau)$ defined as
\begin{equation}
\theta(z,\tau)=q^{1/8}\frac{y^{1/2}-y^{-1/2}}{\iu}\prod_{l=1}^\infty(1-q^l)(1-q^ly)(1-q^ly^{-1}),
\end{equation}
where $q=\eu^{2\pi\iu\tau}$ and $y=\eu^{2\pi\iu z}$. In fact if $F$ is any vector bundle over $X$ with Chern roots $\{f_i\}$, one can prove \cite{Borisov2000} that
\begin{equation}
\Ellform(F;z,\tau)=\prod_{i=1}^{\rk F}f_i\frac{\theta(f_i/2\pi\iu-z,\tau)}{\theta(f_i/2\pi\iu,\tau)},
\end{equation}
so that finally
\begin{align}
& \Ell^{\rm vir}(X;y,q)=\int_{[X]^{\rm vir}}\prod_{i=1}^{n}x_i\frac{\theta(x_i/2\pi\iu-z,\tau)}{\theta(x_i/2\pi\iu,\tau)}\prod_{j=1}^{m}\frac{1}{u_j}\frac{\theta(u_j/2\pi\iu,\tau)}{\theta(u_j/2\pi\iu-z,\tau)},\\
& \Ell^{\rm vir}(X,V;y,q)=\int_{[X]^{\rm vir}}\left(\sum_{k=1}^r\eu^{v_k}\right)\prod_{i=1}^{n}x_i\frac{\theta(x_i/2\pi\iu-z,\tau)}{\theta(x_i/2\pi\iu,\tau)}\prod_{j=1}^{m}\frac{1}{u_j}\frac{\theta(u_j/2\pi\iu,\tau)}{\theta(u_j/2\pi\iu-z,\tau)}.
\end{align}

Finally, by taking the same steps as in the previous paragraphs we can equivariantly extend the definition of the virtual elliptic genus, and by virtual localization find that
\begin{equation}
\Ell^{\rm vir}(X,\tilde V,z,\tau;\varepsilon_1,\dots,\varepsilon_N)=y^{-\vd/2}\sum_ip_{i*}^{\rm vir}\left(\tilde V_i\otimes\mathcal E(T_X^{\rm vir}\otimes\Lambda_{-y}(\Omega_X^{\rm vir}|_{X_i})/\Lambda_{-1}(N_i^{\rm vir})^\vee\right)
\end{equation}
and $\Ell^{\rm vir}(X,V)=\Ell^{\rm vir}(X,\tilde V;0,\dots,0)$. In particular we get in rank 1
\begin{equation}\label{quattro}
\begin{split}
\Ell^{\rm vir}(\mathcal N(1,n_0,\dots,n_N),\tilde V;\varepsilon,\varepsilon_2)=&\sum_{\substack{\mu_1\subseteq\cdots\subseteq\mu_0\\ |\mu_0\setminus\mu_j|=n_j}}\ddfrac{\mathcal T_{\mu_0,\mu_1}^{z,\tau}(\varepsilon_1,\varepsilon_2)\mathcal W_{\mu_0,\dots,\mu_N}^{z,\tau}(\varepsilon_1,\varepsilon_2)}{\mathcal N_{\mu_0}^{z,\tau}(\varepsilon_1,\varepsilon_2)})\left.\left[\tilde V\right]\right|_{\mu_0,\dots,\mu_N},
\end{split}
\end{equation}
with
\begin{align}
\mathcal N_{\mu_0}^{z,\tau}(\varepsilon_1,\varepsilon_2)&=\prod_{s\in Y_{\mu_0}}\left[\ddfrac{\theta(\epsilon_1(l(s)+1)-\epsilon_2a(s),\tau)}{\theta(\epsilon_1(l(s)+1)-\epsilon_2a(s)-z,\tau)}\cdot\right.\\
&\qquad\cdot\left.\ddfrac{\theta(-\epsilon_1l(s)+\epsilon_2(a(s)+1),\tau)}{\theta(-\epsilon_1l(s)+\epsilon_2(l(s)+1)-z,\tau)}\right],\\
\mathcal T_{\mu_0,\mu_1}^{z,\tau}(\varepsilon_1,\varepsilon_2)&=\prod_{i=1}^{M_0}\prod_{j=1}^{\mu_{0,i}^{\prime}-\mu_{1,i}^{\prime}}\ddfrac{\theta(\epsilon_1i+\epsilon_2(j+\mu_{1,i}')-z,\tau)}{\theta(\epsilon_1i+\epsilon_2(j+\mu_{1,i}'),\tau)},\\ \nonumber
\mathcal W_{\mu_0,\dots,\mu_N}^{z,\tau}(\varepsilon_1,\varepsilon_2)&=\prod_{k=1}^{N}\prod_{i=1}^{M_0}\prod_{j=1}^{N_0}\left[\ddfrac{\theta(\epsilon_1(i-\mu_{k,j})+\epsilon_2(1+\mu_{0,i}'-j),\tau)}{\theta(\epsilon_1(i-\mu_{k,j})+\epsilon_2(1+\mu_{0,i}'-j)-z,\tau)}\cdot\right.\\ \nonumber
&\qquad\qquad\qquad\cdot\ddfrac{\theta(\epsilon_1(i-\mu_{k-1,j})+\epsilon_2(1+\mu_{k,i}'-j),\tau)}{\theta(\epsilon_1(i-\mu_{k-1,j})+\epsilon_2(1+\mu_{k,i}'-j)-z,\tau)}\\ \nonumber
&\qquad\qquad\qquad\cdot\ddfrac{\theta(\epsilon_1(i-\mu_{k,j})+\epsilon_2(1+\mu_{k,i}'-j)-z,\tau)}{\theta(\epsilon_1(i-\mu_{k,j})+\epsilon_2(1+\mu_{k,i}'-j),\tau)}\\
&\qquad\qquad\qquad\left.\cdot\ddfrac{\theta(\epsilon_1(i-\mu_{k-1,j})+\epsilon_2(1+\mu_{0,i}'-j)-z,\tau)}{\theta(\epsilon_1(i-\mu_{k-1,j})+\epsilon_2(1+\mu_{0,i}'-j),\tau)}\right],
\end{align}
with $\epsilon_i=\varepsilon_i/2\pi\iu$.
One can easily see that the virtual elliptic genus we just computed is indeed a Jacobi form, and that its limit $\tau\to\iu\infty$ reproduces the $\chi_{-y}-$genus. Moreover by taking the limit $y\to 0$ in the $\chi_{-y}-$genus one can recover the virtual equivariant holomorphic Euler characteristic.

Finally, if we study the virtual equivariant elliptic genus in the more general case of rank $r\ge 1$, we get
\begin{equation}\label{ell_vir_r}
\begin{split}
\Ell^{\rm vir}(\mathcal N(r,n_0,\dots,n_N),\tilde V;\varepsilon,\varepsilon_2,\{a_i\})=&\sum_{\substack{\boldsymbol\mu_1\subseteq\cdots\subseteq\boldsymbol\mu_0\\ |\boldsymbol\mu_0\setminus\boldsymbol\mu_j|=n_j}}\ddfrac{\mathcal T_{\boldsymbol\mu_0,\boldsymbol\mu_1}^{z,\tau}(\varepsilon_1,\varepsilon_2)\mathcal W_{\boldsymbol\mu_0,\dots,\boldsymbol\mu_N}^{z,\tau}(\varepsilon_1,\varepsilon_2)}{\mathcal N_{\boldsymbol\mu_0}^{z,\tau}(\varepsilon_1,\varepsilon_2)})\left.\left[\tilde V\right]\right|_{\boldsymbol\mu_0,\dots,\boldsymbol\mu_N},
\end{split}
\end{equation}
with
\begin{align}
\mathcal N_{\boldsymbol\mu_0}^{z,\tau}(\varepsilon_1,\varepsilon_2)&=\prod_{a,b=1}^r\prod_{s\in Y_{\mu_0}}\left[\ddfrac{\theta(a_{ab}+\epsilon_1(l(s)+1)-\epsilon_2a(s),\tau)}{\theta(a_{ab}+\epsilon_1(l(s)+1)-\epsilon_2a(s)-z,\tau)}\cdot\right.\\ \nonumber
&\qquad\qquad\qquad\cdot\left.\ddfrac{\theta(a_{ab}+-\epsilon_1l(s)+\epsilon_2(a(s)+1),\tau)}{\theta(a_{ab}+-\epsilon_1l(s)+\epsilon_2(l(s)+1)-z,\tau)}\right],\\
\mathcal T_{\boldsymbol\mu_0,\boldsymbol\mu_1}^{z,\tau}(\varepsilon_1,\varepsilon_2)&=\prod_{a,b=1}^r\prod_{i=1}^{M_0^{(a)}}\prod_{j=1}^{\mu_{0,i}^{(a)\prime}-\mu_{1,i}^{(a)\prime}}\ddfrac{\theta(a_{ab}+\epsilon_1i+\epsilon_2(j+\mu_{1,i}^{(a)\prime}),\tau)}{\theta(a_{ab}+\epsilon_1i+\epsilon_2(j+\mu_{1,i}^{(a)\prime})-z,\tau)},\\ \nonumber
\mathcal W_{\boldsymbol\mu_0,\dots,\boldsymbol\mu_N}^{z,\tau}(\varepsilon_1,\varepsilon_2)&=\prod_{k=1}^{N}\prod_{a,b=1}^r\prod_{i=1}^{M_0^{(a)}}\prod_{j=1}^{N_0^{(b)}}\left[\ddfrac{\theta(a_{ab}+\epsilon_1(i-\mu_{k,j}^{(b)})+\epsilon_2(1+\mu_{0,i}^{(a)\prime}-j),\tau)}{\theta(a_{ab}+\epsilon_1(i-\mu_{k,j}^{(b)})+\epsilon_2(1+\mu_{0,i}^{(a)\prime}-j)-z,\tau)}\cdot\right.\\ \nonumber
&\qquad\qquad\qquad\qquad\cdot \ddfrac{\theta(a_{ab}+\epsilon_1(i-\mu_{k-1,j}^{(b)})+\epsilon_2(1+\mu_{k,i}-j)^{(a)\prime},\tau)}{\theta(a_{ab}+\epsilon_1(i-\mu_{k-1,j}^{(b)})+\epsilon_2(1+\mu_{k,i}^{(a)\prime}-j)-z,\tau)}\cdot\\ \nonumber
&\qquad\qquad\qquad\qquad\cdot\ddfrac{\theta(a_{ab}+\epsilon_1(i-\mu_{k,j}^{(b)})+\epsilon_2(1+\mu_{k,i}^{(a)\prime}-j)-z,\tau)}{\theta(a_{ab}+\epsilon_1(i-\mu_{k,j}^{(b)})+\epsilon_2(1+\mu_{k,i}^{(a)\prime}-j),\tau)}\cdot\\
&\qquad\qquad\qquad\qquad\left.\cdot\ddfrac{\theta(a_{ab}+\epsilon_1(i-\mu_{k-1,j}^{(b)})+\epsilon_2(1+\mu_{0,i}^{(a)\prime}-j)-z,\tau)}{\theta(a_{ab}+\epsilon_1(i-\mu_{k-1,j}^{(b)})+\epsilon_2(1+\mu_{0,i}^{(a)\prime}-j),\tau)}\right].
\end{align}

Notice that by knowing the equivariant virtual elliptic genus one is able to recover both the virtual equivariant holomorphic Euler characteristic and $\chi_{-y}-$ genus. In fact the limit $\tau\to\iu\infty$ of \eqref{ell_vir_r} recovers exactly the $\chi_{-y}-$genus found in \eqref{chi_y_r} and a successive limit $y\to 0$ gives us back the virtual equivariant holomorphic Euler characteristic \eqref{eul_char_r}.


\section{Toric surfaces}\label{sec4}
In this section we will generalize the results we got in the previous ones to the case of nested Hilbert schemes on toric surfaces, and in particular we will be interested in $\mathbb P^2$ and $\mathbb P^1\times\mathbb P^1$. This is because one might expect any complex genus of $\Hilb^{(\bold n)}(S)$ to depend only on the cobordism class of $S$, as it was the case for $\Hilb^n(S)$, \cite{goettsche_ellisgrund}, and the complex cobordism ring $\Omega=\Omega^U\otimes\mathbb Q$ with rational coefficients was showed by Milnor to be a polynomial algebra freely generated by the cobordism classes $[\mathbb P^n]$, $n>0$. Then in the case of complex projective surfaces any case can be reduced to $\mathbb P^2$ and $\mathbb P^1\times\mathbb P^1$ by the fact that $[S]=a[\mathbb P^2]+b[\mathbb P^1\times\mathbb P^1]$. The advantage given by having an ADHM-like construction for the nested punctual Hilbert scheme on the affine plane is that it provides us with the local model of the more general case of smooth projective surfaces. In particular, whenever $S$ is toric, one can construct it starting from its toric fan by appropriately gluing the affine patches (\eg fig. \ref{fig:fansP2} for $\mathbb P^2$ and \ref{fig:fansP1P1} for $\mathbb P^1\times\mathbb P^1$), and computation of topological invariants can still be easily carried out by means of equivariant (virtual) localization.
\begin{figure}[htb!]
\centering
\begin{subfigure}{.45\textwidth}
\begin{tikzpicture}
\draw[fill=red!20,draw opacity=0] (0,0)--(4,0)--(0,4)--(0,0);
\draw[fill=red!10,draw opacity=0] (0,0)--(0,4)--(3*cos{225},3*sin{225})--(0,0);
\draw[fill=red!30,draw opacity=0] (0,0)--(4,0)--(3*cos{225},3*sin{225})--(0,0);
\draw[->,line width=1.5] (0,0)--(2,0) node[above,pos=.9] {$e_1$};
\draw[->,line width=1.5] (0,0)--(0,2) node[right,pos=.9] {$e_2$};
\draw[->,line width=1.5] (0,0)--(1.5*cos{225},1.5*sin{225}) node[above,pos=.8,rotate=45] {$e_2-e_1$};
\draw[dashed] (2,0)--(4,0);
\draw[dashed] (0,2)--(0,4);
\draw[dashed] (1.5*cos{225},1.5*sin{225})--(3*cos{225},3*sin{225});
\end{tikzpicture}\caption{$\mathbb P^2$}\label{fig:fansP2}
\end{subfigure}
\begin{subfigure}{.45\textwidth}
\begin{tikzpicture}[scale=.8]
\draw[fill=red!30,draw opacity=0] (0,0)--(4,0)--(0,4)--(0,0)--(-4,0)--(0,-4)--(0,0);
\draw[fill=red!10,draw opacity=0] (0,0)--(0,4)--(-4,0)--(0,0)--(0,-4)--(4,0)--(0,0);
\draw[dashed] (2,0)--(4,0);
\draw[dashed] (-2,0)--(-4,0);
\draw[dashed] (0,2)--(0,4);
\draw[dashed] (0,-2)--(0,-4);
\draw[<->,line width=1.5] (-2,0)--(2,0) node[above,pos=1] {$e_1$} node[above,pos=0] {$-e_1$};
\draw[<->,line width=1.5] (0,-2)--(0,2) node[right,pos=1] {$e_2$} node[right,pos=0] {$-e_2$};
\end{tikzpicture}\caption{$\mathbb P^1\times\mathbb P^1$}\label{fig:fansP1P1}
\end{subfigure}\caption{Toric fans for $\mathbb P^2$ and $\mathbb P^1\times\mathbb P^1.$}\label{fig:fans}
\end{figure}
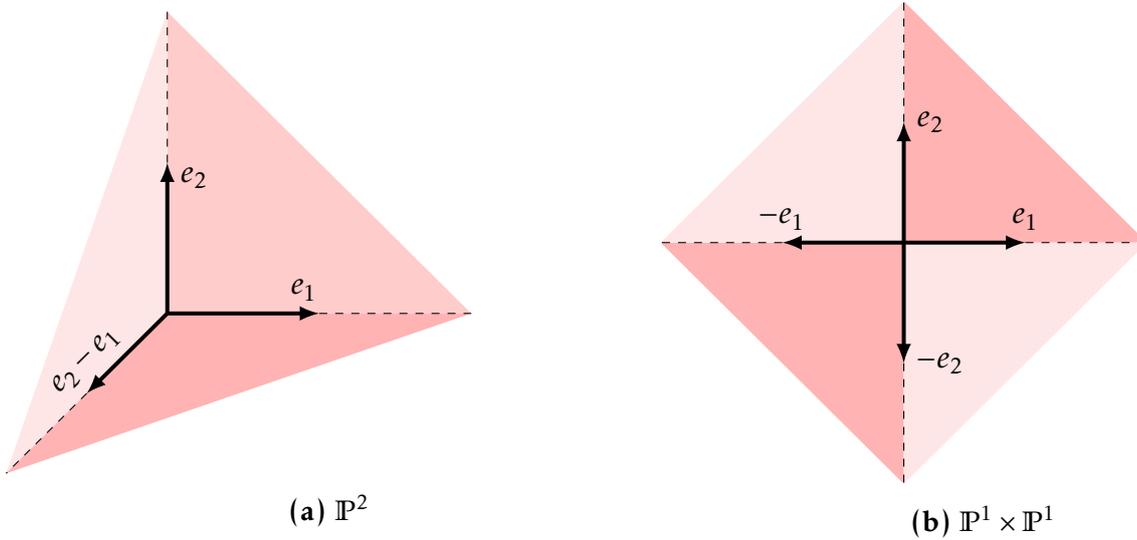
In general, given the toric fan describing the patches which glued together makes up a toric surface $S$, each patch $U_i$ will be $U_i\simeq\mathbb C^2$, with a natural action of $T=(\mathbb C^*)^2$. Moreover, if $S=\mathbb P^2$ or $S=\mathbb P^1\times\mathbb P^1$ and $Z\in\Hilb^{(\bold n)}(S)$ is a fixed point of the $T-$action, its support must be contained in $\{P_0,\dots,P_{\chi(S)-1}\}$ (as a consequence of \cite{ellisgrund_stromme}) with $P_i$ corresponding to the vertices of the polytope associated to the fan, so that one can write in general that $Z=Z_0\cup\cdots\cup Z_{\chi(S)-1}$, with $Z_i$ being supported in $P_i$. This also induces a decomposition of the representation in $R(T)$ of the virtual tangent space at the fixed points:
\begin{equation}
T^{\rm vir}_Z\left(\Hilb^{(\bold n)}(S)\right)=\bigoplus_{\ell=0}^{\chi(S)-1}T^{\rm vir}_{Z_\ell}\left(\Hilb^{(\bold n_\ell)}(U_\ell)\right).
\end{equation}
Let us call then $\chi_{-y}^{\rm vir}(P_{\ell})=\chi_{-y}^{\rm vir}(\mathcal N(1,n_0^{(\ell)},\dots,n_N^{(\ell)});\mathfrak q_{1,(\ell)}\mathfrak q_{2,(\ell)})$: we will see how we will be able to compute $\chi_{-y}^{\rm vir}(\Hilb^{(\bold n)}(\mathbb P^2))$ and $\chi_{-y}^{\rm vir}(\Hilb^{(\bold n)}(\mathbb P^1\times\mathbb P^1))$ in terms of $\chi_{-y}^{\rm vir}(P_\ell)$.
\subsection{Case 1: $S=\mathbb P^2$}
We will be interested in the following generating function
\begin{equation}\label{gen_chiy_P2}
\sum_{\bold n}\chi_{-y}^{\rm vir}\left(\Hilb^{(\hat{\bold n})}(\mathbb P^2)\right)\bold q^{\bold n}=\prod_{\ell=0}^2\left(\sum_{\bold n_\ell\ge\boldsymbol 0}\chi_{-y}^{\rm vir}(P_\ell)\bold q^{\bold n_\ell}\right),
\end{equation}
with $\hat{\bold n}$ defined as in section \ref{sec:nhs_iso}, and since the left-hand side doesn't depend on $\mathfrak q_{1,2}$, we can perform the computation by taking the iterated limits $\mathfrak q_1\to+\infty$, $\mathfrak q_2\to+\infty$ or $\mathfrak q_1\to 0x$, $\mathfrak q_2\to 0$. In each one of the three affine patches the weights of the torus action will be
\begin{equation}
\begin{aligned}
&\mathfrak q_{1,(0)}=\mathfrak q_1\\
&\mathfrak q_{1,(1)}=1/\mathfrak q_1\\
&\mathfrak q_{1,(2)}=1/\mathfrak q_2
\end{aligned}\qquad\qquad\qquad
\begin{aligned}
&\mathfrak q_{2,(0)}=\mathfrak q_2\\
&\mathfrak q_{2,(1)}=\mathfrak q_2/\mathfrak q_1\\
&\mathfrak q_{2,(2)}=\mathfrak q_1/\mathfrak q_2
\end{aligned}
\end{equation}

We will study separately the three patches $\ell=0,1,2$.
First of all we notice that since the $\chi_{-y}-$genus is multiplicative, the first contribution coming from $N_{\mu_0}^y(\mathfrak q_1,\mathfrak q_2)$ coincides with the same contribution arising in the context of standard Hilbert schemes. It wash shown in \cite{Liu2002HirzebruchXG} that
\begin{subequations}
\begin{align}
&\lim_{\mathfrak q_2\to+\infty}\lim_{\mathfrak q_1\to+\infty}\frac{1}{N_{\mu_0,\mu_1}^{-y}(\mathfrak q_{1,(0)},\mathfrak q_{2,(0)})}=y^{|\mu_0|-M_0},\\
&\lim_{\mathfrak q_2\to+\infty}\lim_{\mathfrak q_1\to+\infty}\frac{1}{N_{\mu_0,\mu_1}^{-y}(\mathfrak q_{1,(1)},\mathfrak q_{2,(1)})}=y^{|\mu_0|},\\
&\lim_{\mathfrak q_2\to+\infty}\lim_{\mathfrak q_1\to+\infty}\frac{1}{N_{\mu_0,\mu_1}^{-y}(\mathfrak q_{1,(2)},\mathfrak q_{2,(2)})}=y^{|\mu_0|+s(\mu_0)},\qquad s(\mu_0)=\#\{s\in Y_{\mu_0'}:a(s)\le l(s)\le a(s)+1\},
\end{align}
\end{subequations}
so that we just need to evaluate the other contributions. Starting from $T^{-y}_{\mu_0,\mu_1}$ we get
\begin{subequations}
\begin{align}
&\lim_{\mathfrak q_2\to+\infty}\lim_{\mathfrak q_1\to+\infty}\left[\prod_{i=1}^{M_0}\prod_{j=1}^{\mu_{0,i}'-\mu_{1,i}'}\ddfrac{\left(1-\mathfrak q_{1,(0)}^{-i}\mathfrak q_{2,(0)}^{-j-\mu_{1,i}'}\right)}{\left(1-y\mathfrak q_{1,(0)}^{-i}\mathfrak q_{2,(0)}^{-j-\mu_{1,i}'}\right)}\right]=1,\\
&\lim_{\mathfrak q_2\to+\infty}\lim_{\mathfrak q_1\to+\infty}\left[\prod_{i=1}^{M_0}\prod_{j=1}^{\mu_{0,i}'-\mu_{1,i}'}\ddfrac{\left(1-\mathfrak q_{1,(1)}^{-i}\mathfrak q_{2,(1)}^{-j-\mu_{1,i}'}\right)}{\left(1-y\mathfrak q_{1,(1)}^{-i}\mathfrak q_{2,(1)}^{-j-\mu_{1,i}'}\right)}\right]=y^{-1},\\
&\lim_{\mathfrak q_2\to+\infty}\lim_{\mathfrak q_1\to+\infty}\left[\prod_{i=1}^{M_0}\prod_{j=1}^{\mu_{0,i}'-\mu_{1,i}'}\ddfrac{\left(1-\mathfrak q_{1,(2)}^{-i}\mathfrak q_{2,(2)}^{-j-\mu_{1,i}'}\right)}{\left(1-y\mathfrak q_{1,(2)}^{-i}\mathfrak q_{2,(2)}^{-j-\mu_{1,i}'}\right)}\right]=1,
\end{align}
\end{subequations}
whence
\begin{subequations}
\begin{align}
&\lim_{\mathfrak q_2\to+\infty}\lim_{\mathfrak q_1\to+\infty}T_{\mu_0,\mu_1}^{-y}(\mathfrak q_{1,(0)},\mathfrak q_{2,(0)})=1,\\
&\lim_{\mathfrak q_2\to+\infty}\lim_{\mathfrak q_1\to+\infty}T_{\mu_0,\mu_1}^{-y}(\mathfrak q_{1,(1)},\mathfrak q_{2,(1)})=y^{-|\mu_0\setminus\mu_1|},\\
&\lim_{\mathfrak q_2\to+\infty}\lim_{\mathfrak q_1\to+\infty}T_{\mu_0,\mu_1}^{-y}(\mathfrak q_{1,(2)},\mathfrak q_{2,(2)})=1.
\end{align}
\end{subequations}
Finally we need to take care of the limit involving $W^{-y}_{\mu_0,\dots,\mu_N}(\mathfrak q_1,\mathfrak q_2)$ and in order to tackle let us first point out that we can rewrite $W^{-y}_{\mu_0,\dots,\mu_N}$ in the following simpler form:
\begin{equation}
\begin{split}
W_{\mu_0,\dots,\mu_N}^{-y}(\mathfrak q_1,\mathfrak q_2)=\prod_{k=1}^{N}\prod_{s\in Y_{\mu_0^{\rm rec}}}\ddfrac{\left(1-\mathfrak q_1^{l_k(s)}\mathfrak q_2^{-a_0(s)-1}\right)\left(1-\mathfrak q_1^{l_{k-1}(s)}\mathfrak q_2^{-a_k(s)-1}\right)}{\left(1-y\mathfrak q_1^{l_k(s)}\mathfrak q_2^{-a_0(s)-1}\right)\left(1-y\mathfrak q_1^{l_{k-1}(s)}\mathfrak q_2^{-a_k(s)-1}\right)}\cdot\\
\cdot\ddfrac{\left(1-y\mathfrak q_1^{l_k(s)}\mathfrak q_2^{-a_k(s)-1}\right)\left(1-y\mathfrak q_1^{l_{k-1}(s)}\mathfrak q_2^{-a_0(s)-1}\right)}{\left(1-\mathfrak q_1^{l_k(s)}\mathfrak q_2^{-a_k(s)-1}\right)\left(1-\mathfrak q_1^{l_{k-1}(s)}\mathfrak q_2^{-a_0(s)-1}\right)},
\end{split}
\end{equation}
where $\mu_0^{\rm rec}$ is the smallest rectangular partition containing $\mu_0$ and $a_k(s)$ (resp. $l_k(s)$) denotes the arm length (resp. leg length) of the box $s$ with respect to $Y_{\mu_k}$. Then, by recalling that the partitions labelling the $T-$fixed points are included one into the other as $\mu_1\subseteq\cdots\subseteq\mu_N\subseteq\mu_0\subseteq\mu_0^{\rm rec}$ it's easy to realize that, in the case $\ell=0$, one gets
\begin{equation}
\lim_{\mathfrak q_2\to+\infty}\lim_{\mathfrak q_1\to+\infty}\ddfrac{1-\mathfrak q_1^{l_k(s)}\mathfrak q_2^{-a_0(s)-1}}{1-y\mathfrak q_1^{l_k(s)}\mathfrak q_2^{-a_0(s)-1}}=\begin{cases}
1 &{\rm for}\ l_k(s)\le 0\\
y^{-1} &{\rm for}\ l_k(s)>0
\end{cases},
\end{equation}
and similarly in every other case:
\begin{subequations}
\begin{align*}
&\lim_{\mathfrak q_2\to+\infty}\lim_{\mathfrak q_1\to+\infty}\ddfrac{1-\mathfrak q_1^{l_{k-1}(s)}\mathfrak q_2^{-a_k(s)-1}}{1-y\mathfrak q_1^{l_{k-1}(s)}\mathfrak q_2^{-a_k(s)-1}}=\begin{cases}
1 &{\rm for}\ l_{k-1}(s)\le 0\\
y^{-1} &{\rm for}\ l_{k-1}(s)>0
\end{cases},\\
&\lim_{\mathfrak q_2\to+\infty}\lim_{\mathfrak q_1\to+\infty}\ddfrac{1-y\mathfrak q_1^{l_k(s)}\mathfrak q_2^{-a_k(s)-1}}{1-\mathfrak q_1^{l_k(s)}\mathfrak q_2^{-a_k(s)-1}}=\begin{cases}
1 &{\rm for}\ l_{k}(s)\le 0\\
y &{\rm for}\ l_k(s)>0
\end{cases},\\
&\lim_{\mathfrak q_2\to+\infty}\lim_{\mathfrak q_1\to+\infty}\ddfrac{1-y\mathfrak q_1^{l_{k-1}(s)}\mathfrak q_2^{-a_0(s)-1}}{1-\mathfrak q_1^{l_{k-1}(s)}\mathfrak q_2^{-a_0(s)-1}}=\begin{cases}
1 &{\rm for}\ l_{k-1}(s)\le 0\\
y &{\rm for}\ l_{k-1}(s)>0
\end{cases},
\end{align*}
\end{subequations}
so that finally
\begin{equation}
\lim_{\mathfrak q_2\to+\infty}\lim_{\mathfrak q_1\to+\infty}W^{-y}_{\mu_0,\dots,\mu_N}(\mathfrak q_{1,(0)},\mathfrak q_{2,(0)})=1.
\end{equation}
It is easy to see that the same holds true also for $\ell=2$:
\begin{equation}
\lim_{\mathfrak q_2\to+\infty}\lim_{\mathfrak q_1\to+\infty}W^{-y}_{\mu_0,\dots,\mu_N}(\mathfrak q_{1,(2)},\mathfrak q_{2,(2)})=1,
\end{equation}
while the case $\ell=1$ is more difficult, even though the analysis of the different cases can be carried out exactly in the same way. We then introduce the following notation:
\begin{equation}
s(\mu_{i_1},\mu_{i_2})=\#\left\{s\in Y_{\mu_0^{\rm rec}}:l_{i_1}(s)>a_{i_2}(s)+1\vee l_{i_1}(s)=a_{i_2}(s)+1,a_{i_2}(s)<-1\right\},
\end{equation}
and we get
\begin{equation}
\lim_{\mathfrak q_2\to+\infty}\lim_{\mathfrak q_1\to+\infty}W^{-y}_{\mu_0,\dots,\mu_N}(\mathfrak q_{1,(1)},\mathfrak q_{2,(1)})=\prod_{k=1}^Ny^{s(\mu_k,\mu_k)+s(\mu_{k-1},\mu_0)-s(\mu_k,\mu_0)-s(\mu_{k-1},\mu_k)}.
\end{equation}
Finally, by putting everything together, we have an explicit expression for \eqref{gen_chiy_P2}.
\begin{equation}\label{cinque}
\begin{split}
\sum_{\bold n}\chi_{-y}^{\rm vir}\left(\Hilb^{(\hat{\bold n})}(\mathbb P^2)\right)\bold q^{\bold n}=\left(\sum_{\bold n}\bold q^{\bold n}\sum_{\{\mu_i\}}y^{|\mu_0|+M_0}\right)\left(\sum_{\bold n}\bold q^{\bold n}\sum_{\{\mu_i\}}y^{|\mu_0|-|\mu_0\setminus\mu_1|}\prod_{k=1}^Ny^{s(\mu_k,\mu_k)+s(\mu_{k-1},\mu_0)}\right.\\
y^{-s(\mu_k,\mu_0)-s(\mu_{k-1},\mu_k)}\Bigg)\left(\sum_{\bold n}\bold q^{\bold n}\sum_{\{\mu_i\}}y^{|\mu_0|-s(\mu_0)}\right)
\end{split}
\end{equation}

\subsection{Case 2: $S=\mathbb P^1\times\mathbb P^1$}
Similarly to previous case, we are interested in studying the following generating function
\begin{equation}
\sum_{\bold n}\chi_{-y}^{\rm vir}\left(\Hilb^{(\hat{\bold n})}(\mathbb P^1\times\mathbb P^1)\right)\bold q^{\bold n}=\prod_{\ell=0}^3\left(\sum_{\bold n_\ell\ge\boldsymbol 0}\chi_{-y}^{\rm vir}(P_\ell)\bold q^{\bold n_\ell}\right),
\end{equation}
and we can still perform the computation by taking the successive limits $\mathfrak q_1\to+\infty$, $\mathfrak q_2\to+\infty$ or $\mathfrak q_1\to 0$, $\mathfrak q_2\to 0$. The four patches are now indexed by $\ell=(00),(01),(10),(11)$, and the characters $\mathfrak q_{i,(\ell)}$ can be identified to be in this case
\begin{equation}
\begin{aligned}
&\mathfrak q_{1,(00)}=\mathfrak q_1\\
&\mathfrak q_{1,(01)}=\mathfrak q_1\\
&\mathfrak q_{1,(10)}=1/\mathfrak q_1\\
&\mathfrak q_{1,(11)}=1/\mathfrak q_1\\
\end{aligned}\qquad\qquad\qquad
\begin{aligned}
&\mathfrak q_{2,(00)}=\mathfrak q_2\\
&\mathfrak q_{2,(01)}=1/\mathfrak q_2\\
&\mathfrak q_{2,(10)}=\mathfrak q_2\\
&\mathfrak q_{2,(11)}=1/\mathfrak q_2\\
\end{aligned}
\end{equation}

An analysis similar to the one carried out in the previous section enables then us to conclude the following
\begin{equation}
\begin{aligned}
&\lim_{\mathfrak q_2\to+\infty}\lim_{\mathfrak q_1\to+\infty}1/N^{-y}_{\mu_0}(\mathfrak q_{1,(00)},\mathfrak q_{2,(00)})=y^{|\mu_0|-M_0},\\
&\lim_{\mathfrak q_2\to+\infty}\lim_{\mathfrak q_1\to+\infty}1/N^{-y}_{\mu_0}(\mathfrak q_{1,(00)},\mathfrak q_{2,(01)})=y^{|\mu_0|},\\
&\lim_{\mathfrak q_2\to+\infty}\lim_{\mathfrak q_1\to+\infty}1/N^{-y}_{\mu_0}(\mathfrak q_{1,(00)},\mathfrak q_{2,(10)})=y^{|\mu_0|},\\
&\lim_{\mathfrak q_2\to+\infty}\lim_{\mathfrak q_1\to+\infty}1/N^{-y}_{\mu_0}(\mathfrak q_{1,(00)},\mathfrak q_{2,(11)})=y^{|\mu_0|+M_0},\\
\end{aligned}\qquad\qquad
\begin{aligned}
&\lim_{\mathfrak q_2\to+\infty}\lim_{\mathfrak q_1\to+\infty}T^{-y}_{\mu_0,\mu_1}(\mathfrak q_{1,(00)},\mathfrak q_{2,(00)})=1,\\
&\lim_{\mathfrak q_2\to+\infty}\lim_{\mathfrak q_1\to+\infty}T^{-y}_{\mu_0,\mu_1}(\mathfrak q_{1,(00)},\mathfrak q_{2,(01)})=1,\\
&\lim_{\mathfrak q_2\to+\infty}\lim_{\mathfrak q_1\to+\infty}T^{-y}_{\mu_0,\mu_1}(\mathfrak q_{1,(00)},\mathfrak q_{2,(10)})=y^{-|\mu_0\setminus\mu_1|},\\
&\lim_{\mathfrak q_2\to+\infty}\lim_{\mathfrak q_1\to+\infty}T^{-y}_{\mu_0,\mu_1}(\mathfrak q_{1,(00)},\mathfrak q_{2,(11)})=y^{-|\mu_0\setminus\mu_1|},\\
\end{aligned}
\end{equation}
and
\begin{equation}
\begin{aligned}
&\lim_{\mathfrak q_2\to+\infty}\lim_{\mathfrak q_1\to+\infty}W^{-y}_{\mu_0,\mu_1,\dots,\mu_N}(\mathfrak q_{1,(00)},\mathfrak q_{2,(00)})=1,\\
&\lim_{\mathfrak q_2\to+\infty}\lim_{\mathfrak q_1\to+\infty}W^{-y}_{\mu_0,\mu_1,\dots,\mu_N}(\mathfrak q_{1,(00)},\mathfrak q_{2,(01)})=1,\\
&\lim_{\mathfrak q_2\to+\infty}\lim_{\mathfrak q_1\to+\infty}W^{-y}_{\mu_0,\mu_1,\dots,\mu_N}(\mathfrak q_{1,(00)},\mathfrak q_{2,(10)})=1,\\,
&\lim_{\mathfrak q_2\to+\infty}\lim_{\mathfrak q_1\to+\infty}W^{-y}_{\mu_0,\mu_1,\dots,\mu_N}(\mathfrak q_{1,(00)},\mathfrak q_{2,(11)})=1,\\
\end{aligned}
\end{equation}
so that, by putting everything together, we have
\begin{equation}\label{sei}
\begin{split}
\sum_{\bold n}\chi_{-y}^{\rm vir}\left(\Hilb^{(\hat{\bold n})}(\mathbb P^1\times\mathbb P^1)\right)\bold q^{\bold n}=\left(\sum_{\bold n}\bold q^{\bold n}\sum_{\{\mu_i\}}y^{|\mu_0|-M_0}\right)\left(\sum_{\bold n}\bold q^{\bold n}\sum_{\{\mu_i\}}y^{|\mu_0|}\right)\left(\sum_{\bold n}\bold q^{\bold n}\sum_{\{\mu_i\}}y^{|\mu_0|-|\mu_0\setminus\mu_1|}\right)\\
\left(\sum_{\bold n}\bold q^{\bold n}\sum_{\{\mu_i\}}y^{|\mu_0|-|\mu_0\setminus\mu_1|+M_0}\right).
\end{split}
\end{equation}


\bibliographystyle{amsplain}
\bibliography{refs.bib}

\vspace{20mm}

\begin{minipage}{.65\textwidth}
\noindent\textsc{S.I.S.S.A. - Scuola Internazionale Superiore di Studi Avanzati -
via Bonomea, 265 - 34136 Trieste ITALY}\vspace{3mm}

\noindent\textsc{I.N.F.N. – Istituto Nazionale di Fisica Nucleare - Sezione di Trieste}\vspace{3mm}

\noindent\textsc{I.G.A.P. – Institute for Geometry and Physics - 
via Beirut, 4 - 34100 Trieste ITALY}
\end{minipage}


\end{document}